\documentclass[11pt]{article}
\usepackage{amssymb, amsmath, amsthm, amsfonts,enumerate,appendix}
\usepackage{mathtools}
\usepackage{amsmath,setspace,scalefnt}
\usepackage[dvipsnames,svgnames,table]{xcolor}
\usepackage{graphicx,tikz,caption,subcaption}
\usetikzlibrary{patterns}
\usetikzlibrary{shapes}
\usepackage{fullpage}
\usepackage{hyperref}
\usepackage{enumitem,verbatim}
\usetikzlibrary{decorations.pathmorphing}
\usepackage{subcaption}
\usepackage[left=2.5cm, right=2.5cm, top=2.5cm, bottom=2.25cm]{geometry}

\tikzset{snake it/.style={decorate, decoration=snake}}
\usetikzlibrary{calc}

\newtheorem{lemma}{Lemma}[section]

\newtheorem{corollary}[lemma]{Corollary}
\newtheorem{theorem}[lemma]{Theorem}

\newtheorem{conjecture}[lemma]{Conjecture}

\theoremstyle{definition}
\newtheorem{defn}[lemma]{Definition}

\newtheorem{claim}[lemma]{Claim}

\theoremstyle{remark}

\global\long\def\eps{\varepsilon}

\global\long\def\N{\mathbb{N}}
\global\long\def\R{\mathbb{R}}

\global\long\def\P{\mathbb{P}}

\global\long\def\cF{\mathcal{F}}

\global\long\def\cK{\mathcal{K}}

\global\long\def\E{\mathbb{E}}

\newcommand{\claimproofstart}[1][Proof]{\begin{proof}[#1]
\renewcommand{\qedsymbol}{$\boxdot$}}
\newcommand\claimproofend{
\end{proof}
\renewcommand{\qedsymbol}{$\square$}}

\def\NatOpt#1{}
\def\RichardOpt#1{}
\def\DanielOpt#1{}
\def\nat#1{}
\def\richard#1{}
\def\daniel#1{}
\let\nat=\NatOpt 
\let\richard=\RichardOpt
\let\daniel=\DanielOpt

\makeatletter
\newcommand{\labelinthm}[1]{%
   \label{temp#1}
   \protected@write \@auxout {}{\string \newlabel{#1}{{\emph{\ref{temp#1}}}{\thepage}{\emph{\ref{temp#1}}}{temp#1}{}} }%
}
\makeatother
\newcounter{propcounter}

\global\long\def\cP{\mathcal{P}}

\title{A proof of the Kim-Vu sandwich conjecture}

\author{Natalie Behague\thanks{Mathematics Institute, University of Warwick, Coventry CV4 7AL, UK. natalie.behague@warwick.ac.uk, daniel.ilkovic@warwick.ac.uk, and richard.montgomery@warwick.ac.uk.}
\and Daniel  Il'kovi\v{c}$^*$\thanks{Institute of Mathematics, Leipzig University, Augustusplatz 10, 04109 Leipzig, Germany.
\newline \vspace{-0.2cm}
\newline NB and RM were supported by the European
Research Council (ERC) under the European Union Horizon 2020 research and innovation programme (grant agreement No.\ 947978).
\newline
 DI was supported by the Alexander von Humboldt Foundation in the framework of the Alexander von Humboldt Professorship of Daniel Král' endowed by the Federal Ministry of Education and Research.\vspace{-1cm}} \and Richard Montgomery$^*$}

\begin{document}

\maketitle

\begin{abstract}
In 2004, Kim and Vu conjectured that, when $d=\omega(\log n)$, the random $d$-regular graph $G_d(n)$ can be sandwiched with high probability between two random binomial graphs $G(n,p)$ with edge probabilities asymptotically equal to $\frac{d}{n}$. That is, there should exist $p_*=(1-o(1))\frac{d}{n}$, $p^*=(1+o(1))\frac{d}{n}$ and a coupling $(G_*,G,G^*)$ such that $G_*\sim G(n,p_*)$, $G\sim G_d(n)$, $G^*\sim G(n,p^*)$, and $\P(G_*\subset G\subset G^*)=1-o(1)$.
Known as the \emph{sandwich conjecture}, such a coupling is desirable as it would allow properties of the random regular graph to be inferred from those of the more easily studied binomial random graph.
The conjecture was recently shown to be true when $d\gg\log^4n$ by Gao, Isaev and McKay. In this paper, we prove the sandwich conjecture in full.
We do so by analysing a natural coupling procedure introduced in earlier work by Gao, Isaev and McKay, which had only previously been done when $d\gg n/\sqrt{\log n}$.
\end{abstract}

\section{Introduction}\label{sec:intro}

Random graphs have been a central object of study in Combinatorics since the foundational work of Erd\H{o}s and R\'enyi~\cite{ER59} in 1959.
The binomial random graph, or Erd\H{o}s-R\'enyi random graph, $G(n,p)$ has vertex set $[n]=\{1,\ldots,n\}$ and each potential edge included independently at random with probability $p$.
Along with the closely related uniformly random graph with $n$ vertices and $\lceil p\binom{n}{2}\rceil$ edges, $G(n,p)$ is the most studied random graph model. The next most studied model is probably that of the random regular graph $G_d(n)$, which is chosen uniformly at random from all $d$-regular graphs with vertex set $[n]$. Throughout this paper, and as is common, we will implicitly assume that $dn$ is even, so that the set of such graphs is non-empty.

The study of random regular graphs began in earnest in the late 1970's, with early work including that by Bender and Canfield~\cite{bender1974asymptotic,bender1978asymptotic}, Bollob\'as~\cite{bollobas1980probabilistic}, and
Wormald~\cite{wormald1981asymptoticconnect,wormald1981asymptoticcycle}. Compared to $G(n,p)$, where the independence of the edges allows the use of a wide variety of techniques, studying $G_d(n)$ is often much more difficult. For example, from developments~\cite{bollobas1984evolution,komlos1983limit} of the breakthrough work of P\'osa~\cite{posa1976hamiltonian} in 1976, it has been long understood when we may expect $G(n,p)$ to be Hamiltonian. On the other hand, it was widely anticipated that, for each $3\leq d\leq n-1$, $G_d(n)$ should be Hamiltonian with high probability (i.e., with probability $1-o(1)$), but proving this took the combined work of many authors~\cite{bollobas1983almost,cooper2002randomconnect,cooper2002randomindependence,krivelevich2001random,robinson1992almost,robinson1994almost} over the course of 20 years (see the surveys~\cite{frieze2019hamilton,wormald1999models} for more details), using a variety of tools in different regimes for $d$, from the configuration model~\cite{bollobas1983almost}, through switching methods~\cite{mckay1981subgraphs}, to estimates on the number of regular graphs~\cite{mckay1990asymptotic,shamir1984large}.

This increased difficulty and technicality gives great appeal to finding links from $G_d(n)$ to $G(n,p)$, so that we may hopefully deduce properties of $G_d(n)$ from those known for $G(n,p)$. In 2004, Kim and Vu~\cite{kim2004sandwiching} formalised this desire in their famous `sandwich conjecture'. They conjectured that, if $d=\omega(\log n)$, then there are $p_*,p^*=(1+o(1))d/n$ and a coupling $(G_*,G,G^*)$ such that $G_*\sim G(n,p_*)$, $G\sim G_d(n)$, $G^*\sim G(n,p^*)$,  and $\P(G_*\subset G\subset G^*)=1-o(1)$. If true, the requirement $d=\omega(\log n)$ cannot be removed; as is well-known, for each $C>0$ there is some $\eps>0$ such that if $d=C\log n$ then, with probability $1-o(n^{-1})$, $G(n,(1+\eps)d/n)$ has minimum degree less than $d$, and hence contains no $d$-regular subgraph.

Kim and Vu~\cite{kim2004sandwiching} proved the lower part of their conjecture when $\log n\ll d\ll n^{1/3}/\log^2n$, and a weakened upper part for the same range of $d$. More specifically, with high probability their coupling $(G_*,G,G^*)$ satisfied $G_*\subset G$ and $\Delta(G\setminus G^*)\leq (1+o(1))\log n$.
In 2017, Dudek, Frieze, Ruci{\'n}ski and {\v{S}}ileikis~\cite{dudek2017embedding} extended this by showing that the lower part holds when $d=o(n)$, and that it can be generalised to hypergraphs. Subsequently, a major breakthrough was made by Gao, Isaev and McKay~\cite{gao2020sandwichingFIRSTBOUND,gao2022sandwichingFIRSTBOUND}, who gave the first coupling $(G_*,G,G^*)$ without a weakened upper part, that is, in which $\P(G_*\subset G\subset G^*)=1-o(1)$. This allowed them to show that the sandwich conjecture is true if $\min\{d,n-d\}\gg n/\sqrt{\log n}$. Klimo{\v{s}}ov{\'a}, Reiher, Ruci{\'n}ski  and {\v{S}}ileikis~\cite{klimovsova2023sandwiching} then extended this to show that it holds for $\min\{d,n-d\}\gg(n\log n)^{3/4}$ (while generalising it to biregular graphs).
Very significant progress was then made again by  Gao, Isaev and McKay~\cite{gao2020kim}, who showed that the conjecture is true provided that  $d\gg \log^4n$.

Here,  we will prove the sandwich conjecture in full, in the following very slightly stronger form.
\begin{theorem}\label{thm:sandwich}
For each $\eps>0$ there is some $C>0$ such that the following holds for each $d\geq C\log n$. There is some $(1-\eps)d/n\leq p_*,p^*\leq (1+\eps)d/n$ and a coupling $(G_*,G, G^*)$ of random graphs such that $G_*\sim G(n,p_*)$, $G \sim G_d(n)$, $G^* \sim G(n,p^*)$, and $\P(G_*\subset G \subset G^*)=1-o(1)$.
\end{theorem}

In their breakthrough initial work, Gao, Isaev and McKay~\cite{gao2020sandwichingFIRSTBOUND,gao2022sandwichingFIRSTBOUND} introduced a beautiful natural coupling process, and analysed it in the regime $\min\{d,n-d\}\gg n/\sqrt{\log n}$.
In order to prove Theorem~\ref{thm:sandwich}, we analyse this process throughout the whole regime $d\geq C\log n$. As highlighted by Gao, Isaev and McKay~\cite[Question 2.4]{gao2022sandwichingFIRSTBOUND},
key to this analysis is showing that in certain random graphs $F$ a uniformly chosen $d$-regular subgraph is likely to include any given edge in $F$ with roughly equal probability, and our success in doing so is of some independent interest (see Theorem~\ref{thm:maintechnicalthm-forprocess}).
It allows us to avoid the complications of running two versions of this process, one after another, as was done in the
later work of Gao, Isaev and McKay~\cite{gao2020kim}. While this two-process coupling ingeniously avoided the most significant challenge in analysing a coupling generated by the single process, it encounters very significant barriers when $d$ is below $\log ^4n$.

We outline our methods in context with previous work in Section~\ref{sec:proofsketch}, but for now make some brief remarks.
As in~\cite{gao2020kim,gao2020sandwichingFIRSTBOUND,gao2022sandwichingFIRSTBOUND}, we work via switching techniques, requiring the counting of certain `switching paths'. The accuracy we use for our bounds has to be carefully chosen, particularly in the most critical regime, and our analysis is completely separate to previous work. Very broadly, our improvements come in two main areas. Firstly, we work much more directly with random graphs than~\cite{gao2020kim}, which used pseudorandom properties from the spectral properties of random graphs. 
Secondly, our bounds on the number of switching paths do not hold for every $d$-regular subgraph of our random graph (as they do in \cite{gao2020kim}), but typically hold for sufficiently many for our proof. This allows us to pivot in a slight but very useful way, considering instead `switching paths' between two random graphs (a random $d$-regular graph and a random graph containing it) rather than between a random graph and any one of its $d$-regular subgraphs.

While, due to~\cite{gao2020kim}, the previously open regime for Theorem~\ref{thm:sandwich} is the upper part of the sandwich for $d=O(\log^4n)$, our methods here naturally work as long as $d\leq n^{\sigma}$ for some fixed $\sigma>0$. With some brief additional work (see Section~\ref{sec:Gndedgedistribution}), we use our methods to give the upper part of the sandwich for the whole range $d\geq C\log n$. This gives a more straightforward proof of the sandwich conjecture in the whole regime as, for example, it avoids the complex analytic techniques used by Gao, Isaev and McKay~\cite[Section 7]{gao2022sandwichingFIRSTBOUND} when $d=\Omega(n)$. For completion, and as our methods give a slightly alternative approach, we also include a full proof of the lower part of the sandwich
(see Section~\ref{sec:lowerpart}).
Altogether, then, we give a self-contained proof of Theorem~\ref{thm:sandwich} which only relies on some quoted results on the concentration of certain random variables (see Section~\ref{sec:concentration}).

It follows immediately from Theorem~\ref{thm:sandwich} that, if $d_2\geq d_1=\omega(\log n)$ and $d_2-d_1=\Omega(d_1)$, and $d_1n$ and $d_2n$ are even, then there is a coupling $(G_1,G_2)$ with $G_1\sim G_{d_1}(n)$, $G_2\sim G_{d_2}(n)$, and $\P(G_1\subset G_2)=1-o(1)$. (This naturally improves the same conclusion when $d\gg \log^4 n$ by Gao, Isaev and McKay~\cite{gao2020kim}.)
That this monotonic coupling property for $G(n,d)$ should be possible for $d_1\leq d_2$ more widely was the subject of some discussion in the community (see~\cite{gao2022sandwichingFIRSTBOUND}), but has appeared only recently in print, in the following form.

\begin{conjecture}[Gao, Isaev and McKay~\cite{gao2022sandwichingFIRSTBOUND}]\label{conj:coupletwoGnd}
Let $d_1,d_2\in [n-1]$ with $d_1\leq d_2$ be such that $d_1n$ and $d_2n$ are even and $(d_1,d_2)\notin\{(1,2),(n-3,n-2)\}$. Then, there is a coupling $(G_1,G_2)$ of random graphs such that $G_1\sim G_{d_1}(n)$, $G_2\sim G_{d_2}(n)$, and, with probability $1-o(1)$, $G_1\subset G_2$.
\end{conjecture}

While the current progress on Conjecture~\ref{conj:coupletwoGnd} can be found outlined in~\cite{hollom2025monotonicity}, the most interesting cases are perhaps where $n$ is even and $d_2=d_1+1$ (so that good bounds on the probabilities involved may imply the conjecture more widely). Here, the conjecture is known as long as $d_1=\omega(1)$  and $d_1\leq n^{1/7}/\log n$, due to work by Gao~\cite{gao2023number} and then Hollom, Lichev, Mond, Portier and Wang~\cite{hollom2025monotonicity}.

A version of the sandwich conjecture has also been proposed for graphs chosen uniformly at random from those with vertex set $[n]$ and degree sequence $\mathbf{d}=(d_1,\ldots,d_n)$, by Gao, Isaev and McKay (see~\cite[Conjecture~1.4]{gao2022sandwichingFIRSTBOUND}), to generalise the sandwich conjecture beyond the sequences $\mathbf{d}=(d,\dots,d)$. They conjectured that the corresponding coupling should be possible for sequences $\mathbf{d}=(d_1,\ldots,d_n)$ with $d_i=(1+o(1))d$ for each $i\in [n]$ under the same condition $d=\omega(\log n)$. The results in~\cite{gao2020kim,gao2020sandwichingFIRSTBOUND,gao2022sandwichingFIRSTBOUND} were shown in this wider generality (sometimes with a stronger bound needed on the distribution of the $d_i$). Our methods are likely to be able to make further progress here, but to avoid further notational and technical cost we will prove only Theorem~\ref{thm:sandwich}.

As well as this potential generalisation, it is also natural to ask for which functions $\eps=\eps(n)$ and $d=d(n)$ the coupling in Theorem~\ref{thm:sandwich} exists. Broadly, it seems we should anticipate that the coupling is likely to be possible if, with high probability, $\Delta(G(n,(1-\eps)d/n))\leq d\leq \delta(G(n,(1+\eps)d/n))$. In this direction, we find the following conjecture appealing, for which we recall that the $n$-vertex \emph{random graph process} $G_0,G_1,\ldots,G_{\binom{n}{2}}$ begins with the graph $G_0$ with vertex set $[n]$ and no edges and, for each $i\geq 1$, $G_i$ is formed from $G_{i-1}$ by the addition of an edge selected uniformly and independently at random from $[n]^{(2)}\setminus E(G_{i-1})$.

\begin{conjecture}\label{conj:hittingtime} Let $G_0,G_1,\ldots,G_{\binom{n}{2}}$ be the $n$-vertex random graph process and let $1\leq d\leq n-1$ be such that $dn$ is even. Then, there is some $G\sim G_d(n)$ for which the following holds with high probability. For each $0\leq i\leq \binom{n}{2}$, if $\Delta(G_i)\leq d$ then $G_i\subset G$, and if $\delta(G_i)\geq d$ then $G\subset G_i$.
\end{conjecture}

If true, proving Conjecture~\ref{conj:hittingtime} seems to fundamentally require additional ideas rather than simply a refinement of the techniques presented here.
In the next section, we will give some brief details of our notation before discussing in detail our proof and its relation to previous methods.
The rest of the paper is then outlined in detail, where the proof of the upper part of Theorem~\ref{thm:sandwich} appears in Sections~\ref{sec:process} to \ref{sec:maintechnicalconclusion} and the proof of the lower part appears in Section~\ref{sec:lowerpart}.

\vspace{-0.2cm}

\subsection*{Acknowledgements}

\vspace{-0.2cm}

The authors would like to thank participants of the workshop `Combinatorics, Probability and Algorithms' at the Mathematisches Forschungsinstitut Oberwolfach in September 2025 for conversations that led to Conjecture~\ref{conj:hittingtime} and inspired a simplification of the proof of Theorem~\ref{thm:sandwich}.

\section{Preliminaries and proof discussion}\label{sec:proofsketch}

In this section we will cover some of our notation, before discussing the proof of Theorem~\ref{thm:sandwich}
in detail and giving a brief outline of the paper. 

\subsection{Notation}
A graph $G$ has vertex set $V(G)$ and edge set $E(G)$, and we let $|G|=|V(G)|$ and $e(G)=|E(G)|$.
For any two vertex sets $U,V$ of a graph $K$, we use $e_K(U,V)=|\{(u,v):u\in U,v\in V,uv\in E(K)\}|$, and the set $N_G(U)$ is the \emph{neighbourhood of $U$ in $G$}, that is, the set of vertices in $V(G)\setminus U$ with at least one neighbouring edge to $U$. Given a graph $G$, $v\in V(G)$ and $e\in V(G)^{(2)}$, $G-v$ is the graph $G[V(G)\setminus \{v\}]$, while $G-e$ and $G+e$ are, respectively, the graphs with vertex set $V(G)$ and edge set $E(G)\setminus \{e\}$ and $E(G)\cup \{e\}$. We extend such notation in the usual manner to, for example, $G-A$ and $G+E$ where $A\subset V(G)$ and $E\subset V(G)^{(2)}$. Given graphs $G$ and $H$, $G\setminus H$ is the graph with vertex set $V(G)\setminus V(H)$ and edge set $E(G)\setminus E(H)$.

All logarithms are natural. We use `hierarchy' notation to manage the constants we use, where $\alpha \ll \beta$ means that there is some non-negative decreasing function $f$ such that what follows will hold for all $\alpha\leq f(\beta)$. For hierarchies with more than 2 variables these functions are chosen from right to left. For any $a, b, c \in \R$, we say $a = b \pm c$ if $b-c \leq a \leq  b+c$. We say an event occurs with high probability if it occurs with probability $1 - o(1)$.

For each $d,n\in \mathbb{N}$, $\cK_d(n)$ is the set of $d$-regular graphs with vertex set $[n]$, and, for each graph $F$ with $V(F)=[n]$, $\cK_d(F)=\{K\in \cK_d(n):K\subset F\}$. For distinct vertices $x,y$, an $x,y$-path is a path with endvertices $x$ and $y$. Though our paths are undirected, we implicitly assume the path starts at $x$ and ends at $y$. We say the path is \emph{$(F,K)$-alternating} for graphs $F$ and $K$ if, starting from $x$, the edges are alternately in $F$ and then $K$. For each $n\in \N$, we use $K_n$ for the complete graph with vertex set $[n]$.

\subsection{Proof discussion}
Let $\eps>0$ be constant and, for convenience in this discussion, take the slightly stronger condition that $d=\omega(\log n)$. To prove Theorem~\ref{thm:sandwich} we will create the lower part of sandwich (i.e., $(G_*,G)$ with the required properties) using an edge addition process and the upper part of the sandwich (i.e., $(G,G^*)$) using an edge removal process. These processes are essentially those used by Gao, Isaev and McKay~\cite{gao2020sandwichingFIRSTBOUND,gao2022sandwichingFIRSTBOUND} for their first result on the sandwich conjecture, and the main difference here is in how we analyse it. The first process is (roughly) the complement of the other, but the analysis is separate so we discuss them one after another in Section~\ref{subsec:coupling} below, starting with the process for the lower part of the sandwich. Both processes have a key result, Lemma~\ref{lem:maintechnicallem-forprocess} and Theorem~\ref{thm:maintechnicalthm-forprocess} respectively, that shows that the critical property we will need to construct the coupling is likely to hold. These are understandable out of context, using only Definitions~\ref{defn:Fminusndm} and~\ref{defn:Fndm}.  We will prove these key results using switching methods, as discussed in Section~\ref{subsec:switching}. The rest of the paper is briefly outlined in Section~\ref{subsec:paperoutline}.


\subsubsection{Coupling processes}\label{subsec:coupling}

\noindent\textbf{The lower part of the sandwich.} Start with $F$ as the graph with vertex set $[n]$ and  no edges. Iteratively, as long as $F$ is not a $d$-regular graph, pick a random edge $e$ from $[n]^{(2)}$ and if it is not in $E(F)$ then add it to $F$ with probability
\begin{equation}\label{eqn:goodpartial}
\frac{|\{K\in \cK_d(n):F+e\subset K\}|}{\max_{f\in [n]^{2}\setminus E(F)}|\{K\in \cK_d(n):F+f\subset K\}|},
\end{equation}
where $\cK_d(n)$ is the set of $d$-regular graphs with vertex set $[n]$.
Observe that here an edge $e$ is only added to $F$ if $F$ is not $d$-regular but there is some $d$-regular graph containing $F+e$; thus the denominator at \eqref{eqn:goodpartial} is always non-zero. Furthermore, as long as the process runs there is some edge $e$ in $[n]^{(2)}\setminus E(F)$ which is added to $F$ with strictly positive probability if it is chosen, so the process stops with probability 1. Let $G$ be the graph produced at the end of the process, which is then a $d$-regular graph with vertex set $[n]$.
The probability at \eqref{eqn:goodpartial} is chosen so that $G\sim G_d(n)$ (as we will later show in Section~\ref{subsec:lowerboundFdist}).

While doing this, we will simultaneously build a random graph $G_*$, starting similarly with the graph with vertex set $[n]$ and no edges. When we consider the random edge $e$, we also consider whether to add the random edge $e$ to $G_*$ or not; when $e\notin E(G_*)$ we do this with the fixed probability $1-\eta$, where $0<\eta \ll \eps$. Then, if $e$ is not in $F$ when it is considered and the probability at \eqref{eqn:goodpartial} is at least $1-\eta$, we can couple the random addition of $e$ so that $e$ is always added to $F$ if it is added to $G_*$. As discussed below, with high probability, for most of this process, the probability at \eqref{eqn:goodpartial} is very close to 1. Thus, if we build such a random graph $G_*$ and stop when it has $\lceil(1-2\eta)dn/2\rceil$ edges, then with high probability we will have $G_*\subset F$.

Now, observe that by symmetry $\P(G_*=H)$ is the same for any graph $H$ with $V(H)=[n]$ and $e(H)=\lceil(1-2\eta)dn/2\rceil$. Therefore, $G_*$ has the distribution of a uniformly random graph with vertex set $[n]$ and $\lceil(1-2\eta)dn/2\rceil$ edges. As $e(G(n,(1-\eps)d/n))<\lceil(1-2\eta)dn/2\rceil$ with high probability, we can use a standard coupling of $G(n,(1-\eps)d/n)$ (so that it is within $G_*$ with high probability), to get the exact coupling that we want.

As observed and proved by Gao, Isaev and McKay~\cite{gao2020kim}, for each $0\leq i\leq dn/2$, conditioned on having $i$ edges, the distribution of $F$ is simple to state. For this, then, for each $0\leq i\leq dn/2$ let $F_i$ be the value of $F$ in this process when it has $i$ edges. We use the following definition.

\begin{defn}\label{defn:Fminusndm} Let $d,n,m\in \N$ with $0\leq d\leq n-1$, $dn$ even, and $0\leq m\leq dn/2$. We say $F\sim F^-(n,d,m)$ if $F$ is a random graph with vertex set $[n]$ which has the same distribution as $H-E$ if $H\sim G_d(n)$ and $E$ is a uniformly random set of $m$ edges from $E(H)$.
\end{defn}

Then, for each $0\leq i\leq dn/2$, $F_i\sim F^-(n,d,(dn/2)-i)$. We include a proof of this (along the lines of \cite{gao2020kim}) in Section~\ref{subsec:lowerboundFdist}. Given this, and the discussion above, to show that the probability at \eqref{eqn:goodpartial} is likely above $1-\eta$ until $F$ has at least $\lfloor(1-\eta)dn/2\rfloor$ edges (which we can then expect is later than when $G_*$ reaches  $\lceil(1-2\eta)dn/2\rceil$ edges), it will suffice to prove the following result.

\begin{lemma}\label{lem:maintechnicallem-forprocess}
Let $1/n\ll 1/C\ll \eta,\eps\ll 1$. Let $d\geq C\log n$ and $\eps dn \leq  m\leq d n/2$. Let $F\sim F^-(n,d,m)$.  Then, with probability $1-o(n^{-2})$, for each $e,f\notin E(F)$, we have
\begin{equation}\label{eq:dregcontainingeorfinlem:lowerbound}
|\{K\in \mathcal{K}_d(n):F+e\subset K\}|\leq \left(1+\eta\right)|\{K\in \mathcal{K}_d(n):F+f\subset K\}|.
\end{equation}
\end{lemma}

We prove Lemma~\ref{lem:maintechnicallem-forprocess} using switching methods (as discussed in Section~\ref{subsec:switching} and carried out in Section~\ref{subsec:LemmaTechlowerboundproof} using the work more generally in Section~\ref{sec:lowerpart}).

\medskip


\noindent\textbf{The upper part of the sandwich.} The following process is the natural complement of the above process producing $F$.
Start with $F$ as the graph with vertex set $[n]$ and edge set $[n]^{(2)}$. Iteratively, as long as $F$ is not a $d$-regular graph, pick a random edge $e$ from $[n]^{(2)}$ and if it is in $E(F)$ then remove it from $F$ with probability
\begin{equation}\label{eqn:goodpartialupper}
\frac{|\{K\in \cK_d(n):K\subset F-e\}|}{\max_{f\in E(F)}|\{K\in \cK_d(n):K\subset F-f\}|}=\frac{|\cK_d(F-e)|}{\max_{f\in E(F)}|\cK_d(F-f)|},
\end{equation}
where $\cK_d(F-e)$ is the set of $d$-regular subgraphs of $F-e$ with vertex set $[n]$.
Observe that here an edge $e$ is only removed from $F$ if there is some $d$-regular graph in $F-e$; thus the denominator at \eqref{eqn:goodpartialupper} is always non-zero. Furthermore, as before, the process stops with probability 1 with a $d$-regular graph, $G$ say. The probability at \eqref{eqn:goodpartialupper} is chosen so that $G\sim G_d(n)$ (as we will later show in Section~\ref{subsec:clm:main:b}).

As before, we will want to argue that the probability at \eqref{eqn:goodpartialupper} is, with high probability, close to 1 throughout much of the process (which will allow us, similarly, to couple it with a binomial random graph, but one likely to contain $F$). However, a bound of the form $1-\eta$, with $\eta>0$ fixed, will not be good enough here as, instead of adding $dn/2$ edges, the process will remove $\binom{n}{2}-dn/2$ edges to get a $d$-regular graph. This is potentially far many more iterations, which would allow our error terms to accumulate ruinously. We will have to use then a bound in place of $1-\eta$ that depends on the number of edges removed so far.

So much to the negative. In our favour, if we show a natural bound similar to Lemma~\ref{lem:maintechnicallem-forprocess} then it gives a better bound when translated into \eqref{eqn:goodpartialupper}. That is, if we can show that, for each $e,f\in E(F)$,
\begin{equation}\label{eq:dregcontainingeorf}
|\{K\in \mathcal{K}_d(F):e\in E(K)\}|\leq (1+\eta)|\{K\in \mathcal{K}_d(F):f\in E(K)\}|,
\end{equation}
we can deduce (as we do in Section~\ref{subsec:forclmmainb}) that
\begin{align}
\frac{|\mathcal{K}_d(F-e)|}{\max_{f\in E(F)}|\mathcal{K}_d(F-f)|}
&=
\frac{|\mathcal{K}_d(F)|-|\{K\in \mathcal{K}_d(F):e\in E(K)\}|}{\max_{f\in E(F)}(|\mathcal{K}_d(F)|-|\{K\in \mathcal{K}_d(F):f\in E(K)\}|)}\geq 1-\frac{\eta dn/2}{e(F)-dn/2}.\label{eq:actuallycloseto1upp}
\end{align}
Thus, we will be able to show from a bound like \eqref{eq:dregcontainingeorf} that the probability at \eqref{eqn:goodpartialupper} will be very close to 1, and only when $F$ has close to $dn/2$ edges (as it nears a $d$-regular graph) will our bound begin to resemble $1-O(\eta)$.

This, however, is not good enough and our `error term' $\eta$ in \eqref{eq:dregcontainingeorf} will still need to change during the process. For some $C_0,\mu$ with $1/n\ll 1/C_0\ll \mu\ll \eps$, we will set
$\eta_m=C_0\max\Big\{\frac{\mu}{\log n},\left(\frac{n\log n}{m}\right)^{1/8}\Big\}$,
for each $m$, and our desired value of $\eta$ in \eqref{eq:dregcontainingeorf} will be (roughly) $1-\eta_{e(F)}$, so that we can show that \eqref{eqn:goodpartialupper} is always at least the right hand side of \eqref{eq:actuallycloseto1upp} with (roughly) $\eta_{e(F)}$ replacing $\eta$.
From \eqref{eq:actuallycloseto1upp} (and as proved in Section~\ref{sec:proveclmsub:5}), as we will use this for each value of $e(F)$ from $\binom{n}{2}$ down to $\left\lfloor\left(1+\frac{\eps}{2}\right)\frac{dn}{2}\right\rfloor$, one key point is that the sum of $\frac{\eta_mdn/2}{m-dn/2}$ over $m=e(F)$ in this range is $O(C_0\mu/\eps)$, so that this will be much smaller than $\eps$.

The changing value of $\eta_{e(F)}$ in the process makes it more difficult to create the corresponding coupling to create $G^*$. Indeed, starting with $G^*$ as the complete graph with vertex set $[n]$, when we consider an edge $e$, if $e\in E(G^*)$, then (throughout most of the process) we will wish to remove it with some probability likely to be lower than that at \eqref{eqn:goodpartialupper}, except this will now change depending on the value of $e(F)$. This, however, will be predictable enough that in creating $G^*$ we can use a bound that does not depend on $F$ (see Section~\ref{subsec:clm:main:c}).

The most significant challenge, then, is to prove that bounds like that at \eqref{eq:dregcontainingeorf} are likely to hold for much of the process (with our new, varying, error term in place of $\eta$). As we will see (and as discussed below), this is much more challenging than the corresponding bound in the lower part of sandwich. Therefore, here, it is critical for us here that, as observed and proved by Gao, Isaev and McKay~\cite{gao2020kim}, conditioned on its number of edges, $F$ has a nice distribution using the following definition.

\begin{defn}\label{defn:Fndm} Let $d,n,m\in \N$ with $0\leq d\leq n-1$, $dn$ even, and $0\leq m\leq \binom{n}{2}-dn/2$. We say $F\sim F(n,d,m)$ if $F$ is a random graph with vertex set $[n]$ which has the same distribution as $H+E$ if $H\sim G_d(n)$ and $E$ is a uniformly random set of $m$ edges from $[n]^{(2)}\setminus E(H)$.
\end{defn}

For each $dn/2\leq i\leq \binom{n}{2}$, let $F_i$ be the value of $F$ in this process when it has $i$ edges.
Then, for each $dn/2\leq i\leq \binom{n}{2}$, $F_i\sim F(n,d,i-(dn/2))$. We include a proof of this (along the lines of \cite{gao2020kim}) in Section~\ref{subsec:clm:main:b}. We arrive now at the main challenge, requiring the key novelties of this paper. From our discussion above (and as proved later in detail in Section~\ref{sec:process}), it will suffice to prove  the following.

\begin{theorem}\label{thm:maintechnicalthm-forprocess}
Let $1/n\ll 1/C\ll \mu \ll 1/C_0\ll \eps\ll 1$. Let $d\geq C\log n$ and $\eps dn\leq  m\leq (n-1-d)n/2$. Let $F\sim F(n,d,m)$.  Then, with probability $1-o(n^{-2})$, for each $e,f\in E(F)$, we have
\begin{equation}\label{eq:dregcontainingeorfinthm}
|\{K\in \mathcal{K}_d(F):e\in E(K)\}|\leq \left(1+C_0\max\left\{\frac{\mu}{\log n},\left(\frac{n\log n}{m}\right)^{1/8}\right\}\right)|\{K\in \mathcal{K}_d(F):f\in E(K)\}|.
\end{equation}
\end{theorem}

Similar to Lemma~\ref{lem:maintechnicallem-forprocess}, we will prove Theorem~\ref{thm:maintechnicalthm-forprocess} using switching methods, but this will be significantly more difficult. 
Why this is more challenging is explained below, but the new ideas we use to do this are the heart of our ability to analyse this process as long as $d=\omega(\log n)$, compared to the regime $d\gg n/\sqrt{\log n}$ in which this was done by Gao, Isaev and McKay~\cite{gao2022sandwichingFIRSTBOUND}. (In \cite{gao2020kim} a modified two-step process is used to neatly avoid some of these difficulties.)


\subsubsection{Switching methods to prove Lemma~\ref{lem:maintechnicallem-forprocess} and Theorem~\ref{thm:maintechnicalthm-forprocess}}\label{subsec:switching}

The switching method was introduced by McKay~\cite{mckay1981subgraphs}, and is a crucial part of our analysis for both the lower and upper part of the sandwich. We will explain these methods first in the context of the lower part of the sandwich as they are much simpler, before discussing the challenges of switching for the upper part and the key new ideas we will introduce. This also allows us to note one key change in the analysis of the lower part of the sandwich which allows a simpler approach when $d$ is large (and is also used for the upper part of the sandwich).


\medskip

\noindent\textbf{Switching methods for the lower part of the sandwich, i.e., for Lemma~\ref{lem:maintechnicallem-forprocess}.}
Let $F$ have $V(F)=[n]$, and let $e,f\in [n]^{(2)}\setminus E(F)$ share no vertices. Suppose, for \eqref{eq:dregcontainingeorfinlem:lowerbound}, we wish to compare the number of $d$-regular subgraphs in $\cK_d(n)$ containing $F+e$ with the number containing $F+f$. 
Let $\mathcal{K}^e=\{K\in \cK_d(n):F+e\subset K,f\notin E(K)\}$ and $\mathcal{K}^f=\{K\in \cK_d(n):F+f\subset K,e\notin E(K)\}$, and note that if we can show that $|\cK^e|\leq (1+\eta)|\cK^f|$ then \eqref{eq:dregcontainingeorfinlem:lowerbound} will hold.

To use the switching method, we define a notion of `switching' and consider when we can transform $K\in \cK^e$ into $K'\in \cK^f$ using a switching. Comparing the number of ways we can do such a transformation on $K$ with the number of ways we could transform $K'$, will allow us to compare $|\cK^e|$ to $|\cK^f|$. This is best seen via an example.

Suppose for simplicity that $d=o(n)$. Take the auxiliary bipartite graph $L$ with vertex classes $\mathcal{K}^e$ and $\mathcal{K}^f$ where there is an edge between $K\in \cK^e$ and $K'\in \cK^f$ if $E(K)\triangle E(K')$ is a cycle with length $10$ in which $e$ and $f$ appear on opposite sides of the cycle (see Figure~\ref{fig:switch}). Note that such a cycle alternates between $K\setminus K'$ and $K'\setminus K$. We now aim to show that, for some $D>0$, for each $K\in \cK^e$ we have $d_L(K)\geq (1-\eta/3)D$ and for each $K'\in \cK^f$ we have $d_L(K')\leq (1+\eta/3)D$, so that, double-counting the edges of $L$ gives
\begin{equation}\label{eq:sketchdoublecounting}
|\cK^e|\cdot (1-\eta/3)D\leq e(L) \leq |\cK^f|\cdot (1+\eta/3)D,
\end{equation}
and, as $\eta\ll 1$, we thus have $|\cK^e|\leq (1+\eta)|\cK^f|$.

Letting $K\in \cK^e$, for each $K'\in \cK^f$ with $KK'\in E(L)$, $E(K)\triangle E(K')$ is the edge set of a length 10 cycle $S$ which alternates between $K\setminus F$ and $K_n\setminus K$ with $e$ and $f$ on opposite sides. Thus, we can bound $d_L(K)$ by counting such cycles. We will have that the vertex degrees of $K\setminus F$ are all $(1\pm \eta/100)\Delta$ for some $\Delta$ (in Lemma~\ref{lem:maintechnicallem-forprocess} this will be $2m/n$). Arbitrarily labelling $V(e)=\{v_1,v_6\}$, we will use the notation in Figure~\ref{fig:switch} to choose such a cycle $S$ vertex by vertex. After labelling the vertices of $f$ as $v_5$ and $v_{10}$ (with 2 choices in total), and then choosing $v_2$ and $v_7$ (with at most $n$ choices for each), we then have at most $(1+ \eta/100)\Delta$ choices for choosing each of $v_3,v_4,v_8,v_9$, as they are the neighbour in $K\setminus F$ of some previously chosen vertex. Therefore, we have $d_L(K)\leq (1+\eta/100)^4\Delta^4n^2\leq (1+\eta/3)\Delta^4n^2$.

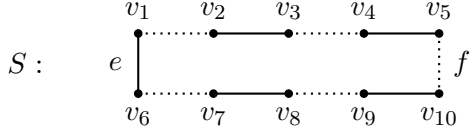
\begin{figure}\centering
\begin{tikzpicture}
\def\hspacer{1cm}
\def\vspacer{0.8cm}
\foreach \n in {1,...,5}
{
\coordinate (A\n) at ($(\n*\hspacer,0)$);
}
\foreach \n in {6,7,8,9,10}
{
\coordinate (A\n) at ($(\n*\hspacer-5*\hspacer,-\vspacer)$);
}

\foreach \n in {1,...,5}
{
\draw ($(A\n)+(0,0.3)$) node {$v_\n$};
\draw [fill] (A\n) circle[radius=0.05cm];
}
\foreach \n in {6,7,8,9,10}
{
\draw ($(A\n)-(0,0.3)$) node {$v_{\n}$};
\draw [fill] (A\n) circle[radius=0.05cm];
}

\draw ($0.5*(A1)+0.5*(A6)-(0.3,0)$) node {$e$};
\draw ($0.5*(A5)+0.5*(A10)+(0.3,0)$) node {$f$};

\foreach \x/\y in {1/2,3/4,5/10,6/7,8/9}
{
\draw[thick,dotted] (A\x) -- (A\y);
}
\foreach \x/\y in {1/6,2/3,4/5,9/10,7/8}
{
\draw[thick] (A\x) -- (A\y);
}

\draw ($0.5*(A1)+0.5*(A6)-(1.5,0)$) node {$S:$};

\end{tikzpicture}
\caption{The cycle $S$ alternating between $K\setminus K'$ (solid lines) and $K'\setminus K$ (dotted lines).}\label{fig:switch}
\end{figure}

This upper bound however, is actually very close to $d_L(K)$: to ensure for a choice of $v_2$,$v_3$,$v_4$,$v_7$,$v_8$, and $v_9$ that there is some $K'\in \cK^f$ with $E(S)=E(K)\triangle E(K')$ we need only make sure that the vertices are all distinct, and that $v_1v_2,v_3v_4,v_6v_7,v_8v_9\in E(K_n\setminus K)$. When $d=o(n)$, these additional conditions  are rather weak as $K_n\setminus K$ is very dense, so we will have at least $(1-\eta/3)\Delta^4n^2$ such cycles
(see Section~\ref{subsec:concentrationforlowerpart}). Each such cycle gives rise to an edge in $L$ from $K$ to a different $K'=K-E(S)\cap E(K)+E(S)\setminus E(K)$. Therefore, including the lower bound, we have $d_L(K)=(1\pm \eta/3)\Delta^4n^2$. As a similar bound can be shown for $f$ (or using the symmetry between $e$ and $f$ in our set-up), we can do the double-counting at \eqref{eq:sketchdoublecounting} using $D=\Delta^4n^2$.

For Lemma~\ref{lem:maintechnicallem-forprocess}, as $F\sim F^-(n,d,m)$, with $m=\omega(n\log n)$ and $d=\omega(\log n)$, a simple concentration inequality (see Section~\ref{sec:concentration}) and a union bound shows that $K\setminus F$ will be approximately regular, allowing this sketch to be carried out when $d=o(n)$. For larger $d$ it is harder to count these switching cycles as accurately, as no longer is $K_n\setminus K$ so dense. In previous work~\cite{gao2020kim}, inequalities like \eqref{eq:dregcontainingeorfinlem:lowerbound} in such a regime were shown to be likely to hold using complex analytic methods. Here, we will use instead two simple ideas to complete the proof when $d=\Omega(n)$. These ideas are very useful for the upper part of the sandwich, so we will introduce them now in the simpler setting of the lower part.

Firstly, we note that we do not need to prove good bounds on the degrees in $L$ that hold for all $K\in \cK^e\cup \cK^f$. That is, the double-counting at \eqref{eq:sketchdoublecounting} will tolerate good bounds which hold only for most of the graphs $K\in \cK^e\cup \cK^f$, while for any outliers we can use weaker bounds that are more straightforward to prove. This will allow us to focus on proving good bounds for graphs $K\in \cK^e\cup \cK^f$ which satisfy some pseudorandom conditions.

Then, in order to show it is likely that most of the graphs $K\in \cK_e\cup \cK_f$ satisfy some pseudorandom conditions, we exploit the distribution of $F\sim F^-(n,d,m)$. Intuitively, if we form $F$ by taking $K\sim G_d(n)$ and deleting $m$ random edges, then if a property is very likely to hold for $K$ it should be likely to hold for most of the $d$-regular graphs containing $F$. (This is proved with the exact parameters we use in Sections~\ref{sec:translating} and~\ref{sec:translatingforlowerbound}.) This means that instead of studying the properties of the $d$-regular graphs containing $F$, we can more simply study a random $d$-regular graph.

\medskip

\noindent\textbf{Switching methods for the upper part of the sandwich, i.e., for Theorem~\ref{thm:maintechnicalthm-forprocess}.}
After using the two ideas above, carrying out a similar switching argument for the upper part (i.e., for Theorem~\ref{thm:maintechnicalthm-forprocess}) will essentially require us to show that if $K\sim G_d(n)$ and $F$ is formed from $K$ by adding $m$ edges uniformly at random from $[n]^{(2)}\setminus E(K)$ (and thus $F\sim F(n,d,m)$) then with high probability we have the following for some suitable $\ell$.
\stepcounter{propcounter}
\begin{enumerate}[label = {{\textbf{\Alph{propcounter}}}}]
\item For each distinct $x,y\in [n]$ the number of $(F\setminus K,K)$-alternating $x,y$-paths of length $2\ell$ is well concentrated.\label{prop:crux}
\end{enumerate}

Indeed, if some version of this holds where the paths avoid an arbitrary small set of vertices, then we can put together two such paths between $e$ and $f$ to count degrees in the natural analogue of the auxiliary graph $L$ discussed above. This is done in detail in Section~\ref{sec:switching}, while here we will focus on \ref{prop:crux}.

Compared to similar tasks in the lower part, this is much more difficult in the region $d=o(n)$ as no longer is one of the $K$ or $F\setminus K$ graphs very close to being complete. Our previous discussion would correspond to paths of length $2\ell=4$, but now we will have to take $\ell=\Omega(\log n/\log \log n)$ when $d=\log^{O(1)}n$, otherwise between some distinct $x,y\in [n]$ no such alternating paths will exist.
This pinpoints why the process for the upper part is so much harder to analyse via switchings than the process for the lower part. In their work, Gao, Isaev, and McKay~\cite{gao2020kim} used two processes, one run after the other, so that (very roughly), when such $(F\setminus K,K)$-alternating paths were found, at least $K$ was sparse when compared to $F\setminus K$ (see \cite{gao2020kim} for more details).

When $F\setminus K$ is sufficiently dense, with $m\geq n^{1+\sigma}$ for any fixed $\sigma>0$, we use the results of Kim and Vu~\cite{kim2000concentration} on the concentration of multivariate polynomials to get sufficiently good bounds on the counts for \ref{prop:crux} (see Section~\ref{sec:concpoly}).
Thus, we can fix some small $\sigma>0$, and assume that $d\leq n^{\sigma}$ and $m\leq n^{1+\sigma}$. Take distinct $x,y\in [n]$ and let $\delta=2m/n$, and note that we may expect $F\setminus K$ to have vertex degrees $(1+o(1))\delta$. Due to the randomness of $F\setminus K$ and $K$, if we randomly walk out from $x$ using edges alternately from $F\setminus K$ and then $K$ we should expect this random walk to rapidly mix so that, as long as (say) $\ell\geq 2\log n/\log (\delta d)$, after $2\ell$ edges we are roughly no more likely to arrive at $y$ than any other vertex. Heuristically, considering the vertex degrees in $F\setminus K$ and in the $d$-regular graph $K$, we should expect there are likely to be $(1+o(1))\delta^\ell d^\ell/n$ $(F\setminus K,K)$-alternating $x,y$-paths of length $2\ell$.

Bounds of this form were proved in certain regimes (roughly when $\delta\gg d\log n$) by Gao, Isaev and McKay~\cite{gao2020kim} by working via the spectral properties of binomial random graphs (and using only that $K$ is $d$-regular), while similar counts on paths within a single graph have been shown by showing certain Markov chains mix rapidly (see, for example, as in \cite{joos2022fractional}, building on work in \cite{cuckler2009hamiltonian}). For Theorem~\ref{thm:maintechnicalthm-forprocess}, we need to work with the two different random graph models, in particular using properties of the more difficult to study random regular graph, and we need to deal with various subtleties due to the accuracy we need and the weaker bound on $d$ that we use. As far as is convenient, then, we rely on pseudorandom conditions of the random graphs involved, in part so that the results we prove may ultimately perhaps prove useful for the study of Conjecture~\ref{conj:hittingtime}.

The number of random walks discussed above can be bounded using bounds on the degrees of $F\setminus K$ and the regularity of $K$. However, when $d=C\log n$ (for our large fixed $C$), our paths will have length $2\ell$ with $\ell=\Omega(\log n/\log\log n)$ and the degrees of $F\setminus K$ are likely to vary by up to $\Omega_{C}(\delta)$. Thus, we cannot rely alone on the likely concentration of the degrees of $F\setminus K$. However, after revealing the edges of $K$, for each $v\in [n]$, we will have that $\sum_{u\in N_K(v)}d_{F\setminus K}(u)$ is better concentrated. For any fixed vertex $v$, even when $d=C\log n$ and $\delta=2\eps d$, this is likely to be concentrated up to an error term of $O(1/\ell)$. However, to get variables that are this concentrated for \emph{every} vertex $v$, we ultimately look a step further and consider $\sum_{v'\in N_{F\setminus K}(v)}\sum_{u\in N_K(v')}d_{F\setminus K}(u)$ (see \ref{eq:nbrhoodsinFK:random} in Lemma~\ref{lem:nbrhoodsum}).

Roughly, then, we count walks alternating between $F\setminus K$ and $K$ using these sums of degrees, and show that most of these walks are actually paths. To do so we rely on pseudorandom properties that later appear as~\ref{eq:fewdoubled:K} and~\ref{eq:fewdoubled:F-K} in Lemma~\ref{lem:lbpathcountF}. Then, considering alternating paths from $x$ and $y$ respectively to sets of different endvertices we use lower bounds on the number of edges in $F\setminus K$ between these different sets of endvertices (using \ref{prop:F-Kconnection} later) to give a lower bound between the number of alternating paths between $x$ and $y$.
A subtlety that is hard to anticipate from our discussion so far is that perhaps the most difficult case for Theorem~\ref{thm:maintechnicalthm-forprocess} is when $d=C\log n$ and $\delta=\log^{O(1)}n$, but $\delta=\omega(\log^{8}n)$. This is because the error term, essentially $\eta$, that we wish to use in the counting of our alternating paths is to be smaller than $C_0\mu /\log n$ (see \eqref{eq:dregcontainingeorfinthm}), but we still have $\ell=\Omega(\log n/\log\log n)$, so that when we build our alternating paths of length $\ell$ from $x$, if the next edge from the current endpoint, say $v$, is to be in $K$ then $v$ may have a fraction of neighbours in the path so far which exceeds $\eta d$, which here is $O(1)$. This creates some additional small but important complication that, for example, is reflected in the error term we use in \ref{prop:lowerboundKFKalternating:newnewnew} in Theorem~\ref{thm:randomproperties:lowerbounds:newnew} later.

To show that the number of $(F\setminus K,K)$-alternating $x,y$-paths of length $2\ell$ is likely bounded above, we exploit our likely lower bound (applied with a slightly different path length) as well as an inductive upper bound on the number of alternating paths from $y$. Revealing the edges in $F\setminus K$ not containing $x$, we deduce that most vertices in $[n]\setminus \{x,y\}$ -- those not in $Z$, say -- have at most $(1+O(\eta))\delta^{\ell-1}d^{\ell}/n$  $(K,F\setminus K)$-alternating paths of length $2\ell-1$ to $y$. Then, revealing the edges in $F\setminus K$ next to $x$, we get that there are at most $(1+O(\eta))\delta^{\ell}d^{\ell}/n$ $(F\setminus K, K)$-alternating $x,y$-paths of length $2\ell$ which do not have their second vertex in $Z\cap N_{F\setminus K}(x)$. As $Z$ was a relatively small set (it will have $O(\eta n)$ vertices), using this and similar properties iteratively will allow us to give an upper bound for all the $x,y$-paths we are interested in apart from those in an increasingly constrained set of paths. After sufficiently many iterations, a simple bound on the remaining possibilities will be sufficient to give us our desired upper bound. This argument is carried out in Section~\ref{sec:upper}.

\subsubsection{Paper outline}\label{subsec:paperoutline}

In Section~\ref{sec:process}, we give the details of the edge deletion process we use to create the upper part of the sandwich, reducing the proof of the upper part of the sandwich to Theorem~\ref{thm:maintechnicalthm-forprocess}. In Section~\ref{sec:switching}, using our main switching argument we reduce proving Theorem~\ref{thm:maintechnicalthm-forprocess} to proving certain pseudorandom conditions concerning switching paths are likely to hold for random graphs $F$ and $K$.
In Sections~\ref{sec:uppernottethered}--\ref{sec:upper} we prove these pseudorandom conditions are likely to hold when $m\leq n^{1+\sigma}$ (for any fixed $\sigma>0$)
In Section~\ref{sec:uppernottethered}, we give likely upper bounds on the number of $(F\setminus K,K)$-alternating paths of a given length which start from a fixed vertex. In Section~\ref{sec:lower}, we give likely lower bounds on the number of $(F\setminus K,K)$-alternating paths between any two fixed vertices (moreover, on such paths which avoid an arbitrary small set of vertices). In Section~\ref{sec:upper}, we give likely upper bounds on the number of $(F\setminus K,K)$-alternating paths between any two fixed vertices. In Section~\ref{sec:noncritical}, we give similar bounds to Sections~\ref{sec:uppernottethered}--\ref{sec:upper} when $m\geq n^{1+\sigma}$ (for any fixed small $\sigma$).
In Section~\ref{sec:maintechnicalconclusion}, we put this together to prove Theorem~\ref{thm:maintechnicalthm-forprocess}.
In combination, this completes the proof for the upper part of the sandwich. In Section~\ref{sec:lowerpart}, we give the lower part of the sandwich and its analysis. 

Note that a reader new to the subject may find it more accessible to start with the proof for the lower part of the sandwich in Section~\ref{sec:lowerpart} before moving on to the proof for the upper part of the sandwich in Sections~\ref{sec:process}--\ref{sec:maintechnicalconclusion}.

\subsection{Concentration inequalities}\label{sec:concentration}
We will use the following well-known result of McDiarmid~\cite[Theorem 2.3(b)]{mcdiarmid1998concentration}.

\begin{theorem}\label{thm:mcd:conc} Let $X_1,\ldots,X_n$ be independent random variables with $0\leq X_i\leq 1$ for each $i\in [n]$. Let $S=\sum_{i=1}^nX_i$ and $\mu=\E S$. Then, for any $\eps>0$,
\[
\P(S\geq (1+\eps)\mu)\leq \exp\left(-\frac{\eps^2\mu}{2(1+\eps/3)}\right).
\]
\end{theorem}
We will use Chernoff's bound, in the following standard form (see, for example, \cite[Corollary~2.2 and Theorem~2.10]{janson2011random}).
\begin{lemma}[Chernoff's bound]\label{chernoff}
Let $X$ be a random variable with mean $\mu$ which is binomially distributed. 
Then for any $0<\gamma<1$, we have that \[\mathbb{P}(|X-\mu|\geq \gamma \mu)\leq 2\exp({-\mu \gamma^2/3}).\]
\end{lemma}

In addition, we will use work of Kim and Vu~\cite{kim2000concentration} on the concentration of multivariate polynomials, in the form of Theorem~\ref{thm:concofpolys}, but quote this in Section~\ref{sec:concpoly} so that the notation required is close to its application.

\section{The upper part of the sandwich}\label{sec:process}

We now give the process producing $(G,G^*)$ for Theorem~\ref{thm:sandwich} (essentially that of Gao, Isaev and McKay~\cite{gao2020sandwichingFIRSTBOUND}) in Section~\ref{subsec:process}, before in the rest of this section showing that the required properties hold (see Claim~\ref{clm:analysisoverview}), subject only to the proof of Theorem~\ref{thm:maintechnicalthm-forprocess}.

\subsection{Edge deletion process}\label{subsec:process}
Let $0<1/C\ll \eps\ll 1$ and $d\geq C\log n$. Let $\mu$ and $C_0$ be such that $1/C\ll \mu \ll 1/C_0 \ll \eps$. Note that we can assume that $(1+\eps)d/n\leq 1$, for otherwise in Theorem~\ref{thm:sandwich} we can take $p^*=1$.
For each $0\leq i< \binom{n}{2}-dn/2$, let
\begin{equation}\label{eqn:etai}
\eta_i=C_0\max\left\{\frac{\mu}{\log n},\left(\frac{n\log n}{\binom{n}{2}-dn/2-i}\right)^{1/8}\right\},
\end{equation}
and for each $i\geq \binom{n}{2}-dn/2$, let $\eta_i=0$.
Let $e_1,e_2,\ldots$ be edges drawn independently and uniformly at random from $[n]^{(2)}$. Let $x_1,x_2,\ldots$ be independent random variables with $x_i\sim U([0,1])$ for each $i\geq 1$. We will choose the random graphs $G$ and $G^*$ using the sequences $e_1,e_2,\ldots$ and $x_1,x_2,\ldots$.

\smallskip 

\noindent\textbf{Produce $G$:} Let $F_0$ be the complete graph with vertex set $[n]$, and let $\ell(0)=0$. For each $1\leq i\leq \binom{n}{2}-dn/2$,
let $\ell(i)$ be the least $\ell(i)>\ell(i-1)$ for which $e_{\ell(i)}\in E(F_{i-1})$ and
\begin{equation}\label{eqn:goodprobupperinprocessforF}
x_{\ell(i)}\leq \frac{|\cK_d(F_{i-1}-e_{\ell(i)})|}{\max_{e\in E(F_{i-1})}|\cK_d(F_{i-1}-e)|},
\end{equation}
and let $F_i=F_{i-1}-e_{\ell(i)}$.
Let $G=F_{\binom{n}{2}-dn/2}$.

\smallskip 

\noindent\textbf{Produce $G^*$:} Let $G^*_0$ be the complete graph with vertex set $[n]$, and let $m(0)=0$. Let $R=\lfloor \eps dn/8\rfloor$. For each $1\leq i\leq \binom{n}{2}$, 
let $m(i)$ be the least $m(i)>m(i-1)$ for which $e_{m(i)}\in E(G^*_{i-1})$ and
\begin{equation}\label{eqn:goodprobupperinprocessforGstar}
x_{m(i)}\leq \begin{cases}
    1-\frac{\eta_{i+R-1}dn/2}{\binom{n}{2}-dn/2 - R - (i-1)} \quad &\text{if } i \le \binom{n}{2} - \frac{dn}{2} - R \\
    1 &\text{otherwise,}
\end{cases}
\end{equation}
and let $G^*_i=G^*_{i-1}-e_{m(i)}$. 
Let $M\sim \mathrm{Bin}\left((1+\eps)\frac{d}{n},\binom{n}{2}\right)$, and set $G^*=G_{\binom{n}{2}-M}$.

\smallskip 

\noindent\textbf{Note:} 
Observe that, for each $1\leq i\leq \binom{n}{2}-dn/2$, an edge $e\in E(F_{i-1})$ is only removed from $F_{i-1}$ with probability bigger than 0 if $|\cK_d(F_{i-1}-e)|> 0$, and thus $F_i$ always contains at least one $d$-regular graph. Furthermore, as $F_{i-1}$ is not $d$-regular, and $\cK_d(F_{i-1})\neq \emptyset$, there is some edge $e\in E(F_{i-1})$ for which $|\cK_d(F_{i-1}-e)|>0$, and thus with probability $1$, we will successfully produce $F_1,\ldots,F_{\binom{n}{2}-dn/2}$, and hence $G$. Similarly, with probability 1, we will successfully produce $G^*$.


\subsection{Analysis overview and deduction of Theorem~\ref{thm:sandwich}}\label{subsec:analysisoverview}
We will show that this process has the following 3 main properties. Recall from Definition~\ref{defn:Fndm} that $F \sim F(n,d,m)$ is a random graph on $n$ vertices with the same distribution as $H + E$ where $H \sim G_d(n)$ and $E$ is a uniformly random set of $m$ edges from $[n]^{(2)}\setminus E(H)$.

\begin{claim}\label{clm:analysisoverview}
\begin{enumerate}[label = {\alph{enumi})}]
\item $G^*\sim G(n,(1+\eps)d/n)$.\label{clm:main:a}
\item For each $0\leq i\leq \binom{n}{2}-dn/2$, $F_i\sim F(n,d,\binom{n}{2}-dn/2-i)$.\label{clm:main:b}
\item With high probability, $G\subset G^*$.\label{clm:main:c}
\end{enumerate}
\end{claim}

As $G=F_{\binom{n}{2}-dn/2}$, Claim~\ref{clm:analysisoverview}~\ref{clm:main:b} implies that $G\sim G_d(n)$. Therefore, Claim~\ref{clm:analysisoverview}~\ref{clm:main:a}, \ref{clm:main:b} and \ref{clm:main:c} directly imply that $(G,G^*)$ satisfies the required properties in Theorem~\ref{thm:sandwich}. Thus, to show the upper part of the sandwich in Theorem~\ref{thm:sandwich} exists, it is sufficient to prove Claim~\ref{clm:analysisoverview}. In the rest of this section we will do this, subject only to the proof of Theorem~\ref{thm:maintechnicalthm-forprocess}.


\subsection{Distribution of $G^*$: proof of Claim~\ref{clm:analysisoverview}~\ref{clm:main:a}}\label{subsec:clm:main:a}
For each $0\leq i\leq \binom{n}{2}$, we have that $e(G^*_i)=\binom{n}{2}-i$, and, by symmetry (as the right hand side of \eqref{eqn:goodprobupperinprocessforGstar} depends only on $i$), the distribution of $G^*_i$ is uniform across all graphs with vertex set $[n]$ and $\binom{n}{2}-i$ edges. Therefore, the distribution of $G^*$ is uniform when conditioned on $e(G^*)=i$, for any $0\leq i\leq \binom{n}{2}$.
As $G^*=G_{\binom{n}{2}-M}$, for each $0\leq i\leq \binom{n}{2}$, we have
$\P(e(G^*)=i)=\P(M=i)$. Since $M\sim \mathrm{Bin}\left((1+\eps)\frac{d}{n},\binom{n}{2}\right)$, 
we therefore have that $G^*\sim G\left(n,(1+\eps)\frac{d}{n}\right)$.


\subsection{Distribution of $F_i$: proof of Claim~\ref{clm:analysisoverview}~\ref{clm:main:b}}\label{subsec:clm:main:b}
We now prove Claim~\ref{clm:analysisoverview}~\ref{clm:main:b}, as shown similarly by Gao, Isaev, and McKay  for the corresponding result in the lower bound of the sandwich (see~\cite[Lemma 3]{gao2020sandwichingFIRSTBOUND}).
We will prove by induction on $i$ that, for each $0\leq i\leq \binom{n}{2}-dn/2$, $F_i\sim F(n,d,\binom{n}{2}-dn/2-i)$.

First note that this is trivially true for $i=0$, and that, if $F$ is any graph with $V(F)=[n]$ and $e(F)=m$, where $dn/2\leq m\leq \binom{n}{2}$, then
\begin{equation}\label{eq:distofFndm}
\P(F(n,d,m-dn/2)=F)=\frac{|\cK_d(F)|}{|\cK_d(n)|}\cdot\binom{\binom{n}{2}-dn/2}{m-dn/2}^{-1}.
\end{equation}
Assume then that $F_{i-1}\sim F(n,d,\binom{n}{2}-dn/2-(i-1))$ and let $F$ be an arbitrary graph with $V(F)=[n]$, $e(F)=\binom{n}{2}-i$, and $\cK_d(F)\neq \emptyset$. Our aim is to show that
\begin{equation}\label{eq:probFiF}
\P(F_i=F)=\frac{|\cK_d(F)|}{|\cK_d(n)|}\cdot\binom{\binom{n}{2}-dn/2}{\binom{n}{2}-dn/2-i}^{-1},
\end{equation}
so that, as $F$ was arbitrary subject to $V(F)=[n]$, $e(F)=\binom{n}{2}-i$, and $\cK_d(F)\neq\emptyset$, we have $F_i\sim F(n,d,\binom{n}{2}-dn/2-i)$.

Let $f\in [n]^{(2)}\setminus E(F)$. Consider $\P(F_i=F|F_{i-1}=F + f)$. Observe that \begin{equation}\label{eq:doublecount}\sum_{e \in E(F + f)}|\cK_d(F + f - e)| = (e(F) + 1 - dn/2)|\cK_d(F+f)|.\end{equation} If $F_{i-1}=F + f$, the probability that, for the first $j'>\ell(i-1)$ with $e_{j'}\in E(F_{i-1})$, the edge $e_{j'}$ is not immediately removed from $F_{i-1}=F+f$ is, using \eqref{eqn:goodprobupperinprocessforF},~\eqref{eq:doublecount} and that $x_{j'}\sim U([0,1])$,
\begin{multline*}
\frac{1}{e(F)+1}\sum_{e'\in E(F+f)}\left(1-\frac{|\cK_d(F+f-e')|}{\max_{e\in E(F+f)}|\cK_d(F+f-e)|}\right) \\
=1-\frac{e(F)+1-dn/2}{e(F)+1}\cdot \frac{|\cK_d(F+f)|}{\max_{e\in E(F+f)}|\cK_d(F+f-e)|},
\end{multline*}
while the probability that $e_{j'}=f$ and $e_{j'}$ is removed from $F_{i-1}=F+f$ is
\[
\frac{1}{e(F)+1}\cdot \frac{|\cK_d(F)|}{\max_{e\in E(F+f)}|\cK_d(F+f-e)|}.
\]
Therefore, letting $\hat{p}=\P(F_i=F|F_{i-1}=F+f)$, we have
\begin{multline*}
\hat{p}=\frac{1}{e(F)+1}\cdot \frac{|\cK_d(F)|}{\max_{e\in E(F+f)}|\cK_d(F+f-e)|}\\
+\left(1-\frac{e(F)+1-dn/2}{e(F)+1}\cdot \frac{|\cK_d(F+f)|}{\max_{e\in E(F+f)}|\cK_d(F+f-e)|}\right)\hat{p},
\end{multline*}
so that
\begin{equation}\label{eq:toapply}
\P(F_i=F|F_{i-1}=F + f)=\hat{p}=\frac{|\cK_d(F)|}{(e(F)+1-dn/2)|\cK_d(F+f)|}.
\end{equation}

Then, as by assumption we have that $F_{i-1}\sim F(n,d,\binom{n}{2}-dn/2-(i-1))$, we see that
\begin{align*}
\P(F_i=&F)=\sum_{f\in [n]^{(2)}\setminus E(F)}\P(F_i=F|F_{i-1}=F + f)\cdot \P(F_{i-1}=F+ f)\\
&\hspace{-0.4cm}\overset{\eqref{eq:toapply},\eqref{eq:distofFndm}}{=}\sum_{f\in [n]^{(2)}\setminus E(F)}\frac{|\cK_d(F)|}{(e(F)+1-dn/2)|\cK_d(F+f)|}\cdot \frac{|\cK_d(F+f)|}{|\cK_d(n)|}\cdot\binom{\binom{n}{2}-dn/2}{\binom{n}{2}-dn/2-(i-1))}^{-1}\\
&=\frac{|\cK_d(F)|}{|\cK_d(n)|}\cdot \frac{\binom{n}{2}-e(F)}{e(F)+1-dn/2}\cdot
\binom{\binom{n}{2}-dn/2}{\binom{n}{2}-dn/2-(i-1))}^{-1}\\
&=\frac{|\cK_d(F)|}{|\cK_d(n)|}\cdot\binom{\binom{n}{2}-dn/2}{\binom{n}{2}-dn/2-i}^{-1},
\end{align*}
and thus \eqref{eq:probFiF} holds, as required.


 \subsection{Containment of $G$ in $G^*$: proof of Claim~\ref{clm:analysisoverview}~\ref{clm:main:c}}\label{subsec:clm:main:c}

To prove Claim~\ref{clm:analysisoverview}~\ref{clm:main:c}, it will be convenient to define two more sequences of random graphs. Let $H_0$ have vertex set $[n]$ and edge set $[n]^{(2)}$. Let $k(0)=0$. For each $1\leq i\leq \binom{n}{2}$ in turn, let $k(i)$ be the least $k>k(i-1)$ such that $e_{k}\in E(H_{i-1})$, and let $H_i=H_{i-1}-e_{k(i)}$. That is, we form the sequence $H_0,H_1,\ldots$ by removing an edge as soon as it appears in the sequence $e_1,e_2,\ldots$ and $k(i)$ is the number of places we need to look along in this sequence to see $i$ different edges.

Let $G^+_0$ have vertex set $[n]$ and edge set $[n]^{(2)}$. Recall that $R=\lfloor \eps dn/8\rfloor$. For each $1\leq i\leq \binom{n}{2}$, if
\begin{equation}\label{eq:goodupperGplus}
x_{k(i)}\leq  \begin{cases} 1-\frac{\eta_{i+R-1}dn/2}{\binom{n}{2}-dn/2-R-i+1} \quad &\text{if $i \le \binom{n}{2}-dn/2-R$} \\
1 &\text{otherwise}
\end{cases}
\end{equation}
then let $G^+_i=G^+_{i-1}-e_{k(i-1)}$ and, otherwise, let $G^+_i=G^+_{i-1}$. Note that the right hand side of \eqref{eq:goodupperGplus} is the same as that in \eqref{eqn:goodprobupperinprocessforGstar}. Thus, when an edge $e$ is considered for deletion in the sequences $G_0^*,G_1^*,\ldots$ and $G_0^+,G_1^+,\ldots$, it will be deleted from one only if it is deleted from the other; the difference is that, if it survives this, $e$ will never be deleted in the sequence $G_0^+,G_1^+,\ldots$. It is perhaps also worth noting here that, instead of the choice in Section~\ref{subsec:process}, we could instead choose $G^*$ using the sequence $G_0^+,G_1^+,\ldots$. However, this would be a more difficult choice there as we do not have the property that $e(G_i^+)=\binom{n}{2}-i$ for each $0\leq i\leq \binom{n}{2}$, and we prefer to keep our processes generating the coupling $(G,G^*)$ as clean as possible.

We will consider the following two properties, both of which we show will occur with high probability, and which together will imply that $G\subset G^*$. For them, we use $N=\binom{n}{2}-dn/2-2R$.

\stepcounter{propcounter}
\begin{enumerate}[label = {{\textbf{\Alph{propcounter}\arabic{enumi}}}}]
\item For each $1\leq i\leq \binom{n}{2}-dn/2-R$, and each $e\in E(F_{i-1})$,\label{prop:forGstarF:1}
\begin{equation}\label{eq:withroom}
\frac{|\cK_d(F_{i-1}-e)|}{\max_{f\in E(F_{i-1})}|\cK_d(F_{i-1}-f)|}\geq 1-\frac{\eta_{i-1} dn/2}{\binom{n}{2}-dn/2-i+1}.
\end{equation}
\item $e(G_N^+)\leq e(H_N)+R$.\label{prop:forGstarF:2}
\end{enumerate}
Roughly speaking (see also Claim~\ref{clm:subanalysisoverview} below), comparing the right hand side of \eqref{eq:withroom} with that of \eqref{eqn:goodprobupperinprocessforF} and \eqref{eqn:goodprobupperinprocessforGstar} we see that, as long as \ref{prop:forGstarF:1} holds, if an edge $e_j$ in the sequence $e_1,e_2,\ldots$ is deleted in the sequence $G_0^*,G_1^*,\ldots$ then it will also be deleted from $F_0,F_1,\ldots$ as long as the graph $F_{i'}$ from which it is to be deleted does not have more than $R$ more edges than the graph $G^+_{i''}$ from which $e_j$ was deleted. That the difference in the number of edges here is not too large will follow as long as \ref{prop:forGstarF:2} holds.

To prove Claim~\ref{clm:analysisoverview}~\ref{clm:main:c}, we will prove the following claim (using Theorem~\ref{thm:maintechnicalthm-forprocess} for Claim~\ref{clm:subanalysisoverview}~\ref{clm:subanalysisoverview:4}).

\begin{claim}\label{clm:subanalysisoverview}
\begin{enumerate}[label = \roman{enumi})]
\item For each $1\leq i\leq\binom{n}{2}-dn/2$, $k(i)\leq \ell(i),m(i)$.\label{clm:subanalysisoverview:1}
\item If \ref{prop:forGstarF:2} holds, then $m(i)\leq k(i+R)$ for all $0\leq i\leq N-R$.\label{clm:subanalysisoverview:2}
\item For each $1 \le i \le N$, if \ref{prop:forGstarF:1} holds and $m(j)\leq \ell(j+R)$ for all $1\leq j\leq i$, then $F_{i+R}\subset G^*_i$.\label{clm:subanalysisoverview:3}
\item \ref{prop:forGstarF:1} holds with high probability.\label{clm:subanalysisoverview:4}
\item \ref{prop:forGstarF:2} holds with high probability.\label{clm:subanalysisoverview:5}
\end{enumerate}
\end{claim}

We will now show that Claim~\ref{clm:analysisoverview} follows from Claim~\ref{clm:subanalysisoverview}. By Claim~\ref{clm:subanalysisoverview}~\ref{clm:subanalysisoverview:4} and \ref{clm:subanalysisoverview:5}, \ref{prop:forGstarF:1} and \ref{prop:forGstarF:2} hold with high probability. When this happens, from  Claim~\ref{clm:subanalysisoverview}~\ref{clm:subanalysisoverview:1} and \ref{clm:subanalysisoverview:2} we have that, for all $1\leq j\leq N-R$, $m(j)\leq k(j+R)\leq \ell(j+R)$. Thus, by
Claim~\ref{clm:subanalysisoverview}~\ref{clm:subanalysisoverview:3}, $G\subset F_{N}\subset G^*_{N-R}$.
As $G^*=G^*_{\binom{n}{2}-M}$, when $M\geq \binom{n}{2}-N+R = dn/2 + 3R$, we have that $G^*_{N-R}\subset G^*$.
As $M\sim \mathrm{Bin}\left((1+\eps)\frac{d}{n},\binom{n}{2}\right)$ and $R=\lfloor \eps dn/8\rfloor$, by a Chernoff bound (Lemma~\ref{chernoff}), with high probability we have $M \ge (1+\eps)dn/2-R \ge dn/2 + 3R$. Therefore $G \subset G^*$  with high probability.
Thus, it is left only to prove Claim~\ref{clm:subanalysisoverview}~\ref{clm:subanalysisoverview:1}--\ref{clm:subanalysisoverview:5} which
we do in Sections~\ref{sec:proveclmsub:first}--\ref{sec:proveclmsub:5} respectively (using Theorem~\ref{thm:maintechnicalthm-forprocess}).


\subsubsection{Proof of Claim~\ref{clm:subanalysisoverview}~\ref{clm:subanalysisoverview:1}}\label{sec:proveclmsub:first}
Let $1\leq i\leq \binom{n}{2}$. Observe that it follows from the definition of the processes that there are at most $i-1$ unique edges in $e_1,\ldots,e_{k(i)-1}$, while the edges $e_{\ell(1)},\ldots,e_{\ell(i)}$ are all unique and the edges $e_{m(1)},\ldots,e_{m(i)}$ are all unique. Thus, $\ell(i),m(i)\geq k(i)$.


\subsubsection{Proof of Claim~\ref{clm:subanalysisoverview}~\ref{clm:subanalysisoverview:2}}
Suppose that  \ref{prop:forGstarF:2} holds. Let $0\leq i\leq N-R$ and
suppose, for contradiction, that $k(i+R)<m(i)$. Then, there is a set $I\subset [i+R]$ such that $|I|=R+1$ and, for each $j\in I$, $k(j)\notin \{m(i'):i'\in [i]\}$. Note that in particular $|I|\geq 1$. Take $j\in I$. As $k(j)\leq k(i+R)<m(i)$, there is some $i'\in [i]$ such that
$m(i'-1)<  k(j)<m(i')$. From Claim~\ref{clm:subanalysisoverview}~\ref{clm:subanalysisoverview:1}, we have $k(i-1)\leq m(i-1)$, so therefore $j\geq i'$.

As $e_{k(j)}$ is the first appearance of that edge in the sequence $e_1,e_2,\dots$, we have $e_{k(j)}\in E(G^*_{i'-1})$. Thus, as $m(i'-1)<k(j)<m(i')$, we have that \eqref{eqn:goodprobupperinprocessforGstar} does not hold with $(k(j),i')$ replacing $(m(i),i)$, that is,
\begin{equation*}
x_{k(j)}>  1-\frac{\eta_{i'+R-1}dn/2}{\binom{n}{2}-dn/2 - R -i'+1}\geq   1-\frac{\eta_{j+R-1}dn/2}{\binom{n}{2}-dn/2-R-j+1},
\end{equation*}
where we have used that $j\geq i'$. Therefore, from the condition at \eqref{eq:goodupperGplus}, we have $e_{k(j)}\in E(G^+_{j})$, and hence $e_{k(j)}\in E(G^+_{j'})$ for each $j'\geq j$. In particular, as $j\leq i+R\leq N$, $e_{k(j)}\in E(G^+_N)$.
On the other hand, $e_{k(j)}\notin E(H_j)$, and thus $e_{k(j)}\notin E(H_k)$ for any $k\geq j$. In particular, again, $e_{k(j)}\notin E(H_N)$.

Therefore, $\{e_{k(j)}:j\in I\}\subset E(G^+_N)\setminus E(H_N)$, and thus $|E(G^+_N)\setminus E(H_N)|\geq R+1$. This contradicts \ref{prop:forGstarF:2}, so we must have had $k(i+R)\geq m(i)$.


\subsubsection{Proof of Claim~\ref{clm:subanalysisoverview}~\ref{clm:subanalysisoverview:3}}
Fix $1 \le i \le N$. 
Suppose that \ref{prop:forGstarF:1} holds and $m(j)\leq \ell(j+R)$ for all $1\leq j\leq i$. Let $e\in[n]^{(2)}\setminus E(G_i^*)$. Then, as $e\notin E(G_i^*)$, there is some $j\in [i]$ with $e_{m(j)}=e$ for which \eqref{eqn:goodprobupperinprocessforGstar} holds with $j$ replacing $i$, so
\begin{equation*}
x_{m(j)}\leq  1-\frac{\eta_{j+R-1}dn/2}{\binom{n}{2}-dn/2-R-j+1}.
\end{equation*}
As $m(j)\leq \ell(j+R)$, there is some $j'\leq j+R$ such that $\ell(j'-1)<m(j)\leq \ell(j')$. Note that
\begin{align*}
x_{m(j)}&\leq 1-\frac{\eta_{j+R-1}dn/2}{\binom{n}{2}-dn/2-j-R+1}
\leq 1-\frac{\eta_{j'-1}dn/2}{\binom{n}{2}-dn/2-j'+1}\overset{\ref{prop:forGstarF:1}}{\leq} \frac{|\cK_d(F_{j'-1}-e_{m(j)})|}{\max_{e\in E(F_{j'-1})}|\cK_d(F_{j'-1}-e)|},
\end{align*}
where we have used that $j'\leq j+R\leq \binom{n}{2} - dn/2-R$.
Then, either $e_{m(j)}\notin E(F_{j'-1})$, or, as $\ell(j'-1)<m(j)\leq \ell(j')$, condition \eqref{eqn:goodprobupperinprocessforF} in the process above implies that $\ell(j')=m(j)$.
 In either case, as $j'\leq j+R\leq i+R$, we thus have $e\notin E(F_{i+R})$. As $e\in[n]^{(2)}\setminus E(G_i^*)$ was arbitrarily chosen, we thus have  $F_{i+R}\subset G_i^*$, as required.


\subsubsection{\ref{prop:forGstarF:1} is likely to hold: proof of Claim~\ref{clm:subanalysisoverview}~\ref{clm:subanalysisoverview:4}}\label{subsec:forclmmainb}
For each $1\leq i\leq \binom{n}{2}-dn/2$, we have $F_{i-1}\sim F(n,d,\binom{n}{2}-dn/2-(i-1))$ by  Claim~\ref{clm:analysisoverview}~\ref{clm:main:b}.
Therefore, by Theorem~\ref{thm:maintechnicalthm-forprocess} and a union bound, with high probability we have, for each $1\leq i\leq \binom{n}{2}-dn/2-R$, and each $e,f\in E(F_{i-1})$,
\begin{equation}\label{eq:dregcontainingeorfinprocessinproof}
|\{K\in \mathcal{K}_d(F_{i-1}):e\in E(K)\}|\leq \left(1+{\eta_{i-1}}\right)|\{K\in \mathcal{K}_d(F_{i-1}):f\in E(K)\}|.
\end{equation}
Note also that, by double counting,
\begin{align}\label{eq:doublecount}
\frac{dn}{2}\cdot 
|\cK_d(F_{i-1})| &= \sum_{f \in E(F_{i-1})} |\{K\in \mathcal{K}_d(F_{i-1}):f\in E(K)\}| \nonumber\\
&\ge e(F_{i-1}) \cdot \min_{f \in E(F_{i-1})}|\{K\in \mathcal{K}_d(F_{i-1}):f\in E(K)\}|.
\end{align}
Therefore, for each $e\in E(F_{i-1})$, we have
\begin{align*}
\frac{|\mathcal{K}_d(F_{i-1}-e)|}{\max_{f\in E(F_{i-1})}|\mathcal{K}_d(F_{i-1}-f)|}
&=
\frac{|\mathcal{K}_d(F_{i-1})|-|\{K\in \mathcal{K}_d(F_{i-1}):e\in E(K)\}|}{\max_{f\in E(F_{i-1})}(|\mathcal{K}_d(F_{i-1})|-|\{K\in \mathcal{K}_d(F_{i-1}):f\in E(K)\}|)} \\
&\hspace{-1.9cm}= 1-\frac{|\{K\in \mathcal{K}_d(F_{i-1}):e\in E(K)\}|-\min_{f\in E(F_{i-1})}|\{K\in \mathcal{K}_d(F_{i-1}):f\in E(K)\}|}
{|\mathcal{K}_d(F_{i-1})|-\min_{f\in E(F_{i-1})}|\{K\in \mathcal{K}_d(F_{i-1}):f\in E(K)\}|}\\
&\hspace{-1.9cm}\overset{\eqref{eq:dregcontainingeorfinprocessinproof}}{\geq} 1-\frac{\eta_{i-1} \min_{f\in E(F_{i-1})}|\{K\in \mathcal{K}_d(F_{i-1}):f\in E(K)\}|}
            {|\mathcal{K}_d(F_{i-1})|-\min_{f\in E(F_{i-1})}|\{K\in \mathcal{K}_d(F_{i-1}):f\in E(K)\}|}\\
&\hspace{-1.9cm}\overset{\eqref{eq:doublecount}}{\geq} 1-\frac{\eta_{i-1} \frac{dn}{2e(F_{i-1})}|\mathcal{K}_d(F_{i-1})|}
            {\left(1-\frac{dn}{2e(F_{i-1})}\right)|\mathcal{K}_d(F_{i-1})|}\\
&\hspace{-1.9cm}= 1-\frac{\eta_{i-1}dn/2}
                         {e(F_{i-1}) - dn/2}= 1-\frac{\eta_{i-1}dn/2}
                         {\binom{n}{2} - dn/2-i+1}.
\end{align*}
Thus we have that \ref{prop:forGstarF:1} holds with high probability, as required.


\subsubsection{\ref{prop:forGstarF:2} is likely to hold: Proof of Claim~\ref{clm:subanalysisoverview}~\ref{clm:subanalysisoverview:5}}\label{sec:proveclmsub:5}

For each $1\leq i\leq \binom{n}{2}$, let $y_i=x_{k(i)}$, and observe that $y_1,\ldots,y_{\binom{n}{2}}$ are independent random variables whose distribution is uniform on $[0,1]$. Recalling that $R=\lfloor \eps dn/8\rfloor$, for each $1\leq i\leq \binom{n}{2} - dn/2 - R$, let $Y_i=0$ if
\begin{equation}\label{eq:new:goodupperGplus}
y_{i}\leq 1-\frac{\eta_{i+R-1}dn/2}{\binom{n}{2}-dn/2-R-i+1} 
\end{equation}
and $Y_i=1$ otherwise. As the condition in \eqref{eq:new:goodupperGplus} depends only on $i$, we also have that $Y_1,\ldots,Y_{\binom{n}{2}}$ are independent random variables. 
Recalling that $N=\binom{n}{2}-dn/2-2R$, let $S=\sum_{i=1}^NY_i$.
Observe from our processes producing $H_0,H_1,\dots$ and $G^+_0,G^+_1,\ldots$ that $S=e(G_N^+)-e(H_N)$. 
We have, then,
\begin{align*}
\E S&=\sum_{i=1}^{N}\frac{\eta_{i+R-1} dn/2}{\binom{n}{2}-dn/2-R-i+1}\\
&\overset{\eqref{eqn:etai}}{=}\sum_{i=1}^{N}\frac{C_0dn/2}{\binom{n}{2}-dn/2-R-i+1}
\max\left\{\frac{\mu}{\log n},\left(\frac{n\log n}{\binom{n}{2}-dn/2-R-i+1}\right)^{1/8}\right\}\\
&=\frac{C_0dn}{2}\cdot \sum_{j=R+1}^{\binom{n}{2}-dn/2-R}\frac{1}{j}\max\left\{\frac{\mu}{\log n},\left(\frac{n\log n}{j}\right)^{1/8}\right\}\\
&\leq \frac{C_0dn}{2}\cdot \sum_{j=R+1}^{\binom{n}{2}}\frac{1}{j}\left(\frac{\mu}{\log n}+\left(\frac{n\log n}{j}\right)^{1/8}\right)\\
&\leq \frac{C_0dn}{2}\cdot\frac{2\mu\log(n^2)}{\log n}+\frac{C_0dn}{2}\int_{j=R}^{\binom{n}{2}}\frac{(n\log n)^{1/8}}{j^{9/8}}dj\\
&\leq 2C_0\mu dn+4C_0dn\left(\frac{n\log n}{R}\right)^{1/8} \le 2C_0\mu dn+4C_0dn\left(\frac{10n\log n}{\eps dn}\right)^{1/8} \leq \frac{\lfloor \eps dn/8\rfloor}{2}= \frac{R}{2},
\end{align*}
where we have used that $d\geq C\log n$ and $1/C\ll \mu\ll 1/C_0\ll \eps$.
Therefore, by Theorem~\ref{thm:mcd:conc} applied with $\eps=1$, we have $\P(S\geq R)\leq \exp(-R/8)=o(1)$. Thus, with high probability, $e(G_N^+)-e(H_N)=S\leq R$, and hence \ref{prop:forGstarF:2} holds.

\section{Main switching argument}\label{sec:switching}

In this section we will show that, if $F$ is roughly regular, and for all but very few graphs $K\in \cK_d(F)$ we have good bounds on the number of certain $(F\setminus K,K)$-alternating $x,y$-paths for each fixed $x$ and $y$, we can use a switching argument to prove that the edges in $F$ appear in roughly the same number of $d$-regular subgraphs of $F$. In certain situations, we want these switching $x,y$-paths to avoid a set of up to $40\log n$ vertices $Z$, and here\footnote{For example, using the notation in Theorem~\ref{thm:maintechnicalthm-forprocess}, when $\delta=2m/n=\omega(\log^8n)$ and $d=C\log n$, we have $\eta=C_0\mu/\log n$ and $\eta d= C_0\mu \cdot C=O(1)$.} $|Z|$ can be much larger than $\eta d$ as we only use the bound $\eta d\geq C_0$ for some large constant $C_0$. Thus, if $|(Z\cup \{x\})\cap N_K(y)|$ is large this can affect significantly the number of $x,y$-paths leaving $y$ with an edge in $K$, which explains the error term in 
\ref{prop:lowerboundKFKalternating} below.

\begin{theorem}\label{thm:switching:propstoapply}
Let $d\in [n]$, $\delta>0$ and $0<  1/C_0, \eta\ll 1$  satisfy $\eta d\geq C_0$ and $\eta^2\delta\geq C_0\log n$.
Let $4\leq \ell\leq 10\log n$.
Let $F$ be a graph with vertex set $[n]$ and $\cK_d(F)\neq\emptyset$ for which the following holds.
\stepcounter{propcounter}
\begin{enumerate}[label = {\emph{\textbf{\Alph{propcounter}\arabic{enumi}}}}]
\item For each $v\in [n]$, $d_{F}(v)=d+(1\pm \eta)\delta$.\labelinthm{prop:degreebound}
\end{enumerate}
Suppose that, for all but at most $n^{-6}|\mathcal{K}_d(F)|$ graphs $K\in \mathcal{K}_d(F)$ the following hold.
\begin{enumerate}[label = {\emph{\textbf{\Alph{propcounter}\arabic{enumi}}}}]\addtocounter{enumi}{1}
\item \labelinthm{prop:lowerboundKFKalternating}
For each distinct $x,y\in [n]$ and each $Z\subset [n]$ with $|Z|\leq 4\ell$, the number of $(F\setminus K,K)$-alternating $x,y$-paths of length $2\ell$ which do not have any internal vertices in $Z$ is at least $\left(1-\eta - \frac{|(Z\cup \{x\})\cap N_K(y)|}{d}\right)d^{\ell}\delta^{\ell}/n$.
\item For each distinct $x,y\in [n]$, the number of $(F\setminus K,K)$-alternating $x,y$-paths of length $2\ell$ is at most $\left(1+\eta\right)d^{\ell}\delta^{\ell}/n$.\labelinthm{prop:upperboundFKKalternating}
\end{enumerate}
Then, for each $e,f\in E(F)$, $|\{K\in \mathcal{K}_d(F):e\in E(K)\}|\leq (1+40\eta)|\{K\in \mathcal{K}_d(F):f\in E(K)\}|$.
\end{theorem}


\begin{proof}
Let $\mathcal{K}^{\mathrm{bad}}$ be the set of $K\in \mathcal{K}_d(F)$ for which either \ref{prop:lowerboundKFKalternating} or \ref{prop:upperboundFKKalternating} does not hold, so that, by our assumptions, $|\mathcal{K}^\mathrm{bad}|\leq n^{-6}|\mathcal{K}_d(F)|$.
We will start by showing that each edge $e\in E(F)$ is contained in at least a small portion of the graphs in $\cK_d(F)$, as follows.

\begin{claim}\label{clm:edgeswellspread2} For each $e\in E(F)$, $|\{K\in \cK_d(F):e\in E(K)\}|\geq d|\cK_d(F)|/n^2$.
\end{claim}
\claimproofstart[Proof of Claim~\ref{clm:edgeswellspread2}] Let $e\in E(F)$. Suppose, for contradiction, that  $|\{K\in \cK_d(F):e\in E(K)\}|< |\cK_d(F)|/n^2$.  Let $\cK_e^+=\{K\in \cK_d(F):e\in E(K)\}$ and $\cK_e^-=\{K\in \cK_d(F):e\notin E(K)\}$. Let $L_e$ be the auxiliary bipartite graph with vertex classes $\cK_e^+$ and $\cK_e^-$ where there is an edge between $K\in \cK_e^+$ and $K'\in \cK_e^-$ if $E(K)\triangle E(K')$ is the edge set of a cycle of length $2\ell+2$.

First, note that, for each $K\in \cK_e^+$ and $K'\in \cK_e^-$ with $KK'\in E(L_e)$ we have that $E(K)\triangle E(K')-e$ is an $(F\setminus K,K)$-alternating path of length $2\ell+1$ between the vertices in $e$. Thus, as $\Delta(F\setminus K)\leq (1+\eta)\delta$ from \ref{prop:degreebound}, for each $K\in \mathcal{K}_e^+$ we have $d_{L_e}(K)\leq (1+\eta)^\ell\delta^{\ell}d^{\ell}$.

Now, as $|\cK_e^+| = |\{K\in \cK_d(F):e\in E(K)\}|< |\cK_d(F)|/n^2$, we have $|\cK_e^-|\ge \left(1-\frac{1}{n^2}\right)|\cK_d(F)|$. Thus for at least $|\cK_d(F)|/2$ graphs $K'\in \cK_e^-$ we have $K'\notin \cK^{\mathrm{bad}}$. Let $K'\in \cK_e^-$ be such a $K'$ and let $u,v$ be such that $e=uv$, so that $uv\notin E(K')$. Choose $u'\in N_{K'}(u)$, with $d$ possibilities. Then, by \ref{prop:lowerboundKFKalternating} with $Z=\{u\}$ there are at least $\delta^{\ell}d^{\ell}/2n$
$(F\setminus K',K')$-alternating $u',v$-paths $P$ with length $2\ell$ which do not contain $u$, where we have used that $\eta\ll 1$ and $\eta d\geq C_0$. For each such path $P$, we have
$K'-uu'+uv-(E(P)\cap E(K'))+(E(P)\setminus E(K'))\in \cK_e^+$.
Therefore, $d_{L_e}(K')\geq \delta^{\ell}d^{\ell+1}/2n$.

As this holds for at least $|\cK_d(F)|/2$ graphs $K'\in \cK_e^-$, we have, double-counting the edges of $L$, that $|\cK_e^+|\cdot (1+\eta)^{\ell}\delta^{\ell}d^{\ell}\geq e(L)\geq \delta^{\ell}d^{\ell+1}\cdot|\cK_d(F)|/4n$. Thus, $|\cK_e^+|\geq d|\cK_d(F)|/\left((1+\eta)^{\ell}\cdot 4n\right) \geq  d|\cK_d(F)|/n^2$,
as required, where we have used that $\ell\leq 10\log n$ and $\eta\ll 1$.
\claimproofend

Now, let $e,f\in E(F)$ share no vertices, and suppose, for contradiction, that
\begin{equation}\label{eqn:forcontradiction:new}
\{K\in \mathcal{K}_d(F):e\in E(K)\}|> (1+18\eta)|\{K\in \mathcal{K}_d(F):f\in E(K)\}|.
\end{equation}
Label vertices so that $e=u_1u_2$ and $f=v_1v_2$.
Let $\mathcal{K}^{e}$ be the set of $K\in\cK_d(F)$ which contain $e$ but not $f$, and let $\mathcal{K}^f$ be the set of $K\in \cK_d(F)$ which contain $f$ but not $e$. From~\eqref{eqn:forcontradiction:new} and Claim~\ref{clm:edgeswellspread2}, we have $|\cK^e|>18\eta |\{K\in \mathcal{K}_d(F):f\in E(K)\}|>|\cK_d(F)|/n^2$, where have used that $\eta d\geq C_0$.

Let $L_{e,f}$ be the bipartite auxiliary graph with vertex classes $\mathcal{K}^{e}$ and $\mathcal{K}^{f}$
where there is an edge between $K\in \mathcal{K}^e$ and $K'\in \mathcal{K}^f$ if $(E(K)\triangle E(K'))\setminus \{e,f\}$ is the edge set of two vertex-disjoint paths $P_1$ and $P_2$ such that,
for each $i\in [2]$, $P_i$ is an $(K'\setminus K,K\setminus K')$-alternating $u_i,v_i$-path of length $2\ell+2$. We will show the following claim.

\begin{claim}\label{clm:mindegreeLef2}
For each $K\in \mathcal{K}^e\setminus \mathcal{K}^{\textrm{bad}}$,
$d_{L_{e,f}}(K)\geq (1-10\eta)d^{2\ell+2}\delta^{2\ell+2}/n^2$.
\end{claim}
\claimproofstart[Proof of Claim~\ref{clm:mindegreeLef2}] Let $K\in \mathcal{K}^e\setminus \mathcal{K}^{\textrm{bad}}$.
Fix $v_1' \in N_K(v_1) \setminus \{u_1,u_2,v_2\}$ and $v_2' \in N_K(v_2) \setminus \{u_1,u_2,v_1,v_1'\}$, for which there are at least $(d-3)(d-4) \geq (1- \eta)d^2$ choices, where we have used that $\eta d\geq C_0$. Let $v_1''\in N_{F\setminus K}(v_1')\setminus \{u_1,u_2,v_1,v_2,v_2'\}$, with at least $(1-2\eta)\delta$ choices by \ref{prop:degreebound} and $\eta^2\delta \ge C_0\log{n}$.
Let $P_1$ be an $(F\setminus K,K)$-alternating $u_1,v_1''$-path of length $2\ell$ which does not have any internal vertices in $\{u_2,v_1,v_1',v_2,v_2'\}$, noting that from \ref{prop:lowerboundKFKalternating} there are at least $\left(1-2\eta\right)d^{\ell}\delta^{\ell}/n$ choices for $P_1$, again using that $\eta d\geq C_0$.

Let $Z=V(P_1)\cup \{u_2,v_1,v_1',v_2,v_2'\}$. Let $v_2''\in N_{F\setminus K}(v_2')\setminus Z$ be such that $|N_K(v_2'')\cap Z|\leq \eta d/2$. Note that the number of edges in $K$ adjacent to $Z$ is at least $\eta d /2 \cdot |\{v: |N_K(v) \cap Z | > \eta d /2\}|$ and at most $d|Z|$.
Therefore the number of choices for $v_2''$ is at least
\[
d_{F\setminus K}(v_2')-|Z|-|\{v: |N_K(v) \cap Z | > \eta d /2\}|\ge d_{F\setminus K}(v_2')-|Z|-\frac{d|Z|}{\eta d/2}\geq (1-2\eta)\delta,
\]
where we have used that $d_{F\setminus K}(v_2')\geq (1-\eta)\delta$ by \ref{prop:degreebound}, $\eta^2\delta\geq C_0\log n$ and $\ell\leq 10\log n$.
Then, let $P_2$ be an $(F\setminus K,K)$-alternating $u_2,v_2''$-path of length $2\ell$ which does not have any internal vertices in $Z$, so that, by the choice of $v_2''$ and by \ref{prop:lowerboundKFKalternating}, we have at least $(1-2\eta)d^{\ell}\delta^{\ell}/n$ choices for $P_2$.

Let $K'=K-E(P_1\cup P_2)\cap E(K)-\{u_1u_2,v_1v_1',v_2v_2'\}+E(P_1\cup P_2)\setminus E(K)+\{v_1v_2,v_1'v_1'',v_2'v_2''\}$. Note that $K'\in \cK^f$, and that we have made the unique choices for $v_1',v_2',v_1'',P_1,v_2'',P_2$ which produces this choice of $K'$. Therefore, combining our lower bounds on the choices made, we have
\[
d_{L_{e,f}}(K)\geq(1- \eta)d^2\cdot (1-2\eta)\delta\cdot \left(1-2\eta\right)\frac{d^{\ell}\delta^{\ell}}{n}\cdot (1-2\eta)\delta\cdot (1-2\eta)\frac{d^{\ell}\delta^{\ell}}{n} \geq (1-10\eta)\frac{d^{2\ell+2}\delta^{2\ell+2}}{n^2},
\]
as required.
\claimproofend

From \ref{prop:degreebound} and \ref{prop:upperboundFKKalternating}, we have that, for each $K\in \cK_d(F)\setminus \cK^{\mathrm{bad}}$ and each distinct $x,y\in [n]$ the number of $(F\setminus K,K)$-alternating $x,y$-paths of length $2\ell+2$ is at most $(1+ \eta)\delta\cdot d\cdot \left(1+\eta\right)d^{\ell}\delta^{\ell}/n$.
Thus, applying this to $(u_1,v_1)$ and $(u_2,v_2)$, and using $(1+\eta)^4\leq (1+5\eta)$, we have the following property.

    \stepcounter{propcounter}
\begin{enumerate}[label = {{\textbf{\Alph{propcounter}\arabic{enumi}}}}]
\item For each $K'\in \cK^f\setminus \mathcal{K}^{\textrm{bad}}$, we have $d_{L_{e,f}}(K')\leq (1+5\eta)d^{2\ell+2}\delta^{2\ell+2}/n^2$.\label{prop:maxdegreeLef}
\end{enumerate}
Furthermore, using \ref{prop:degreebound}, we have that for each $K\in \cK_d(F)$ and each distinct $x,y\in [n]$ the number of $(F\setminus K,K)$-alternating $x,y$-paths of length $2\ell+2$ is at most $(1+\eta)^{\ell+1}\delta^{\ell+1}d^{\ell+1}\leq \sqrt{n}\cdot \delta^{\ell+1}d^{\ell+1}$, where we have used that $\ell\leq 10\log n$ and $\eta\ll 1$. Therefore, we have the following property.

\begin{enumerate}[label = {{\textbf{\Alph{propcounter}\arabic{enumi}}}}]
\addtocounter{enumi}{1}
\item For each $K'\in \cK^f$, we have $d_{L_{e,f}}(K')\leq n\cdot \delta^{2\ell+2}d^{2\ell+2}$.\label{prop:unimaxdegreeLef}
\end{enumerate}

Using that $|\cK^{\mathrm{bad}}|\leq n^{-6}|\cK_d(F)|$,
 we have
\begin{align}
\label{eq:eLupperbound2}
e(L_{e,f})&\overset{\ref{prop:maxdegreeLef},\ref{prop:unimaxdegreeLef}}\leq |\cK^f|\cdot \left(1+ 5\eta \right)\frac{\delta^{2\ell+2}d^{2\ell+2}}{n^2}+\frac{|\cK_d(F)|}{n^{6}}\cdot n\cdot \delta^{2\ell+2}d^{2\ell+2}\nonumber\\
&\overset{\mathrm{Clm}~\ref{clm:edgeswellspread2}}{\leq}
|\cK^f|\cdot \left(1+ 5\eta\right)\frac{\delta^{2\ell+2}d^{2\ell+2}}{n^2}+\frac{|\cK^e|}{n}\cdot \frac{\delta^{2\ell+2}d^{2\ell+2}}{n^2},
\end{align}
where we have used that $|\cK^e|>|\cK_d(F)|/n^2$.

Furthermore, as $|\cK^e|>|\cK_d(F)|/n^2$, we have $|\mathcal{K}^e\setminus \mathcal{K}^{\textrm{bad}}|\geq (1-1/n)|\mathcal{K}^e|$. Thus, by Claim~\ref{clm:mindegreeLef2}, we have
\begin{align*}
e(L_{e,f})&\geq (1-1/n)|\mathcal{K}^e|\cdot (1-10\eta)\frac{d^{2\ell+2}\delta^{2\ell+2}}{n^2}.
\end{align*}
 In combination with \eqref{eq:eLupperbound2}, we have that
\[
(1-3/n)\cdot |\cK^e|\cdot \left(1-10\eta\right)\frac{\delta^{2\ell+2}d^{2\ell+2}}{n^2}\leq |\cK^f|\cdot \left(1+ 5\eta\right)\frac{\delta^{2\ell+2}d^{2\ell+2}}{n^2}.
\]
As $\eta d\geq C_0$ implies $\eta n\geq C_0$, we have that $|\{K\in \mathcal{K}_d(F):e\in E(K)\}|\leq \left(1+18\eta\right)|\{K\in \mathcal{K}_d(F):f\in E(K)\}|$, a contradiction to \eqref{eqn:forcontradiction:new}. 

Therefore, for each $e,f\in E(F)$ which share no vertices we have $|\{K\in \mathcal{K}_d(F):e\in E(K)\}|\leq \left(1+18\eta\right)|\{K\in \mathcal{K}_d(F):f\in E(K)\}|$. Let, then, $e,f\in E(F)$. Pick $e'\in E(F)$ with no vertices in $V(e)\cup V(f)$. Then, we have
\begin{align*}
|\{K\in \mathcal{K}_d(F):e\in E(F)\}|&\leq \left(1+18\eta \right)|\{K\in \mathcal{K}_d(F):e'\in E(K)\}|\\
&\leq \left(1+40\eta \right)|\{K\in \mathcal{K}_d(F):f\in E(K)\}|,
\end{align*}
 as required.
\end{proof}


\section{Upper bounds for alternating paths from a fixed vertex}\label{sec:uppernottethered}

We now prove the following lemma, which gives a good upper bound on the number of $(F\setminus K,K)$-alternating paths starting from a fixed vertex. (That is, unlike our ultimate goal of \ref{prop:upperboundFKKalternating}, one of the endpoints of the paths considered is not fixed). As with the next two sections, we consider the regime corresponding to $m\leq n^{1+\sigma}$ in Theorem~\ref{thm:maintechnicalthm-forprocess}, for some small fixed $\sigma>0$, but we prove that one of the properties, \ref{eq:onlydegreesinFK:random} in Lemma~\ref{lem:nbrhoodsum}, holds more widely so that we can use it in Section~\ref{sec:noncritical} and also to conclude that \ref{prop:degreebound} is likely to hold.
As with our results across the next three sections, instead of working in $F(n,d,m)$ here we work with the slightly easier to study random graph $F$ formed by taking $K\sim G_d(n)$ and adding edges independently at random with some probability $p$ chosen so that the expected number of added edges is $m$. When we apply the main results of these sections, in Section~\ref{sec:maintechnicalconclusion}, at the cost of a manageable increase in the probability the events considered hold, we will easily be able to translate these results into the model $F(n,d,m)$.

\begin{lemma}\label{lem:ubpathcountF:untethered}
Let $1/n\ll 1/C\ll \eps\leq 1$ and $1/C\ll \sigma, \mu \ll 1$. Let $d\in [n]$ and $m\in \N$ satisfy $d\geq C\log n$ and  $\eps dn/2\leq m\leq n^{1+\sigma}$. Let $\delta=2m/n$ and 
let
\begin{equation*}
\lambda = \max\left\{\frac{\mu^2}{\log^2 n},\left(\frac{n\log n}{m}\right)^{1/4}\right\}.
\end{equation*}
Let $p=m(\binom{n}{2}-dn/2)^{-1}$. Let $K\in \cK_d(n)$, form $E$ by including each element of $[n]^{(2)}\setminus E(K)$ uniformly at random with probability $p$, and let $F=K+E$. 

Then, with probability $1-o(n^{-14})$, we have the following.
\stepcounter{propcounter}
\begin{enumerate}[label = {\emph{\textbf{\Alph{propcounter}}}}]
\item For each $1\leq \ell \leq 10\log n$ and $v\in [n]$, the number of $(F\setminus K,K)$-alternating paths with length $2\ell$ starting at $v$ is at most $(1+30\lambda)d^{\ell}\delta^{\ell}$.\labelinthm{prop:upperboundFKKalternating:random:newnew}
\end{enumerate}
\end{lemma}

In Section~\ref{sec:pseudforupper} we give pseudorandom conditions (\ref{eq:onlydegreesinFK} and \ref{eq:nbrhoodsinFK}) and show they imply \ref{prop:upperboundFKKalternating:random:newnew}.
These conditions concern the degrees in $F\setminus K$, as well as certain sums of degrees, and we prove the likely bounds we need in Section~\ref{sec:F-Kdegreebounds}. Finally, we use this to conclude that Lemma~\ref{lem:ubpathcountF:untethered} holds in Section~\ref{sec:ubpathcountF:untethered}.


\subsection{Pseudorandom properties for upper bounds for paths from a vertex}\label{sec:pseudforupper}
To show the number of $(F\setminus K,K)$-alternating paths from a vertex are bounded we will use bounds on the degrees of $F\setminus K$ and, as discussed in Section~\ref{subsec:switching}, a more complicated condition on certain degree sums required for the accuracy we need. Due to this more complicated condition, we prove the lemma by induction -- considering not the number of paths, but the sum of the degrees in $F\setminus K$ of the endpoints of the paths we wish to consider. 

\begin{lemma}\label{lem:ubpathcountF}
Let $0< \lambda \ll 1$. Let $d,\delta>0$, let $K\in \cK_d(n)$ and let $F$ be a graph with vertex set $[n]$ for which the following properties hold. 
\stepcounter{propcounter}
\begin{enumerate}[label = {\emph{\textbf{\Alph{propcounter}\arabic{enumi}}}}]
\item For each $v\in [n]$, $d_{F\setminus K}(v)\leq \left(1+ \lambda\right)\delta$.\labelinthm{eq:onlydegreesinFK}
\item For each $v\in [n]$,\labelinthm{eq:nbrhoodsinFK}
\begin{equation*}
\sum_{u\in N_{F\setminus K}(v)}\sum_{u'\in N_K(u)}d_{F\setminus K}(u')\leq
\left(1+ \frac{\lambda}{\log n}\right)\delta d\cdot d_{F\setminus K}(v).
\end{equation*}
\end{enumerate}
Then, the following holds.
\begin{enumerate}[label = {\emph{\textbf{\Alph{propcounter}\arabic{enumi}}}}]\addtocounter{enumi}{2}
\item For each $1\leq \ell \leq 10\log n$ and $v\in [n]$, the number of $(F\setminus K,K)$-alternating paths with length $2\ell$ starting at $v$ is at most $(1+30\lambda)d^{\ell}\delta^{\ell}$.\labelinthm{prop:upperboundFKKalternating:random:new}
\end{enumerate}
\end{lemma}
\begin{proof} Let $v'\in [n]$.
We will first show by induction that, for each $1\leq i\leq 10\log n$,
\begin{align}
\sum_{u\in V(G)}&\#\{\text{$(K,F\setminus K)$-alternating $v',u$-paths length $2i-1$}\}\cdot d_{F\setminus K}(u)\nonumber\\
&\hspace{4cm}\leq \left(1+ \lambda\right)\cdot \left(1+\frac{\lambda}{\log n}\right)^{i-1}\delta^id^i.\label{eq:forind}
\end{align}
For this, first note that \eqref{eq:forind} holds for $i=1$ by \ref{eq:onlydegreesinFK}, using that there are $d$ vertices $u\in V(G)$ for which there is a $(K,F\setminus K)$-alternating $v',u$-path of length 1. Suppose then that $1<i\leq 10\log n$ and that \eqref{eq:forind} holds with $i$ replaced by $i-1$. Then,
\begin{align*}
\sum_{u\in V(G)}&\#\{\text{$(K,F\setminus K)$-alternating $v',u$-paths of length $2i-1$}\}\cdot d_{F\setminus K}(u)\nonumber\\
&\hspace{-0.7cm}\leq \sum_{u'\in V(G)}\#\{\text{$(K,F\setminus K)$-alternating $v',u'$-paths of length $2i-3$}\}\cdot \sum_{u''\in N_{F\setminus K}(u')}\sum_{u\in N_{K}(u'')}d_{F\setminus K}(u)\\
&\hspace{-0.7cm}\overset{\ref{eq:nbrhoodsinFK}}{\leq} \left(1+ \frac{\lambda}{\log n}\right)\cdot \delta d\cdot
\sum_{u'\in V(G)}\#\{\text{$(K,F\setminus K)$-alternating $v',u'$-paths length $2i-3$}\}\cdot d_{F\setminus K}(u')\\
&\hspace{-0.7cm}\leq \left(1+ \lambda\right)\cdot \left(1+\frac{\lambda}{\log n}\right)^{i-1}\delta^id^i.
\end{align*}
Therefore, \eqref{eq:forind} holds for $i$. Thus, \eqref{eq:forind} holds for all $1\leq i\leq 10\log n$.

Now let $v'$ be in $N_{F\setminus K}(v)$. 
Clearly if $\ell = 1$ then  the number of $(F\setminus K,K)$-alternating paths with length $2\ell$ starting at $v$ is $|N_{F\setminus K}(v)|\cdot d \le (1+\lambda)\delta d$  by \ref{eq:onlydegreesinFK}. For $\ell \ge 2$, it follows from \eqref{eq:forind} that, for each $v\in [n]$, the number of $(F\setminus K,K)$-alternating paths with length $2\ell$ starting at $v$ is at most
\begin{align*}
|N_{F\setminus K}(v)|\cdot \left(1+ \lambda\right)\left(1+\frac{\lambda}{\log n}\right)^{\ell-2}\delta^{\ell-1}d^{\ell-1} \cdot d
&\leq (1+\lambda)^2\left(1+\frac{2\ell\lambda}{\log n}\right)\delta^{\ell}d^{\ell}\leq \left(1+30\lambda\right)\delta^id^i,
\end{align*}
where we have used \ref{eq:onlydegreesinFK} and that $\ell \le 10 \log{n}$ and $\lambda\ll 1$. Thus, \ref{prop:upperboundFKKalternating:random:new} holds. 
\end{proof}


\subsection{Likely properties of $F\setminus K$: degree bounds}\label{sec:F-Kdegreebounds}

We now prove a result we mainly use to deduce the pseudorandom conditions used in Lemma~\ref{lem:ubpathcountF}. Note that, as mentioned at the start of the section, the first of these properties is proved for $(d,m)$ more widely so that it can be applied in Section~\ref{sec:noncritical} and to show that~\ref{prop:degreebound} is likely to hold.

\begin{lemma}\label{lem:nbrhoodsum}
Let $1/n\ll 1/C\ll 1/C',\sigma \ll 1$ and $1/C\ll \eps$. Let $d\geq C\log n$ and let $m$ satisfy $\eps dn/2 \leq  m\leq \binom{n}{2}-dn/2$. Let $\delta=2m/n$ and let $p=m\left(\binom{n}{2}-dn/2\right)^{-1}$. Let $K\in \mathcal{K}_d(n)$, form $E$ by including each element of $[n]^{(2)}\setminus E(K)$ uniformly at random with probability $p$, and let $F=K+E$.
Then, the following holds with probability $1-o(n^{-14})$.
\stepcounter{propcounter}
\begin{enumerate}[label = {\emph{\textbf{\Alph{propcounter}\arabic{enumi}}}}]
\item For each $v\in [n]$,\labelinthm{eq:onlydegreesinFK:random}
\begin{equation*}
\left(1- C'\sqrt{\frac{\log n}{\delta}}\right)\delta\leq d_{F\setminus K}(v)\leq \left(1+ C'\sqrt{\frac{\log n}{\delta}}\right)\delta.
\end{equation*}
\end{enumerate}
Furthermore, if $d\leq n^{\sigma}$ and $m\leq n^{1+\sigma}$, the following holds with probability $1-o(n^{-14})$.
\begin{enumerate}[label = {\emph{\textbf{\Alph{propcounter}\arabic{enumi}}}}]\addtocounter{enumi}{1}
\item For each $v\in [n]$,\labelinthm{eq:nbrhoodsinFK:random}
\begin{equation*}
\left(1-\frac{2}{\delta}\right)d_{F\setminus K}(v)\cdot \delta d\leq
\sum_{u\in N_{F\setminus K}(v)}\sum_{u'\in N_K(u)}d_{F\setminus K}(u')\leq
\left(1+ \frac{2}{\delta}\right)d_{F\setminus K}(v)\cdot \delta d.
\end{equation*}
\end{enumerate}
\end{lemma}
\begin{proof}
For each $v\in [n]$, by an application of Chernoff's bound, as $d_{F\setminus K}(v)$ is binomially distributed with parameters $n-1-d$ and $p$, and $p(n-1-d)=2m/n=\delta$, we have that
\begin{equation}\label{eq:defFprimeH}
\P\Bigg(d_{F\setminus K}(v)\neq \left(1\pm C'\sqrt{\log n/\delta}\right)\delta\Bigg)\leq 2\exp\left(-\frac{C'^2\log n}{3\delta}\cdot \delta\right)=o(n^{-15}).
\end{equation}
Thus, taking a union bound over all $v\in [n]$, with probability $1-o(n^{-14})$, \ref{eq:onlydegreesinFK:random} holds, as desired.

Let, then, $v\in [n]$ and suppose $d \le n^{\sigma}$ and $m \le n^{1+\sigma}$. Reveal the edges of $E$ next to $v$, and let $A=N_{F\setminus K}(v)$. Note that the probability that there is a vertex $u\in [n]$ with $|N_K(u)\cap A|\geq 100$ is, as $d\leq n^{\sigma}$ and $p\leq 4m/n^2\leq 4n^{\sigma-1}$, at most
\[
n\cdot \binom{d}{100}p^{100}\leq n(pd)^{100}=o(n^{-15}).
\]
Thus, with probability $1-o(n^{-15})$, for each vertex $u\in [n]$, $|N_K(u)\cap A|< 100$.
Furthermore, by \eqref{eq:defFprimeH}, we have $\delta/2\leq |A|\leq 2\delta$ with probability $1-o(n^{-15})$.

Let $\hat{E}_v$ be the set of edges $e\in ([n]\setminus \{v\})^{(2)}\setminus E(K)$.
For each $xy\in \hat{E}_v$, let $w_{xy}=|N_K(x)\cap A|+|N_K(y)\cap A|$, so that $w_{xy}\leq 200$ as  $|N_K(u)\cap A|< 100$ for each $u\in [n]$. Note that  
\begin{equation}\label{eq:sumwxy}
\sum_{xy\in \hat{E}_v}w_{xy}=\sum_{u\in A}\sum_{u'\in N_K(u)}d_{(K_n\setminus K)-v}(u')=\left(1\pm \frac{1}{4\delta}\right)|A|d(n-1-d),
\end{equation}
where we have used that, for each $u'\in [n]$, we have $d_{(K_n\setminus K)-v}(u')\in \{n-2-d,n-1-d\}$.
Therefore, we can take subsets $\hat{E}_{v,i}$, $i\in [200]$, in $\hat{E}_v$ such that $|\hat{E}_{v,i}|\geq |A|d(n-1-d)/400$ and, for each $xy\in \hat{E}_v$, $|\{i\in [200]:xy \in \hat{E}_{v,i}\}|=w_{xy}$. For example, if we randomly assign each $xy\in \hat{E}_v$ to $w_{xy}$ of the $\hat{E}_{v,i}$, $i\in [200]$, for each $xy \in \hat{E}_v$. Then, for each $i\in [200]$, we have $\mathbb{E}|\hat{E}_{v,i}| = \sum_{xy \in \hat{E}_v} \frac{w_{xy}}{200}$, and by Theorem~\ref{thm:mcd:conc} there is a positive probability that $|\hat{E}_{v,i}| \ge 2\mathbb{E}|\hat{E}_{v,i}|/3\geq |A|d(n-1-d)/400$ for all $i\in [200]$, where we have used \eqref{eq:sumwxy}.

For each $i\in [200]$, let $X_{v,i}=|\hat{E}_{v,i}\cap E|$. Revealing the edges of $E$ not containing $v$, we have that, for each $i\in [200]$, $X_{v,i}$ is a binomial random variable with mean $\E|X_{v,i}|=p|\hat{E}_{v,i}|\geq |A|dp(n-1-d)/400\geq \delta^2d/800$. Therefore, we have, via a Chernoff bound, that, for each $i\in [200]$,
\begin{align}\label{eq:probXvi}
\P\left(|X_{v,i}-\E X_{v,i}|> \frac{1}{10\delta}\cdot \E X_{v,i}\right)&\leq  2\exp\left(- \frac{\E X_{v,i}}{3\cdot 100\delta^2}\right)\leq 2\exp\left(- \frac{d}{10^6}\right)=o(n^{-15}),
\end{align}
where we have used that $d\geq C\log n$.

Let 
\[
X_v=\sum_{u\in A}\sum_{u'\in N_K(u)}d_{(F\setminus K)-v}(u')=\sum_{i\in [200]} X_{v,i},
\]
Then, by \eqref{eq:probXvi}, with probability $1-o(n^{-15})$, we have ${|X_{v,i}-\E X_{v,i}|}\leq  \frac{1}{10\delta}\cdot \E X_{v,i}$ for each $i\in [200]$ and, hence, 
\begin{equation}\label{eq:Xvdeviance}
|X_v-\E X_v|\leq \sum_{i\in [200]}(|X_{v,i}-\E X_{v,i}|)\leq \sum_{i\in [200]} \frac{1}{10\delta}\cdot \E X_{v,i}=\frac{1}{10\delta}\cdot \E X_v.
\end{equation}
Furthermore, by \eqref{eq:sumwxy}, we have that $\E X_v=\left(1\pm \frac{1}{4\delta}\right)|A|dp(n-d-1)=\left(1\pm \frac{1}{4\delta}\right)|A|\cdot d\delta$. Thus, \eqref{eq:Xvdeviance} implies that $X_v=\left(1\pm \frac{1}{2\delta}\right)|A|\cdot \delta d=\left(1\pm \frac{1}{2\delta}\right)d_{F\setminus K}(v)\cdot \delta d$.

Noting that
\[
\left|X_v-\sum_{u\in N_{F\setminus K}(v)}\sum_{u'\in N_K(u)}d_{F\setminus K}(u')\right|\leq \sum_{u\in N_{F\setminus K}(v)}\sum_{u'\in N_K(u)}1\leq d\cdot d_{F\setminus K}(v)= \frac{1}{\delta}\cdot d_{F\setminus K}(v)\cdot \delta d,
\]
we have that \ref{eq:nbrhoodsinFK:random} holds for $v$ with probability $1-o(n^{-15})$. Taking a union bound over all $v\in [n]$ thus shows that \ref{eq:nbrhoodsinFK:random} holds with probability $1-o(n^{-14})$, as desired.
\end{proof}


\subsection{Proof of Lemma~\ref{lem:ubpathcountF:untethered}}\label{sec:ubpathcountF:untethered}

Combining Lemma~\ref{lem:ubpathcountF} and Lemma~\ref{lem:nbrhoodsum}, it is now easy to prove Lemma~\ref{lem:ubpathcountF:untethered}.
\begin{proof}[Proof of Lemma~\ref{lem:ubpathcountF:untethered}]
 Take the variables as set up in Lemma~\ref{lem:ubpathcountF:untethered}, so that $1/n\ll 1/C \ll \eps\leq 1$, $1/C\ll \sigma,\mu\ll 1$, $d\geq C\log n$, $\eps dn/2\leq m\leq n^{1+\sigma}$, $\delta=2m/n$ and $p=m\left(\binom{n}{2}-dn/2\right)^{-1}$. Furthermore, we have
\begin{equation}\label{eq:lambdadefn:1}
\lambda = \max\left\{\frac{\mu^2}{\log^2 n},\left(\frac{n\log n}{m}\right)^{1/4}\right\}.
\end{equation}
Let $K\sim G_d(n)$, form $E$ by including each element of $[n]^{(2)}\setminus E(K)$ uniformly at random with probability $p$, and let $F=K+E$. Thus, we want to show that \ref{prop:upperboundFKKalternating:random:newnew} holds with probability $1-o(n^{-14})$.

Let $C'$ satisfy $1/C \ll 1/C'\ll 1$.
As $\delta=2m/n \geq \eps C \log n$, we have that
\[
C'\sqrt{\frac{\log n}{\delta}}\overset{\eqref{eq:lambdadefn:1}}{\leq} C'\sqrt{\frac{\log n}{\delta}}\cdot \lambda \left(\frac{m}{n\log n}\right)^{1/4}\leq C'\sqrt{\frac{\log n}{\delta}}\cdot \lambda \left(\frac{\delta}{2 \log n}\right)^{1/4}=C'\lambda \left(\frac{\log n}{2 \delta}\right)^{1/4}\leq \lambda.
\]
 Also, since $m \ge \eps C n\log{n}/2$, we have
\[
\frac 2 \delta = \frac{n}{m} = \left(\frac {n \log n} {m}\right)^{3/4} \cdot \frac{1}{\log{n}}\left(\frac{n \log n}{m}\right)^{1/4} \overset{\eqref{eq:lambdadefn:1}}{\leq} \frac{\lambda}{\log n}.
\]
Thus, by Lemma~\ref{lem:nbrhoodsum}, we have that \ref{eq:onlydegreesinFK} and \ref{eq:nbrhoodsinFK} hold with probability $1-o(n^{-14})$.
Then, by Lemma~\ref{lem:ubpathcountF}, we have that \ref{prop:upperboundFKKalternating:random:newnew} holds with probability $1-o(n^{-14})$.
\end{proof}

\section{Lower bounds for switching paths}\label{sec:lower}

We now prove our likely lower bound on the number of $(F\setminus K,K)$-alternating paths between two fixed vertices, that appears as \ref{prop:lowerboundKFKalternating} in Theorem~\ref{thm:switching:propstoapply}, as follows.

\begin{theorem}\label{thm:randomproperties:lowerbounds:newnew}
Let $1/n\ll 1/C\ll 1/C_0\ll \eps\leq 1$ and $1/C\ll \sigma, \mu \ll 1$. Let $d\in [n]$ and $m\in \N$ satisfy $d\geq C\log n$ and  $\eps dn/2\leq m\leq n^{1+\sigma}$. Let $\delta=2m/n$ and 
let
\begin{equation*}
\lambda = \max\left\{\frac{\mu^2}{\log^2 n},\left(\frac{n\log n}{m}\right)^{1/4}\right\}.
\end{equation*}
Let $p=m(\binom{n}{2}-dn/2)^{-1}$. Let $K\sim G_d(n)$, form $E$ by including each element of $[n]^{(2)}\setminus E(K)$ uniformly at random with probability $p$, and let $F=K+E$. Let $\ell_0$ be the smallest integer for which
$\delta^{(\ell_0-1)}d^{(\ell_0-1)}\geq n$.
Then, with probability $1-o(n^{-14})$, we have the following.
\stepcounter{propcounter}
\begin{enumerate}[label = {\emph{\textbf{\Alph{propcounter}}}}]\addtocounter{enumi}{1}
\item \labelinthm{prop:lowerboundKFKalternating:newnewnew} For each distinct $x,y\in [n]$, each $Z\subset [n]$ with $|Z|\leq 40\log n$, and each $\ell$ with $10\ell_0\leq \ell\leq 10\log n$ the number of $(F\setminus K,K)$-alternating $x,y$-paths of length $2\ell$ which do not have any internal vertices in $Z$ is at least $\left(1-C_0\lambda - \frac{|(Z\cup \{x\})\cap N_K(y)|}{d}\right)d^{\ell}\delta^{\ell}/n$.
\end{enumerate}
\end{theorem}

In Section~\ref{sec:pseudforlower}, we give pseudorandom conditions (\ref{eq:degreesinFK:pseud}--\ref{prop:F-Kconnection}) and show they imply \ref{prop:lowerboundKFKalternating:newnewnew}. Of these pseudorandom conditions, \ref{eq:degreesinFK:pseud} and \ref{eq:nbrhoodsinFK:pseud} will already follow from Lemma~\ref{lem:nbrhoodsum}. The conditions \ref{eq:fewdoubled:K}, \ref{eq:fewdoubled:F-K} and \ref{prop:new} will be used to show expansion conditions in the graphs $K$ and $F\setminus K$, and we prove \ref{eq:fewdoubled:K} is likely to hold in Section~\ref{sec:Kexpansion} and \ref{eq:fewdoubled:F-K} and \ref{prop:new} is likely to hold in Section~\ref{sec:F-Kexpansion}. The condition \ref{prop:F-Kconnection} concerns edges in $F\setminus K$ between large sets and is shown to likely hold in Section~\ref{sec:connection}. Finally, we put this together in Section~\ref{sec:thmproof:lowerbounds} to prove Theorem~\ref{thm:randomproperties:lowerbounds:newnew}.


\subsection{Pseudorandom properties for lower bounds for path counts}\label{sec:pseudforlower}
To show good bounds on the expansion of certain vertex sets $U$ in a graph $H$ (applied with $H=K$ and $H=F\setminus K$ for Theorem~\ref{thm:randomproperties:lowerbounds:newnew}), we will use that few vertices have more than one neighbour in $H$ in $U$, so that the size of $N_H(U)$ is close to the sum of the degrees in $H$ of the vertices in $U$. For this, we will use the following convenient definition.
\begin{defn}
For a vertex set $U$ in a graph $F$, we define $E_{\geq 2,F}(U)$ to be the set of edges $xy\in E(F)$ with $x\in U$, $y\notin U$ for which there is some $x'\in U\setminus \{x\}$ with $x'y\in E(F)$.
\end{defn}

We now give our pseudorandom conditions for bounding below the number of $(F\setminus K,K)$-alternating paths between two fixed vertices, and deduce the bound we need, as follows.

\begin{lemma}\label{lem:lbpathcountF}
Let $1/n\ll 1/C\ll 1/C_0\ll 1$ and $1/C\ll \sigma\ll 1$. Let $d\geq C\log n$ with $d,\delta\leq n^{\sigma}$. Let $\log^{-3}n\leq \lambda \ll 1$ satisfy $\lambda \delta\geq C_0\log n$. Let $K\in \cK_d(n)$ and let $F$ be a graph with vertex set $[n]$, such that $K \subset F$ and the following properties hold.

\stepcounter{propcounter}
\begin{enumerate}[label = {\emph{\textbf{\Alph{propcounter}\arabic{enumi}}}}]
\item For each $v\in [n]$,
$d_{F\setminus K}(v)= \left(1\pm \lambda\right)\delta$. \labelinthm{eq:degreesinFK:pseud}

\item For each $v\in [n]$,\labelinthm{eq:nbrhoodsinFK:pseud}
\begin{equation*}
\sum_{u\in N_{F\setminus K}(v)}\sum_{u'\in N_K(u)}d_{F\setminus K}(u')=
\left(1\pm  \frac{\lambda}{\log n}\right)d_{F\setminus K}(v)\cdot \delta d.
\end{equation*}
\item For each $U\subset [n]$, if there is some $U'\subset [n]$ with $|U'|\leq 4|U|/\delta$ and $U\subset U'\cup N_{F\setminus K}(U')$, then \labelinthm{eq:fewdoubled:K}
\[
|E_{\geq 2,K}(U)|+e_{K}(U)\leq \left\{
\begin{array}{ll}
\frac{\lambda}{\log n}d |U| & \text{ if }|U|\leq \frac{\lambda n}{10^6d\log n}\\
\lambda d |U| & \text{ if }|U|\leq \frac{\lambda n}{10^6d}.
\end{array}\right.
\]
\item For each $U\subset [n]$ with $|U|\leq {\lambda n/100\delta\log n}$, if there is some $U'\subset [n]$ with $|U'|\leq 4|U|/d$ and $U\subset U'\cup N_{K}(U')$, then \labelinthm{eq:fewdoubled:F-K}
\[
|E_{\geq 2,F\setminus K}(U)|+e_{F\setminus K}(U)\leq \frac{\lambda}{\log n}\delta |U|.
\]
\item \labelinthm{prop:new} For each $U\subset [n]$ with $|U|\leq 100\log n$, $|\{v\in [n]\setminus U: |N_{F\setminus K}(v) \cap U|\geq 10\}|\leq 100\log n$.
\item \labelinthm{prop:F-Kconnection} For any disjoint sets $U,V\subset [n]$ with $|U|,|V|\geq \frac{\lambda n}{10^8}$, we have
\[
e_{F\setminus K}(U,V)\geq \left(1- \lambda\right)\frac{\delta}{n}|U||V|.
\]
\end{enumerate}
Let $\ell_0$ be the smallest integer for which
$\delta^{(\ell_0-1)}d^{(\ell_0-1)}\geq n$. Then, the following holds.

\begin{enumerate}[label = {\emph{\textbf{\Alph{propcounter}\arabic{enumi}}}}]\addtocounter{enumi}{6}
\item \labelinthm{prop:lowerboundKFKalternating:pseud:newnew} For each distinct $x,y\in [n]$, each $Z\subset [n]$ with $|Z|\leq 40\log n$, and each $10\ell_0\leq \ell\leq 10\log n$ the number of $(F\setminus K,K)$-alternating $x,y$-paths of length $2\ell$ which do not have any internal vertices in $Z$ is at least $\left(1-C_0\lambda - \frac{|(Z\cup \{x\})\cap N_K(y)|}{d}\right)d^{\ell}\delta^{\ell}/n$.
\end{enumerate}
\end{lemma}
\begin{proof}
Let $\ell_1$ be the smallest integer for which $\delta^{(\ell_1+1)}d^{(\ell_1+1)}\geq \lambda n/(2\cdot 10^7)$, so that $\ell_0-3\leq \ell_1\leq \ell_0-2$ as $\lambda \delta\geq C_0\log n$. As $d,\delta\leq n^{\sigma}$, we have that $\ell_1\geq 1$, and as $d\geq C\log n$ and $\delta\geq C_0\log n$ we have $\ell_1\leq \log n/\log\log n$.  Let $r=\lceil \lambda n/(8\cdot 10^7\cdot \delta^{\ell_1}d^{\ell_1})\rceil\leq \delta d/2$. We will show the following claim.

\begin{claim}\label{clm:lowerbound} Let $X_0,Y_0\subset [n]$ be disjoint with $|X_0|=|Y_0|=r$ such that there are disjoint $X_0',Y_0'\subset [n]$ with $X_0\subset N_K(X_0')$, $Y_0\subset N_K(Y_0')$ and $|X_0'|,|Y_0'|\leq 4r/d$. Then, for each $W\subset [n]$ with $|W|\leq 3r +\lambda \delta$, there are at least $(1-C_0\lambda/4)r^2\delta^{2\ell_1+1}d^{2\ell_1}/n$ $(F\setminus K,K)$-alternating paths of length $4\ell_1+1$ which start in $X_0$, end in $Y_0$, and have no internal vertices in $X_0\cup Y_0\cup W$.
\end{claim}
\claimproofstart[Proof of Claim~\ref{clm:lowerbound}]
Iteratively, for each $1\leq i\leq \ell_1$, let
\[
X_{2i-1}=N_{F\setminus K}(X_{2i-2})\big{\backslash} \left(\left(\cup_{j=0}^{2i-2}(X_j\cup Y_j)\right)\cup N_{F\setminus K}(Y_{2i-2})\cup W\right),
\]
\[
Y_{2i-1}=N_{F\setminus K}(Y_{2i-2})\big{\backslash} \left(\left(\cup_{j=0}^{2i-2}(X_j\cup Y_j)\right)\cup N_{F\setminus K}(X_{2i-2})\cup W\right),
\]
\[
X_{2i}=N_{K}(X_{2i-1})\big{\backslash} \left(\left(\cup_{j=0}^{2i-1}(X_j\cup Y_j)\right)\cup N_{K}(Y_{2i-1})\cup W\right),
\]
and
\[
Y_{2i}=N_{K}(Y_{2i-1})\big{\backslash} \left(\left(\cup_{j=0}^{2i-1}(X_j\cup Y_j)\right)\cup N_{K}(X_{2i-1})\cup W\right).
\]
Note that, for each edge in $F\setminus K$ between $X_{2\ell_1}$ and $Y_{2\ell_1}$, there is an $(F\setminus K,K)$-alternating path of length $4\ell_1+1$ which starts in $X_0$ and ends in $Y_0$, and has no internal vertices in $X_0\cup Y_0\cup W$, and each such path is different. Thus, to prove the claim, it is sufficient to prove that $e_{F\setminus K}(X_{2\ell_1},Y_{2\ell_1})\geq \left(1-{C_0\lambda}/{4}\right)\cdot {r^2\delta^{2\ell_1+1} d^{2\ell_1}}/{n}$.

We will show by induction that, for each $1\leq i\leq \ell_1$,
\begin{align}\label{eq:induction:1}
\left(1-\frac{20\lambda}{\log n}-20\lambda\cdot \mathbf{1}_{\{i=1\}}\right)\cdot \sum_{v\in X_{2i-2}}d_{F\setminus K}(v)\leq |X_{2i-1}|\leq  \sum_{v\in X_{2i-2}} d_{F\setminus K}(v),
\end{align}
\begin{align}\label{eq:induction:2}
\left(1-\frac{20\lambda}{\log n}-20\lambda\cdot \mathbf{1}_{\{i=1\}}\right)\cdot \sum_{v\in Y_{2i-2}}d_{F\setminus K}(v)\leq |Y_{2i-1}|\leq \sum_{v\in Y_{2i-2}}d_{F\setminus K}(v),
\end{align}
\begin{align}\label{eq:induction:3}
\left(1-\frac{20\lambda}{\log n}-20\lambda\cdot \mathbf{1}_{\{i\in \{1,\ell_1\}\}}\right)\cdot d\delta \leq \frac{\sum_{v\in X_{2i}}d_{F\setminus K}(v)}{\sum_{v\in X_{2i-2}}d_{F\setminus K}(v)}\leq \left(1+\frac{20\lambda}{\log n}+20\lambda\cdot \mathbf{1}_{\{i\in \{1,\ell_1\}\}}\right)\cdot d\delta,
\end{align}
and
\begin{align}\label{eq:induction:4}
\left(1-\frac{20\lambda}{\log n}-20\lambda\cdot \mathbf{1}_{\{i\in \{1,\ell_1\}\}}\right)\cdot d\delta \leq \frac{\sum_{v\in Y_{2i}}d_{F\setminus K}(v)}{\sum_{v\in Y_{2i-2}}d_{F\setminus K}(v)}\leq \left(1+\frac{20\lambda}{\log n}+20\lambda\cdot \mathbf{1}_{\{i\in \{1,\ell_1\}\}}\right)\cdot d\delta.
\end{align}
For this, let $1\leq i \leq \ell_1$ and assume that \eqref{eq:induction:1}--\eqref{eq:induction:4} hold for each $1\leq a< i$.

We start by showing that, for each $1\leq a<i$,
\begin{equation}\label{eq:toshow:1}
\frac{4}{5}d^{a}\delta^{a}r\leq |X_{2a}|,|Y_{2a}|\leq \frac{6}{5}d^{a}\delta^{a}r.
\end{equation}
As $|X_0|=r$, using \ref{eq:degreesinFK:pseud} we have that $(1-\lambda)\delta r\leq \sum_{v\in X_0}d_{F\setminus K}(v)\leq (1+\lambda)\delta r$. Thus, by \eqref{eq:induction:3} applied for each $1\leq a\leq i$, we have, as $i\leq \ell_1=o(\log n)$ and $\lambda\ll 1$, that
\begin{align}
 \sum_{v\in X_{2a}}d_{F\setminus K}(v)\geq \left(1-\frac{20\lambda}{\log{n}}-20\lambda\right)\cdot \left(1-\frac{20\lambda}{\log n}\right)^{a-1}d^{a}\delta^{a}\cdot (1-\lambda)\delta r\geq
 \frac{9}{10}d^{a}\delta^{a+1}r\label{eq:dFKsumlower}\end{align}
and
\begin{align}
\sum_{v\in X_{2a}}d_{F\setminus K}(v)\leq  \left(1+\frac{20\lambda}{\log{n}}+20\lambda\right)\cdot\left(1+\frac{20\lambda}{\log n}\right)^{a-1}d^{a}\delta^{a}\cdot (1+\lambda)\delta r\leq \frac{11}{10}d^{a}\delta^{a+1}r.\label{eq:dFKsumupper}
\end{align}
Using \ref{eq:degreesinFK:pseud}, we therefore have that \eqref{eq:toshow:1} holds for $|X_{2a}|$. Similarly,  \eqref{eq:toshow:1} holds for $|Y_{2a}|$.

Now, for each $1\leq a<i$, using $|X_0|=|Y_0|=r$ when $a=1$ and otherwise \eqref{eq:dFKsumlower} and \eqref{eq:dFKsumupper} for $i=a-1$, by \eqref{eq:induction:1} with $i=a$
\begin{equation}\label{eq:toshow:2}
\frac{4}{5}d^{a-1}\delta^{a}r\leq |X_{2a-1}|,|Y_{2a-1}|\leq \frac{6}{5}d^{a-1}\delta^{a}r.
\end{equation}


If $i=1$, then set $X_{-1}=X'_0$ and $Y_{-1}=Y'_0$ (using the sets in the claim statement). 
Whether $i=1$, or not, we then have that $X_{2i-2}\cup Y_{2i-2}\subset N_K(X_{2i-3}\cup Y_{2i-3})$, and, from the properties of $X'_0$ and $Y'_0$ or by \eqref{eq:toshow:1} and \eqref{eq:toshow:2}, $|X_{2i-3}\cup Y_{2i-3}|\leq 3\delta^{i-1}d^{i-2}r\leq 4|X_{2i-2}\cup Y_{2i-2}|/d$.
Furthermore, either as $|X_0|=|Y_0|=r$ or by \eqref{eq:toshow:1} with $a=i-1$, and as $\delta^{\ell_1} d^{\ell_1}<\lambda n/(2\cdot 10^7)$ by the choice of $\ell_1$, we have that
\begin{equation}\label{eq:uboundcheck}
|X_{2i-2}\cup Y_{2i-2}|\leq 2\cdot \frac{6}{5}d^{i-1}\delta^{i-1}r< 2\cdot \frac{\lambda n}{10^7}\cdot d^{i-1-\ell_1}\delta^{i-1-\ell_1}\leq \frac{\lambda n}{\delta d}\leq \frac{\lambda n}{100\delta \log n}.
\end{equation}
Thus, from \ref{eq:fewdoubled:F-K}, we have,
\begin{align}\label{eq:aaa}
|E_{\geq 2,F\setminus K}(X_{2i-2}\cup Y_{2i-2})|\leq \frac{\lambda}{\log n}\cdot \delta |X_{2i-2}|\leq \frac{\lambda}{\log n}\cdot 2\delta^{i}d^{i-1}r.
\end{align}
Furthermore, we have $\cup_{a=0}^{2i-2} (X_{a}\cup Y_{a})\subset \left(\cup_{a=0}^{2i-3} (X_{a}\cup Y_{a})\right)\cup N_{K}\left(\cup_{a=0}^{2i-3} (X_{a}\cup Y_{a})\right)$,
and, using \eqref{eq:toshow:1} and \eqref{eq:toshow:2},  
\[
|\cup_{a=0}^{2i-3} (X_{a}\cup Y_{a})|\leq 2\cdot \frac{6}{5}\sum_{a=0}^{2i-3}\delta^{\lceil a/2\rceil}d^{\lfloor a/2\rfloor}r\leq 3\delta^{i-1}d^{i-2}r \leq\frac{4}{d}|X_{2i-2}\cup Y_{2i-2}|\leq\frac{4}{d}|\cup_{a=0}^{2i-2} (X_{a}\cup Y_{a})|.
\]
Similarly to \eqref{eq:uboundcheck}, we have $|\cup_{a=0}^{2i-2} (X_{a}\cup Y_{a})|\leq 
{\lambda n}/{100\delta \log n}$.
Thus, from \ref{eq:fewdoubled:F-K}, we have
\begin{align}\label{eq:bbb}
e_{F\setminus K}\left(\cup_{a=0}^{2i-2} (X_{a}\cup Y_{a})\right)
&\leq \frac{\lambda}{\log n}\cdot \delta \left|\cup_{a=0}^{2i-2} (X_{a}\cup Y_{a})\right|\leq \frac{\lambda}{\log n}\cdot 3\delta^{i}d^{i-1}r,
\end{align}
where we have used \eqref{eq:toshow:1} and \eqref{eq:toshow:2}.
Therefore, as $\lambda \delta\geq C_0\log n$, this implies that
\begin{equation}\label{eq:sizeW}
|W|\leq 3r+\lambda \delta\leq \left(\frac{\lambda}{\log n}+\lambda\cdot \mathbf{1}_{\{i=1\}}\right)\cdot \delta^id^{i-1}r,
\end{equation}
we have that
\begin{align}
\Bigg(\sum_{v\in X_{2i-2}}d_{F\setminus K}(v)\Bigg)-|X_{2i-1}|&\leq |E_{\geq 2,F\setminus K}(X_{2i-2}\cup Y_{2i-2})|+e_{F\setminus K}\left(\cup_{i'=0}^{2i-2}(X_{i'}\cup Y_{i'})\right)+|W|\nonumber\\
&\hspace{-0.6cm}\overset{\eqref{eq:aaa},\eqref{eq:bbb},\eqref{eq:sizeW}}{\leq}
\left(\frac{\lambda}{\log n}+\lambda\cdot \mathbf{1}_{\{i=1\}}\right)\cdot {6\delta^id^{i-1}r}\nonumber\\
&\overset{\eqref{eq:dFKsumlower}}{\leq} \left(\frac{\lambda}{\log n}+\lambda\cdot \mathbf{1}_{\{i=1\}}\right)\cdot 8\cdot \sum_{v\in X_{2i-2}}d_{F\setminus K}(v).\label{eqn:thisbound}
\end{align}
From \eqref{eqn:thisbound}, we get the lower-bound in \eqref{eq:induction:1} holds.
As $|X_{2i-1}|\leq |N_{F\setminus K}(X_{2i-2})|\leq \sum_{v\in X_{2i-2}}d_{F\setminus K}(v)$, we therefore have that \eqref{eq:induction:1} holds entirely. Furthermore, by a similar deduction to those above using \ref{eq:degreesinFK:pseud}, we have that $\frac{4}{5}d^{i-1}\delta^{i}r\leq |X_{2i-1}|\leq \frac{6}{5}d^{i-1}\delta^{i}r$.
Similarly, we can deduce that
\begin{align}
\Bigg(\sum_{v\in Y_{2i-2}}d_{F\setminus K}(v)\Bigg)-|Y_{2i-1}|&\leq  \left(\frac{\lambda}{\log n}+\lambda\cdot \mathbf{1}_{\{i=1\}}\right)\cdot 8\cdot  \sum_{v\in Y_{2i-2}}d_{F\setminus K}(v),\label{eqn:thisotherbound}
\end{align}
that \eqref{eq:induction:2} holds, and that $\frac{4}{5}d^{i-1}\delta^{i}r\leq |Y_{2i-1}|\leq \frac{6}{5}d^{i-1}\delta^{i}r$.

Now, working similarly to above, we have that $X_{2i-1}\cup Y_{2i-1}\subset N_{F\setminus K}(X_{2i-2}\cup Y_{2i-2})$ and $|X_{2i-2}\cup Y_{2i-2}|\leq 4|X_{2i-1}\cup Y_{2i-1}|/\delta$. Furthermore, we have that
\[
|X_{2i-1}\cup Y_{2i-1}|\leq 2\cdot \frac{6}{5}d^{i-1}\delta^{i}r\leq \frac{2\lambda n}{10^7}\cdot d^{i-1-\ell_1}\delta^{i-\ell_1}\leq \left\{\begin{array}{ll}
\frac{\lambda n}{10^6d \log n} & \text{ if }i<\ell_1,\\
\frac{\lambda n}{10^6d} & \text{ if }i=\ell_1.
\end{array}
\right.
\]
Thus, from \ref{eq:fewdoubled:K}, we have
\begin{align}
|E_{\geq 2,K}(X_{2i-1}\cup Y_{2i-1})|&\leq \left(\frac{\lambda}{\log n}+\lambda\cdot \mathbf{1}_{\{i=\ell_1\}}\right)\cdot d |X_{2i-1}\cup Y_{2i-1}|
\nonumber\\
&\leq \left(\frac{\lambda}{\log n}+\lambda\cdot \mathbf{1}_{\{i=\ell_1\}}\right)\cdot 3\delta^{i}d^ir.\label{eq:ccc}
\end{align}
Furthermore, we have that $\cup_{a=0}^{2i-1} (X_{a}\cup Y_{a})\subset \left(\cup_{a=0}^{2i-2} (X_{a}\cup Y_{a})\right)\cup N_{F\setminus K}\left(\cup_{a=0}^{2i-2} (X_{a}\cup Y_{a})\right)$,
and $|\cup_{a=0}^{2i-2} (X_{a}\cup Y_{a})|\leq 3d^{i-1}\delta^{i-1}r\leq 4|\cup_{a=0}^{2i-1} (X_{a}\cup Y_{a})|/\delta$, while
\[
\left|\cup_{a=0}^{2i-1} (X_{a}\cup Y_{a})\right|\leq 3d^{i-1}\delta^{i}r\leq \frac{3\lambda n}{2\cdot 10^7}\cdot d^{i-1-\ell_1}\delta^{i-\ell_1}\leq \left\{\begin{array}{ll}
\frac{\lambda n}{10^6d \log n} & \text{ if }i<\ell_1,\\
\frac{\lambda n}{10^6d} & \text{ if }i=\ell_1.
\end{array}
\right.
\]
Thus, from \ref{eq:fewdoubled:K}, we have
\begin{align}
e_{K}\left(\cup_{a=0}^{2i-1} (X_{a}\cup Y_{a})\right)
&\leq \left(\frac{\lambda}{\log n}+\lambda\cdot \mathbf{1}_{\{i=\ell_1\}}\right)\cdot d \left|\cup_{a=0}^{2i-1} (X_{a}\cup Y_{a})\right|\nonumber
\\
&\leq \left(\frac{\lambda}{\log n}+\lambda\cdot \mathbf{1}_{\{i=\ell_1\}}\right)\cdot 3\delta^{i}d^{i}r.\label{eq:ddd}
\end{align}
Therefore, we have that
\begin{align}
d|X_{2i-1}|-|X_{2i}|&\leq |E_{\geq 2,K}(X_{2i-1}\cup Y_{2i-1})|+e_{K}\left(\cup_{a=0}^{2i-1}(X_{a}\cup Y_{a})\right)+|W|\nonumber\\
&\hspace{-0.7cm}\overset{\eqref{eq:ccc},\eqref{eq:ddd},\eqref{eq:sizeW}}{\leq}
\left(\frac{\lambda}{\log n}+\lambda\cdot \mathbf{1}_{\{i\in \{1,\ell_1\}\}}\right)\cdot 7\delta^{i}d^{i}r\nonumber \\
&\leq \left(\frac{\lambda}{\log n}+\lambda\cdot \mathbf{1}_{\{i\in \{1,\ell_1\}\}}\right)\cdot 8\cdot d\cdot \sum_{v\in X_{2i-2}}d_{F\setminus K}(v).\label{eqn:thatbound}
\end{align}

Now,
\begin{align*}
\sum_{v\in X_{2i}}d_{F\setminus K}(v)\leq \sum_{u\in X_{2i-2}}\sum_{w\in N_{F\setminus K}(u)}\sum_{v\in N_K(w)}d_{F\setminus K}(v)\overset{\ref{eq:nbrhoodsinFK:pseud}}{\leq} \left(1+\frac{\lambda}{\log n}\right)\cdot \delta d\cdot \sum_{v\in X_{2i-2}}d_{F\setminus K}(v),
\end{align*}
and
\begin{align*}
\sum_{v\in X_{2i}}d_{F\setminus K}(v)&\overset{\ref{eq:degreesinFK:pseud}}{\geq} \left(\sum_{u\in X_{2i-2}}\sum_{w\in N_{F\setminus K}(u)}\sum_{v\in N_K(w)}d_{F\setminus K}(v)\right)-(1+\lambda)\delta d\cdot \left(\sum_{v\in X_{2i-2}}d_{F\setminus K}(v)-|X_{2i-1}|\right)\\
&\hspace{8cm}-(1+\lambda)\delta(d|X_{2i-1}|-|X_{2i}|)\\
&\hspace{-2.25cm}\overset{\ref{eq:nbrhoodsinFK:pseud},\eqref{eqn:thisbound},\eqref{eqn:thatbound}}{\geq} \left(1-\frac{\lambda}{\log n}\right)\cdot \delta d\cdot \sum_{v\in X_{2i-2}}d_{F\setminus K}(v)
-\delta d\cdot \left(\frac{\lambda}{\log n}+\lambda\cdot \mathbf{1}_{\{i\in \{1,\ell_1\}\}}\right)\cdot 19\cdot\sum_{v\in X_{2i-2}}d_{F\setminus K}(v).
\end{align*}
Thus, we have that \eqref{eq:induction:3} holds. Similarly (with \eqref{eqn:thisotherbound} in place of \eqref{eqn:thisbound}), \eqref{eq:induction:4} holds.

Therefore, by induction, \eqref{eq:induction:1}--\eqref{eq:induction:4} hold for each $1\leq i\leq \ell_1$.
Applying \eqref{eq:induction:3} for each $1\leq i\leq \ell_1$, and using that $|\sum_{v\in X_0}d_{F\setminus K}(v)|\geq (1-\lambda)\delta r$ by \ref{eq:degreesinFK:pseud}, we have
\[
|X_{2\ell_1}|\geq \frac{\sum_{v\in X_{2\ell_1}}d_{F\setminus K}(v)}{(1+\lambda)\delta}\geq \frac{(1-40\lambda)\cdot (1-\lambda)d^{\ell_1}\delta^{\ell_1+1}r}{(1+\lambda)\delta}\geq (1-45\lambda)d^{\ell_1}\delta^{\ell_1}r\geq \frac{\lambda n}{10^8},
\]
where we have used \ref{eq:degreesinFK:pseud} again and $\ell_1 = o(\log{n})$.
As, similarly, $|Y_{2\ell_1}|\geq (1-45\lambda)d^{\ell_1}\delta^{\ell_1}r\geq \lambda n/10^8$, we have, by \ref{prop:F-Kconnection}, that
\[
e_{F\setminus K}(X_{2\ell_1},Y_{2\ell_1})\geq (1-\lambda)\frac{\delta}{n}|X_{2\ell_1}||Y_{2\ell_1}|\geq (1-\lambda)\frac{\delta}{n}\cdot \left((1-45\lambda)d^{\ell_1}\delta^{\ell_1}r\right)^2\geq \left(1-\frac{C_0\lambda}{4}\right)\cdot \frac{r^2\delta^{2\ell_1+1} d^{2\ell_1}}{n},
\]
and thus the claim holds.
\claimproofend

Using this claim, we can now connect vertices with paths with length $4\ell_1+5$, as follows.

\begin{claim}\label{clm:lowerbound2} Let $u,v\in [n]$ be distinct and $W\subset [n]\setminus\{x,y\}$ with $|W|\leq \lambda \delta$. Then, there are at least $(1-C_0\lambda/2)\delta^{2\ell_1+3}d^{2\ell_1+2}/n$ $(F\setminus K,K)$-alternating $u,v$-paths of length $4\ell_1+5$ which have no internal vertices in $W$.
\end{claim}
\claimproofstart[Proof of Claim~\ref{clm:lowerbound2}]
Let $U_1=N_{F\setminus K}(u)\setminus (W\cup \{v\}\cup N_{F\setminus K}(v))$ and $V_1=N_{F\setminus K}(v)\setminus (W\cup \{u\}\cup N_{F\setminus K}(u))$. Let $U_2=N_{K}(U_1)\setminus (W\cup \{u,v\}\cup N_{K}(V_1)\cup V_1)$ and $V_2=N_{K}(V_1)\setminus (W\cup \{u,v\}\cup N_{K}(U_1)\cup U_1)$. Using similar arguments to those in the proof of Claim~\ref{clm:lowerbound}, we have $(1-2\lambda)\delta d\leq |U_2|,|V_2|\leq (1+2\lambda)\delta d$.

Let $s=|U_2|$ and $t=|V_2|$, so that $s,t\geq r$, and enumerate $U_2=\{u_1,\ldots,u_s\}$ and $V_2=\{v_1,\ldots,v_t\}$. For each $i\in [s]$,
pick $u'_i\in N_{K}(u_i)\cap U_1$ and for each $j\in [t]$ pick $v_j'\in N_K(v_j)\cap V_1$.
For each $i\in [s]$, let $\tilde{U}_i=\{u'_i,u'_{i+1},\ldots, u'_{i+r-1}\}$ and $\hat{U}_i=\{u_i,u_{i+1},\ldots, u_{i+r-1}\}$, and, for each $j\in [t]$, let $\tilde{V}_j=\{v'_j,v'_{j+1},\ldots, v'_{j+r-1}\}$ and $\hat{V}_j=\{v_j,v_{j+1},\ldots, v_{j+r-1}\}$, with addition in the indices modulo $s$ and $t$ as appropriate. For each $i\in [s]$ and $j\in [t]$, let $W_{i,j}=W\cup \{u,v\}\cup \tilde{U}_i\cup \tilde{V}_j$, so that $|W_{i,j}|\leq 3r +\lambda \delta$. Now, by Claim~\ref{clm:lowerbound}, there are at least $\left(1-\frac{C_0\lambda}{4}\right)\cdot \frac{r^2\delta^{2\ell_1+1} d^{2\ell_1}}{n}$ $(F\setminus K,K)$-alternating $\hat{U}_i,\hat{V}_j$-paths with length $4\ell_1+1$ with no internal vertices in $W_{i,j}$, each of which can be extended to an $(F\setminus K,K)$-alternating $u,v$-path with length $4\ell_1+5$ with no internal vertices in $W$. Therefore, as each such path appears for $r^2$ pairs $(i,j)$, the number of $(F\setminus K,K)$-alternating $u,v$-paths with length $4\ell_1+5$ with no internal vertices in $W$ is at least
\[
\left(1-\frac{C_0\lambda}{4}\right)\cdot \frac{r^2\delta^{2\ell_1+1} d^{2\ell_1}}{n}\cdot st\cdot \frac{1}{r^2}\geq \left(1-\frac{C_0\lambda}{4}\right)(1-2\lambda)^2\cdot \frac{\delta^{2\ell_1+3} d^{2\ell_1+2}}{n}\geq \left(1-\frac{C_0\lambda}{2}\right)\cdot \frac{\delta^{2\ell_1+3} d^{2\ell_1+2}}{n},
\]
as required, where we have used that $s,t\geq (1-2\lambda)\delta d$.
\claimproofend

Finally, we can show that \ref{prop:lowerboundKFKalternating:pseud:newnew} holds. Let $x,y\in [n]$ be distinct and let $Z\subset [n]$ with $|Z|\leq 40\log n$. Let $y'\in N_K(y)\setminus (Z\cup \{x\})$, with $d - |(Z\cup\{x\})\cap N_K(y)|$ choices. For each $i\geq 0$, let $\mathcal{P}_i$ be the set of pairs $(v,P)$ such that $P$ is an $(F\setminus K,K)$-alternating $y',v$-path of length $2i$ with no vertices in $Z\cup \{x\}$ and $|N_{F\setminus K}(v) \cap (Z\cup V(P)\cup \{x,y\})|\leq \left(\mathbf{1}_{\{i=0\}}+1/\log{n}\right)\lambda \delta$. Note that $|\mathcal{P}_0|=1$ as $|Z\cup \{x,y\}|\leq \lambda \delta$.

\begin{claim}\label{clm:annoying} For each $1\leq i\leq 10\log n$,
\[
\sum_{(v,P)\in \mathcal{P}_i}d_{F\setminus K}(v) \geq \left(1-4\lambda \cdot \mathbf{1}_{\{i=1\}}-\frac{21\lambda}{\log{n}}\right)\delta d \cdot\sum_{(v,P)\in \mathcal{P}_{i-1}}d_{F\setminus K}(v).
\]
\end{claim}
\claimproofstart[Proof of Claim~\ref{clm:annoying}] Let $(v,P)\in \mathcal{P}_{i-1}$. Let $Z^+=V(P)\cup Z\cup \{x,y\}$, so that $|Z^+|\leq 100\log n$.
Let $B$ be the set of $w\in [n]\setminus Z^+$ with $|N_{F\setminus K}(w) \cap Z^+|\geq 10$ so that, by \ref{prop:new}, we have $|B|\leq 100\log n$.


For each $(u,v')$ with $u\in N_{F\setminus K}(v)\setminus Z^+$ and $v'\in N_K(u)\setminus (Z^+\cup B)$, there is a path $P'$ such that $P\subset P'$ and $(v',P')\in \mathcal{P}_{i}$ by definition of $B$.
Therefore,
\begin{align*}
\sum_{u\in N_{F\setminus K}(v)}\sum_{v'\in N_K(u)}d_{F\setminus K}(v')~~-&\sum_{(v',P')\in \mathcal{P}_i:P\subset P'}d_{F\setminus K}(v')\\
&\overset{\ref{eq:degreesinFK:pseud}}{\le} \sum_{u\in N_{F\setminus K}(v)\cap Z^+}|N_K(u)|\cdot 2\delta+\sum_{u\in N_{F\setminus K}(v)}\,\sum_{v'\in N_K(u)\cap (Z^+\cup B)}2\delta\\
&\leq 2\delta d\cdot |N_{F\setminus K}(v)\cap Z^+|+(|E_{\geq 2,K}(N_{F\setminus K}(v))|+|Z^+\cup B|)\cdot 2\delta\\
&\overset{\ref{eq:fewdoubled:K}}{\leq} 2d\delta \cdot\left(\mathbf{1}_{\{i=1\}}+1/\log{n}\right)\lambda \delta
+\frac{\lambda d\cdot 2\delta}{\log n}\cdot 2\delta+200\log n\cdot 2\delta\\
&\leq \left(2\cdot \mathbf{1}_{\{i=1\}}+10/\log{n}\right)\cdot \lambda \delta^2d.
\end{align*}
Now, we have $\sum_{u\in N_{F\setminus K}(v)}\sum_{v'\in N_K(u)}d_{F\setminus K}(v')\geq (1-\lambda/\log n)\delta d \cdot d_{F\setminus K}(v)$ by \ref{eq:nbrhoodsinFK:pseud} and $d_{F\setminus K}(v)\geq \delta/2$ by \ref{eq:degreesinFK:pseud}, so it follows that
\[
\sum_{(v',P')\in \mathcal{P}_i:P\subset P'}d_{F\setminus K}(v')\geq  \left(1-4\lambda \cdot \mathbf{1}_{\{i=1\}}-21\lambda \cdot \log^{-1} n\right)\delta d \cdot d_{F\setminus K}(v).
\]

Therefore,
\begin{align*}
\sum_{(v,P)\in \mathcal{P}_i}d_{F\setminus K}(v)&\geq \sum_{(v,P)\in \mathcal{P}_{i-1}}\sum_{(v',P')\in \mathcal{P}_i:P\subset P'}d_{F\setminus K}(v') \\
&\geq \left(1-4\lambda \cdot \mathbf{1}_{\{i=1\}}-21\lambda \cdot \log^{-1} n\right)\delta d \cdot\sum_{(v,P)\in \mathcal{P}_{i-1}}d_{F\setminus K}(v).
\end{align*}
as claimed.
\claimproofend

Now, let $\ell$ satisfy $10\ell_0\leq \ell\leq 10\log n$. Let $\ell_2=\ell-(2\ell_1+3)\geq 0$. Let $(v,P)\in \mathcal{P}_{\ell_2}$, so that, by Claim~\ref{clm:annoying} applied for each $i\in [\ell_2]$ and \ref{eq:degreesinFK:pseud}, the number of choices for $(v,P)$ is at least
\begin{align*}
\frac{\sum_{(v,P)\in \mathcal{P}_{\ell_2}}d_{F\setminus K}(v)}{(1+\lambda)\delta}&\geq \left(1-4\lambda-\frac{21\lambda}{\log{n}}\right)\cdot \left(1-\frac{21\lambda}{\log{n}}\right)^{\ell_2-1}\delta^{\ell_2}d^{\ell_2}\cdot \frac{(1-\lambda)\delta}{(1+\lambda)\delta}\\
&\geq  \left(1-\frac{C_0\lambda}{4}\right)\delta^{\ell_2}d^{\ell_2}.
\end{align*}
Then, let $P'$ be an $(F\setminus K, K)$-alternating $x,v$-path of length $4\ell_1+5$ with no internal vertices in $Z\cup V(P)\cup \{y\}$, with at least $(1-C_0\lambda/2)\delta^{2\ell_1+3}d^{2\ell_1+2}/n$ options by Claim~\ref{clm:lowerbound2}. Then, $(P\cup P')+y'y$ is an $(F\setminus K,K)$-alternating $x,y$-path with length $2\ell$. Combining our choices for $y',P$ and $P'$, the number of choices is at least
\begin{align*}
(d - |(Z\cup\{x\})\cap N_K(y)|)\left(1-\frac{C_0\lambda}{4}\right)\delta^{\ell_2}d^{\ell_2}&\cdot \left(1-\frac{C_0\lambda}{2}\right)\frac{\delta^{2\ell_1+3}d^{2\ell_1+2}}{n}
\\
&\geq \left(1-C_0\lambda-\frac{|(Z\cup\{x\})\cap N_K(y)|}{d}\right)\frac{\delta^{\ell}d^{\ell}}{n},
\end{align*}
as required. Thus, \ref{prop:lowerboundKFKalternating:pseud:newnew} holds.
\end{proof}


\subsection{Likely properties of $G_d(n)$: expansion}\label{sec:Kexpansion}
The following two lemmas, Lemma~\ref{lem:fewbadgraphsexpander} and Lemma~\ref{lem:fewbadgraphsexpandertwo} will be applied with two different values of $\lambda$ ($\lambda$ and $\lambda /\log n$), which motivates why these Lemmas use the bounds $\lambda \geq \frac{1}{64}\log^{-4}n$ and $\lambda \delta d\geq C\log n$. The two lemmas give a likely bound for, respectively, $e_K(U)$ and $|E_{\geq 2,K}(U)|$ with $K\sim G_d(n)$, so that in combination they will be used to prove \ref{eq:fewdoubled:K}. Both of these lemmas are proved using switching techniques. For convenience, after their proofs we will combine them into Corollary~\ref{cor:fewbadgraphsexpander}.

\begin{lemma}\label{lem:fewbadgraphsexpander}
Let $\frac{1}{64}\log^{-4}n\leq \lambda \ll 1$  and $1/C\ll 1$, and let $d$ satisfy $d\geq C\log n$ and $d=o(n)$. Let $\delta\geq C\log n$ satisfy $\lambda \delta d\geq C\log n$, and let $q=2\delta/n$. Let $K\sim G_d(n)$ and $F\sim G(n,q)$. Then, with probability $1-o(n^{-14})$, the following holds.
\stepcounter{propcounter}
\begin{enumerate}[label = {\emph{\textbf{\Alph{propcounter}}}}]
\item For each $U\subset [n]$, if $|U|\leq \frac{\lambda n}{10^4}$ and there is some $W\subset [n]$ with $|W|\leq 4|U|/\delta$ and $U\subset W\cup N_F(W)$, then $e_{K}(U)\leq \frac{\lambda d}{2} |U|$.\labelinthm{prop:expand:only}
\end{enumerate}
\end{lemma}
\begin{proof} By Chernoff's bound and a union bound, with probability $1-o(n^{-14})$, we have $\Delta(F)\leq 3\delta$. Fixing $F$ with this property, we now show the following holds with probability $1-o(n^{-14})$.
\begin{enumerate}[label = {{\textbf{\Alph{propcounter}$'$}}}]
\item For each $W\subset [n]$, if $|W|\leq \frac{\lambda n}{10^3\delta}$, then $e_{K}(W\cup N_F(W))\leq \frac{\lambda d}{8}\cdot \delta  |W|$.\label{prop:expand:only:fixedF}
\end{enumerate}
This will complete the proof of the lemma as \ref{prop:expand:only:fixedF} implies \ref{prop:expand:only}. Indeed, when it holds, if $U\subset [n]$ and $W\subset [n]$ satisfy $|U|\leq \frac{\lambda n}{10^4}$, $|W|\leq 4|U|/\delta$ and $U\subset W\cup N_F(W)$, then $|W|\leq \frac{\lambda n}{10^3\delta}$, so that, from \ref{prop:expand:only:fixedF} we have
\[
e_K(U)\leq e_{K}(W\cup N_F(W))\leq \frac{\lambda d}{8}\cdot \delta  |W|\leq \frac{\lambda d}{2} |U|,
\]
and thus \ref{prop:expand:only} holds.

Let then $W\subset [n]$ with  $|W|\leq \lambda n/10^3\delta$, and set $m=|W|$. Let $W'=W\cup N_F(W)$, so that $|W'|\leq 4\delta m$.
For each $i\in \N_0$, let $\mathcal{K}^i$ be the set of $K\in \cK_d(n)$ for which $e(K[W'])=i$. 
We will show the following.

\begin{claim}\label{clm:fewbadgraphsexpander}
For each $i\geq \lambda d\delta m/16$, $|\cK^{i}|\leq |\cK^{i-1}|/e$.
\end{claim}
\claimproofstart[Proof of Claim~\ref{clm:fewbadgraphsexpander}]
Let $L_i$ be an auxiliary bipartite graph with vertex classes $\mathcal{K}^{i}$ and $\cK^{i-1}$, and put an edge between $K\in \cK^{i}$ and $K'\in \cK^{i-1}$ if $K\triangle K'$ is a cycle with length 6 with exactly two vertices in $W'=W \cup N_F(W)$, which moreover have an edge between them in $K$.

Firstly, let $K\in \cK^{i}$. Pick $v_1,v_2$ with $v_1v_2\in E(K[W'])$, with $2i$ possibilities as $K\in \mathcal{K}^{i}$.
Then, pick $v_3,v_4$ with  $v_3v_4\in E(K-W'-N_K(v_2))$, with, as $|W'|\leq 4\delta m$, at least $2(dn/2-d\cdot(4\delta m+d))\geq 3dn/4$ possibilities. Then, pick $v_5,v_6$ with $v_5,v_6\in E(K-W'-\{v_3,v_4\}-N_K(v_4)-N_K(v_1))$, with at least $2(dn/2-d\cdot(4\delta m+2+2d))\geq 3dn/4$ possibilities.
Then, noting that $K-\{v_1v_2,v_3v_4,v_5v_6\}+\{v_2v_3,v_4v_5,v_6v_1\}\in \cK^{i-1}$ is a neighbour of $K$ in $L_i$, we have
\begin{equation}\label{eqn:Uedges1upperexpansion}
d_{L_i}(K)\geq \frac{1}{2}\cdot 2i\cdot \left(\frac{3dn}{4}\right)^2\geq \frac{1}{2}id^2n^2.
\end{equation}

Then, let $K'\in \cK^{i-1}$. Pick $v_2,v_3$ with $v_2\in W'$ and $v_2v_3\in E(K')$, with at most $4\delta md$ possibilities. Similarly, pick $v_1,v_6$ with $v_1\in W'$ and $v_1v_6\in E(K')$, with at most $4\delta md$ possibilities. Then, pick $v_4$ and $v_5$ with $v_4v_5\in E(K')$, with at most $dn$ possibilities. Therefore, in total, we have $d_{L_i}(K')\leq 16d^3\delta^2m^2n/2$.

Therefore, from this and \eqref{eqn:Uedges1upperexpansion}, double-counting the edges of $L_i$ we get
\[
\frac{id^2n^2}{2} \cdot |\cK^{i}|\leq e(L_i)\leq \frac{16d^3\delta^2m^2n}{2}\cdot |\cK^{i-1}|,
\]
so that, as $i\geq \lambda d\delta m/16$ and $m\leq \lambda n/10^3\delta$,
\[
\frac{|\cK^{i}|}{|\cK^{i-1}|}\leq \frac{16d\delta^2m^2}{in}\leq \frac{256\delta m}{\lambda n}\leq \frac{1}{e},
\]
as required.
\claimproofend

Therefore, applying the claim repeatedly, we have, for each $i\geq \lambda d\delta m/8$,
\[
{|\mathcal{K}^i|}{}\leq \exp\left(-\frac{\lambda d \delta m}{16}\right)\cdot |\mathcal{K}_d(n)| \leq  \exp\left(-\frac{Cm\log n}{16}\right)\cdot |\mathcal{K}_d(n)|,
\]
where we have used that $\lambda d \delta\geq C\log n$.
In particular, 
\[
\sum_{i = \lambda d\delta m/8}^{n^2} \frac{|\mathcal{K}^i|}{|\mathcal{K}_d(n)|} \le n^2  \exp\left(-\frac{Cm\log n}{16}\right).
\]

Now, as there were at most $\binom{n}{m}\leq n^m$ choices for $W$ with $|W|=m$, we have that \ref{prop:expand:only:fixedF} does not hold with probability at most
\[
\sum_{m=1}^{\lambda n/10^3\delta}n^m\cdot n^2\exp\left(-\frac{Cm\log n}{16}\right)=o(n^{-14}),
\]
as required, using that $1/C\ll 1$.
\end{proof}


Similarly to Lemma~\ref{lem:fewbadgraphsexpander}, we now prove the following lemma which shows a likely bound on $|E_{\geq 2,K}(U)|$ for many sets $U$ in $K\sim G_d(n)$ using switching techniques.

\begin{lemma}\label{lem:fewbadgraphsexpandertwo}
Let $\log^{-4}n\leq \lambda \ll 1$ and $1/C\ll 1$. Let $d\geq C\log n$ satisfy $d=o(n)$. Let $\delta$ satisfy $\lambda \delta d\geq C\log n$, and let $q=2\delta/n$. Let $K\sim G_d(n)$ and $F\sim G(n,q)$. Then, with probability $1-o(n^{-14})$, the following holds.
\stepcounter{propcounter}
\begin{enumerate}[label = {\emph{\textbf{\Alph{propcounter}}}}]
\item For each $U\subset [n]$, if $|U|\leq \frac{\lambda n}{10^6d}$ and there is some $W\subset [n]$ with $|W|\leq 4|U|/\delta$ and $U\subset W\cup N_F(W)$, then $|E_{\geq 2,K}(U)|\leq \frac{\lambda d}{2} |U|$.\labelinthm{prop:NKtwice2}\labelinthm{prop:expand:only:EK2}
\end{enumerate}
\end{lemma}
\begin{proof} By Chernoff's bound and a union bound, with probability $1-o(n^{-14})$, we have $\Delta(F)\leq 3\delta$. Fixing $F$ with this property, we now show the following holds with probability $1-o(n^{-14})$.
\begin{enumerate}[label = {{\textbf{\Alph{propcounter}$'$}}}]
\item For each $W\subset [n]$, if $|W|\leq \frac{\lambda n}{2\cdot 10^3\delta d}$, then $|E_{\geq 2,K}(W\cup N_F(W))|\leq \frac{\lambda d}{16}\cdot \delta  |W|$.\label{prop:expand:only:fixedF:EK2}
\end{enumerate}
This will easily complete the proof of the lemma as, by Lemma~\ref{lem:fewbadgraphsexpander}, we have that, with probability $1-o(n^{-14})$, \ref{prop:expand:only} holds with $\lambda/64$ in place of $\lambda$. Then, if \ref{prop:expand:only:fixedF:EK2} holds, for each $U\subset [n]$ and $W\subset [n]$ satisfying $|U|\leq \frac{\lambda n}{10^6d}$, $|W|\leq 4|U|/\delta$ and $U\subset W\cup N_F(W)$, we have $|W\cup N_F(W)|\leq 4\delta |W|\leq 16|U|$ and $|W|\leq \frac{\lambda n}{2\cdot 10^3\delta d}$, so that, from \ref{prop:expand:only:fixedF:EK2} we have
\[
|E_{\geq 2,K}(U)|\leq |E_{\geq 2,K}(W\cup N_F(W))|+e_K(W\cup N_F(W))\leq \frac{\lambda d}{16}\cdot \delta  |W|+\frac{\lambda d}{64}\cdot |W\cup N_F(W)|\leq \frac{\lambda d}{2} |U|,
\]
and thus \ref{prop:expand:only:EK2} holds.

Let $W\subset [n]$ with  $|W|\leq \lambda n/(2\cdot 10^3\delta d)$, and set $m=|W|$. Let $W'=W\cup N_F(W)$, so that $|W'|\leq 4\delta m$.
For each $i\in \N_0$, let $\mathcal{K}^i$ be the set of $K\in \cK_d(n)$ for which
\begin{equation}\label{eq:replaceE2K}
\sum_{v\in [n]\setminus W'}\max\{d_K(v,W')-1,0\}=i.
\end{equation}
We study this rather than $|E_{\geq 2,K}(W')|$, as removing an edge $vv'$ between $v\in [n]\setminus W'$ and $v'\in W'$ when $v$ has exactly 2 neighbours in $W'$ will reduce $|E_{\geq 2,K}(W')|$ by 2 (and 1 if $v$ has more than 2 neighbours in $W'$), but the quantity at \eqref{eq:replaceE2K} will always reduce by 1 when $v$ has at least 2 neighbours in $W'$. We will then use the simple bound that
\begin{equation}\label{eq:relatebacktoE2K}
\sum_{v\in [n]\setminus W'}\max\{d_K(v,W')-1,0\}\leq |E_{\geq 2,K}(W')|\leq 2\sum_{v\in [n]\setminus W'}\max\{d_K(v,W')-1,0\}.
\end{equation}

Using switching methods, we now show the following.

\begin{claim}\label{clm:fewbadgraphsexpanderthisone}
For each $i\geq \lambda d\delta m/32$, $|\cK^{i}|\leq |\cK^{i-1}|/\sqrt{e}$.
\end{claim}
\claimproofstart[Proof of Claim~\ref{clm:fewbadgraphsexpanderthisone}] 
Let $i\geq \lambda d\delta m/32$. Let $L_i$ be an auxiliary bipartite graph with vertex classes $\mathcal{K}^{i}$ and $\cK^{i-1}$, and put an edge between $K\in \cK^{i}$ and $K'\in \cK^{i-1}$ if $K\triangle K'$ is a cycle with length 6 with exactly 1 vertex in $W'$.

Firstly, let $K\in \cK^{i}$. Pick $v_1,v_2$ with $v_1\in W'$ and $v_1v_2\in E_{\geq 2,K}(W')$, with at least $i$ possibilities as $K\in \mathcal{K}^{i}$.
Then, pick $v_3,v_4$ with $v_3v_4\in E(K-W'-N_K(W')-N_K(v_2))$, with at least $2 (dn/2-d\cdot(4\delta m+4\delta m\cdot d+d))\geq 3dn/4$ possibilities, where we have used that $\delta dm\leq \lambda n/(2\cdot 10^3)$. Then, pick $v_5,v_6$ with $v_5,v_6\in E(K-W'-N_K(W')-\{v_3,v_4\}-N_K(v_4))$, with at least $2(dn/2-d\cdot(4\delta m+4\delta m\cdot d+2+d))\geq 3dn/4$ possibilities.
Note that $K-\{v_1v_2,v_3v_4,v_5v_6\}+\{v_2v_3,v_4v_5,v_6v_1\}\in \cK^{i-1}$ is a neighbour of $K$ in $L_i$, as $v_1v_2 \in E_{\geq 2,K}(W')$ and $v_3,v_4,v_5,v_6 \not\in N_K(W')$. Thus we have
\begin{equation}\label{eqn:Uedges1upperexpansion2}
d_{L_i}(K)\geq i\cdot \left(\frac{3dn}{4}\right)^2\geq \frac{1}{2}id^2n^2.
\end{equation}

Then, let $K'\in \cK^{i-1}$. To switch from $K'$ we need to increase $\sum_{v\in [n]\setminus W'}\max\{d_K(v,W')-1,0\}$ by one. If $v_1$ is the vertex of the cycle in $W'$ and $v_1v_2 \not\in E(K')$, it must be the case that $v_2\in N_{K'}(W')$. This is so that $v_2$ already has another neighbour from $W'$ in $K$ other than $v_1$ and so adding $v_1v_2$ to $K'$ will increase $\sum_{v\in [n]\setminus W'}\max\{d_K(v,W')-1,0\}$ by one. Thus we must pick $v_1,v_2$ with $v_1\in W'$ and $v_2\in N_{K'}(W')$, with at most $4\delta m\cdot 4\delta md$ possibilities. Then, pick $v_3,v_6$ with $v_2v_3,v_1v_6\in E(K')$, with at most $d^2$ possibilities. Then, pick $v_4$ and $v_5$ with $v_4v_5\in E(K')$, with at most $dn$ possibilities. Therefore, in total, we have $d_{L_i}(K')\leq 16d^4\delta^2m^2n$.

Therefore, from this and \eqref{eqn:Uedges1upperexpansion2}, double-counting the edges of $L_i$ we get
\[
\frac{id^2n^2}{2} \cdot |\cK^{i}|\leq e(L_i)\leq 16d^4\delta^2m^2n\cdot |\cK^{i-1}|,
\]
so that, as $i\geq \lambda d\delta m/32$ and $m\leq \lambda n/(2\cdot 10^3\delta d)$,
\[
\frac{|\cK^{i}|}{|\cK^{i-1}|}\leq \frac{32d^2\delta^2m^2}{in}\leq \frac{32\cdot 32d\delta m}{\lambda n}\leq \frac{1}{\sqrt{e}},
\]
as required.
\claimproofend

Therefore, applying the claim repeatedly, we have, for each $i\geq \lambda d\delta m/16$,
\[
\frac{|\mathcal{K}^i|}{|\mathcal{K}_d(n)|}\leq \exp\left(-\frac{\lambda d \delta m}{64}\right) \leq \exp\left(-\frac{C m\log n}{64}\right),
\]
where we have used that $\lambda d \delta\geq C\log n$. 
In particular, 
\[
\sum_{i = \lambda d\delta m/16}^{n^2} \frac{|\mathcal{K}^i|}{|\mathcal{K}_d(n)|} \le n^2  \exp\left(-\frac{Cm\log n}{64}\right).
\]
Thus, using \eqref{eq:relatebacktoE2K}, with probability at least $1-n^2\exp\left(-\frac{C m\log n}{64}\right)$, we have that $|E_{\geq 2,K}(N_K(W))|\leq 2\cdot \lambda d\delta m/16$.

Now, as there were at most $\binom{n}{m}\leq n^m$ choices for $W$ with $|W|=m$, we have that \ref{prop:expand:only:fixedF:EK2} does not hold with probability at most
\[
\sum_{m=1}^{\lambda n/2\cdot10^3\delta d}n^m\cdot n^2\exp\left(-\frac{C m\log n}{64}\right)=o(n^{-14}),
\]
as required.
\end{proof}

In combination, Lemmas~\ref{lem:fewbadgraphsexpander} and~\ref{lem:fewbadgraphsexpandertwo} imply the following.
\begin{corollary}\label{cor:fewbadgraphsexpander}
Let $\log^{-4}n\leq \lambda \ll 1$ and $1/C\ll 1$. Let $d\geq C\log n$ satisfy $d=o(n)$. Let $m\in \N$ and suppose $\delta=2m/n$ satisfies $\lambda \delta d\geq C\log n$. 
Let $p=m(\binom{n}{2}-dn/2)^{-1}$. Let $K\sim G_d(n)$, form $E$ by including each element of $[n]^{(2)}\setminus E(K)$ uniformly at random with probability $p$, and let $F=K+E$. Then, with probability $1-o(n^{-14})$, the following holds.
\stepcounter{propcounter}
\begin{enumerate}[label = {\emph{\textbf{\Alph{propcounter}}}}]
\item For each $U\subset [n]$, if $|U|\leq \frac{\lambda n}{10^6d}$ and there is some $W\subset [n]$ with $|W|\leq 4|U|/\delta$ and $U\subset W\cup N_F(W)$, then $e_{K}(U)+|E_{\geq 2,K}(U)|\leq \lambda d |U|$.\labelinthm{prop:expand:only:corollary}
\end{enumerate}
\end{corollary}
\begin{proof}
Let $q=2\delta/n$. Independently, select $K\sim G_d(n)$ and $F'\sim G(n,q)$. When $|E(F')\setminus E(K)|\geq m$, then let $E$ be a uniformly random set of $m$ edges from $E(F')\setminus E(K)$, and otherwise let $E$ be a uniformly random set of $m$ edges from $[n]^{(2)}\setminus E(K)$. Then, by symmetry, we have $F\sim F(n,d,m)$.

Now, as $d=o(n)$, we have $\E|E(F')\setminus E(K)|\geq q(\binom{n}{2}-dn/2)\geq 1.1\delta n/2$. By a simple application of a Chernoff bound, then, we have that, with probability $1 - o(n^{-14})$, $|E(F')\setminus E(K)|\geq m$, and when this happens $F\subset F'$.
Furthermore, if $F\subset F'$ and \ref{prop:expand:only:corollary} holds for $F'$ in place of $F$, then \ref{prop:expand:only:corollary} holds with $F$. Thus, the corollary follows from Lemmas~\ref{lem:fewbadgraphsexpander} and~\ref{lem:fewbadgraphsexpandertwo}.
\end{proof}


\subsection{Likely properties of $F\setminus K$: expansion}\label{sec:F-Kexpansion}
For \ref{eq:fewdoubled:F-K}, similarly to our work in Section~\ref{sec:Kexpansion}, we again split this into Lemma~\ref{lem:gnpexpander} and Lemma~\ref{lem:gnpexpander2}, where the first of these holds more easily. We then show a result that implies that \ref{prop:new} is likely to hold. 

\begin{lemma}\label{lem:gnpexpander}
Let $1/n\ll 1/C\ll  \eps\leq 1$.
Let $d$ and $\lambda$ satisfy $\log^{-3}n\leq \lambda \ll 1$, $10^3d\leq \lambda n$, and $d\geq C \log n$. Let $m\in \N$ satisfy $\eps dn/2 \leq m\leq n^{1+\sigma}$. Let $\delta=2m/n$ and suppose that $\lambda \delta\geq C \log n$.
Let $p=m(\binom{n}{2}-dn/2)^{-1}$. 
Let $K\sim G_d(n)$, form $E$ by including each element of $[n]^{(2)}\setminus E(K)$ uniformly at random with probability $p$, and let $F=K+E$. 

Then, with probability $1-o(n^{-14})$, the following holds.
\stepcounter{propcounter}
\begin{enumerate}[label = {\emph{\textbf{\Alph{propcounter}}}}]
\item For each $U\subset [n]$ with $|U|\leq \lambda n/100\delta \log n$, if there is some $U'\subset [n]$ with $|U'|\leq 4|U|/d$ and $U\subset U'\cup N_{K}(U')$, then \labelinthm{eq:fewdoubled:F-K:bothcases}
\[
e_{F\setminus K}(U)\leq 
\frac{\lambda}{2\log n}\cdot \delta |U| 
\]
\end{enumerate}
\end{lemma}
\begin{proof} 
Let $U\subset [n]$ with $|U|\leq \lambda n/100\delta\log n$ and let $s = |U|$. We will show that, with probability $1-\exp\left(-\omega(s)\right)$, $U$ satisfies the property we want, and then take a union bound over all sets $U$ we must consider.


Let $X=e(F\setminus K[U])$ and let $M = \left\lceil\frac{\lambda \delta s}{2\log{n}}\right\rceil$ so that  again, we have
\begin{align*}
\P\left(X\geq \frac{\lambda \delta s}{2\log n}\right)
\leq \sum_{i=M}^{\binom{s}{2}} \binom{\binom{s}{2}}{i}\cdot p^{i}
\le  \sum_{i=M}^{s^2} \left(\frac{es^2p}{2i}\right)^{i} 
\le \sum_{i=M}^{s^2} \left(\frac{es^2p}{2M}\right)^{i} 
\end{align*}
As 
\[\frac{es^2p}{2M} \le \frac{esp\cdot \log{n}}{\lambda \delta} \le \frac{2es\cdot \log{n}}{\lambda n} \le \frac{2e}{100\delta}\]
and   $\lambda \delta \ge C \log{n}$ , we thus obtain that 
\begin{align*}
\P\left(X\geq \frac{\lambda \delta s}{2\log n}\right)
\le \sum_{i=M}^{s^2} \left( \frac{2e}{100\delta}\right)^{i} \le 2\left( \frac{2e}{100\delta}\right)^{\frac{\lambda \delta s}{2\log{n}}} \leq 2\left( \frac{2e}{100\delta}\right)^{\frac{C s}{2}} = \exp{(-\omega(s))},
 \end{align*}
 where we have used that $\delta \ge C \eps \log{n}$.
Let us now count the number of sets $U\subset [n]$ with $|U|=s$ for which there is some $U'\subset [n]$ with $|U'|\leq 4s/d$ and $U\subset U'\cup N_K(U')$. Note that the condition $|U'| \le 4|U|/d$ implies that if there is some such $U$ then we must have $1\leq 4s/d$, which in turn implies that $s\geq d/4\geq C\log n/4$.
Since $|U' \cup N_K(U')|\leq (d+1)|U'|$, the number of possible choices for such a $U \subset U' \cup N_K(U')$ is at most
\[\sum_{i=1}^{4s/d}\binom{n}{i}\cdot 2^{(d+1)i}\leq \sum_{i=1}^{4s/d}\left(n2^{(d+1)}\right)^i\le 2\left(n2^{(d+1)}\right)^{4s/d}\leq 2^{5s}\]
as $d\geq C\log n$. 
Thus, the probability that \ref{eq:fewdoubled:F-K:bothcases} does not hold is at most
\begin{align*}
\sum_{s=C\log n/4}^n2^{5s}\cdot \exp\left(-\omega(s)\right)=o(n^{-14}),
\end{align*}
as required.
\end{proof}

We now give our likely bound on $|E_{\geq 2,F\setminus K}|$, similarly to the bound proved in Lemma~\ref{lem:gnpexpander}.


\begin{lemma}\label{lem:gnpexpander2}
Let $1/n\ll 1/C\ll \eps,\sigma \ll 1$.
Let $d$ and $\lambda$ satisfy $\log^{-3}n\leq \lambda \ll 1$, $10^3d\leq \lambda n$, and $d\geq C\log n$. Let $m\in \N$ satisfy $\eps dn/2\leq m\leq n^{1+\sigma}$. Let $\delta=2m/n$ and suppose that $\lambda \delta\geq C \log n$.
Let $p=m(\binom{n}{2}-dn/2)^{-1}$. 
Let $K\sim G_d(n)$, form $E$ by including each element of $[n]^{(2)}\setminus E(K)$ uniformly at random with probability $p$, and let $F=K+E$. 
Then, with probability $1-o(n^{-14})$, the following holds.
\stepcounter{propcounter}
\begin{enumerate}[label = {\emph{\textbf{\Alph{propcounter}}}}]
\item For each $U\subset [n]$ with $|U|\leq \lambda n/100\delta\log n$, if there is some $U'\subset [n]$ with $|U'|\leq 4|U|/d$ and $U\subset U'\cup N_{K}(U')$, then \labelinthm{eq:fewdoubled:F-K:bothcases2}
\[
|E_{\geq 2,F\setminus K}(U)|\leq 
\frac{\lambda}{2\log n}\cdot \delta |U|.
\]
\end{enumerate}
\end{lemma}
\begin{proof}
Let $U\subset [n]$ with $|U|\leq \lambda n/100\delta\log n$, and set $s=|U|$. We will show that, with probability $1-\exp\left(-\Omega(s)\right)$, $U$ satisfies the property we want, so that the result follows by a union bound over all such $U$ for which there is some $U'\subset [n]$ with $|U'|\leq 4|U|/d$ and $U\subset U'\cup N_{K}(U')$, as at the end of Lemma~\ref{lem:gnpexpander}.

Consider a process in which we reveal edges in $E$ by taking each $v\in [n]\setminus U$ and revealing potential edges in $E$ into $U$ until a neighbour of $v$ in $U$ has been revealed or it has been revealed that $v$ has no neighbours in $U$. Let $V$ be the set of vertices $v$ after this which have a neighbour in $U$. Then, reveal the rest of the edges between $V$ and $U$ that have not yet been considered, and let $X$ be the number of edges between $V$ and $U$ revealed here. Note that $|E_{\geq 2,F\setminus K}(U)|\leq 2X$.

Now, let $\mathcal{E}$ be the event that $|V|\leq 3\delta s$. As $V=N_{F\setminus K}(U)$, by a simple application of Chernoff's bound, with probability $1-\exp(-\omega(s))$, $\mathcal{E}$ holds. Let $Y=X$ if $\mathcal{E}$ holds, and $Y=0$ otherwise. Thus, with probability $1-\exp(-\omega(s))$, $|E_{\geq 2,F\setminus K}(U)|\leq 2Y$.

For $M=\left\lceil\frac{\lambda \delta s}{4\log n}\right\rceil=\Omega(s)$, we have, then
\begin{align*}
\P\left(Y\geq \frac{\lambda \delta s}{4\log n}\right)&\leq \sum_{i = M}^{|U||V|}\binom{|U| |V|}{i}p^i
\leq  \sum_{i = M}^{3\delta s^2}\binom{s \cdot 3\delta s}{i}p^i
\le \sum_{i = M}^{3\delta s^2}\left(\frac{3e\delta s^2p}{i}\right)^i
\le \sum_{i = M}^{3\delta s^2}\left(\frac{3e\delta s^2p}{M}\right)^i
\end{align*}
As \[
\frac{3e\delta s^2p}{M} \le\frac{3e sp\cdot 4\log{n}}{\lambda } \le \frac{24es\delta\log{n}}{\lambda n} \le \frac{24e}{100} \le e^{-1/3}
\]
we see that 
\begin{align*}
\P\left(Y\geq \frac{\lambda \delta s}{4\log n}\right)
\le \sum_{i = M}^{3\delta s^2} e^{-i/3} \le 4e^{-\frac{\lambda \delta s}{12\log n}} \le \exp{(-C s /20)}.
\end{align*}
Thus $|E_{\geq 2,F'}(U)|\leq 2Y$ with probability $1-\exp{(-C s /20)}$. Similarly to the end of the proof of Lemma~\ref{lem:gnpexpander}, we know that there are at most $2^{5s}$ sets $U$ such that  $|U| = s$, there is some set $U'$ with $|U'| \le 4s/d$ and $U \subset U' \cup N_K(U')$; and this is only possible is $s \ge d/4 \ge C \log{n}/4$. Thus the probability that \ref{eq:fewdoubled:F-K:bothcases2} does not hold is at most
\[
\sum_{s =C \log{n}/4 }^{n}2^{5s}\cdot \exp{(-C s /20)} = o(n^{-14})
\]
as required.
\end{proof}

We now prove our result that will imply that \ref{prop:new} is likely to hold.

\begin{lemma}\label{lem:newforpropnew}
Let $1/n\ll 1/C\ll \eps,\sigma \ll 1$. Let $0<p\leq n^{\sigma-1}$ and let $H\sim G(n,p)$.
Then, with probability $1-o(n^{-14})$, for each $U\subset [n]$ with $|U|\leq 100\log n$, $|\{v\in [n]\setminus U: |N_{H}(v) \cap U|\geq 10\}|\leq 100\log n$.
\end{lemma}
\begin{proof} Let $U\subset [n]$ satisfy $|U|\leq 100\log n$. For each $v\in [n]\setminus U$, the probability that $|N_H(v) \cap U|\geq 10$ is at most $\binom{|U|}{10}p^{10} \le |U|^{10}p^{10}\leq n^{-9}$. Then, the probability that  $|\{v\in [n]\setminus U:d_{H}(v)\geq 10\}|> 100\log n$ is, setting $k=100\log n$, at most $\sum_{i=k}^n\binom{n}{i}n^{-9i} \le \sum_{i=k}^n n^in^{-9i} \le n\cdot n^{-8k}$. Taking a union bound over all sets $U\subset [n]$ satisfying $|U|\leq 100\log n$, the property in the lemma does not hold with probability at most $2n^k\cdot  n^{1-8k}=o(n^{-14})$, as required.
\end{proof}


\subsection{Connection properties of $F\setminus K$}
\label{sec:connection}

We now prove that \ref{prop:F-Kconnection} is likely to hold, in the following lemma.

\begin{lemma} \label{lem:F-Kconnection}
Let $1/n\ll 1/C\ll \eps,\sigma\ll 1$.
Let $d,\lambda$ and $\mu$ satisfy $\log^{-3}n\leq \lambda,\mu \ll 1$ and $d\geq C\log n$. Let $m\in \N$ satisfy $\eps dn\leq m\leq n^{1+\sigma}$. Let $\delta=2m/n$ and suppose that $\lambda^2\mu \delta\geq C\log n$.
Let $p=m(\binom{n}{2}-dn/2)^{-1}$. 
Let $K\sim G_d(n)$, form $E$ by including each element of $[n]^{(2)}\setminus E(K)$ uniformly at random with probability $p$, and let $F=K+E$. 
Then, with probability $1-o(n^{-14})$, for each disjoint sets $U,V\subset [n]$ with $|U|,|V|\geq \mu n$, we have $e_{F \setminus K}(U,V)=(1\pm \lambda)\frac{\delta}{n}|U||V|$.
\end{lemma}
\begin{proof} 
First note that, for each disjoint $U,V\subset [n]$ with $|U|\geq |V|\geq \mu n$, we have, as $\lambda\mu n\geq \sqrt{n} \ge 10 d$ that
\[
\left(1-\frac{\lambda}{5}\right)|U||V|\leq |U||V|- d|U|-d|V|\leq e_{K_n\setminus K}(U,V)\leq |U||V|.
\]
Therefore, we have
\[
\left(1-\frac{\lambda}{5}\right)p|U||V|\leq \E (e_{F\setminus K}(U,V))\leq p|U||V|.
\]
Note that $p = \frac{\delta}{n-d-1} = \left(1\pm\frac{\lambda}{5}\right)\frac{\delta}{n}$.
As $\lambda^2 p|U||V|\geq \lambda^2\cdot \left(1-\frac{\lambda}{5}\right)\frac{\delta}{n}\cdot |U|\cdot \mu n \ge  \lambda^2\mu\delta |U|/2 \ge \frac{C}{2}|U|\log{n}$, we have by an application of a Chernoff bound (Lemma~\ref{chernoff}) that, with probability $1-\exp(-\frac{C}{100}|U|\log{n})$,
\[
\left(1-\frac{\lambda}{2}\right)p|U||V|\leq e_{F\setminus K}(U,V)\leq \left(1+\frac{\lambda}{2}\right)p|U||V|.
\]
Thus, $e_{F\setminus K}(U,V)\neq (1\pm \lambda)\frac{\delta}{n}|U||V|$ for some $U,V\subset [n]$ with $|U|\geq |V|\geq \mu n$ with probability at most
\[
\sum_{m_1=\mu n}^n\sum_{m_2=m_1}^n\binom{n}{m_1}\binom{n}{m_2}\exp(-\tfrac{C}{100}m_2\log{n})=o(n^{-14}).\qedhere
\]
\end{proof}


\subsection{Proof of Theorem~\ref{thm:randomproperties:lowerbounds:newnew}}\label{sec:thmproof:lowerbounds}
Finally, we can combine our work in this section to prove Theorem~\ref{thm:randomproperties:lowerbounds:newnew}.
\begin{proof}[Proof of Theorem~\ref{thm:randomproperties:lowerbounds:newnew}]
 Take the variables as set up in Theorem~\ref{thm:randomproperties:lowerbounds:newnew}, so that $1/n\ll 1/C\ll 1/C_0 \ll \eps$, $1/C\ll \mu,\sigma\ll 1$, $d\geq C\log n$, $\eps dn/2\leq m\leq n^{1+\sigma}$, $\delta=2m/n$, and $p=m(\binom{n}{2}-dn/2)^{-1}$. Furthermore, we have
\begin{equation}\label{eq:lambdadefn-1}
\lambda = \max\left\{\frac{\mu^2}{\log^2 n},\left(\frac{n\log n}{m}\right)^{1/4}\right\},
\end{equation}
and $\ell_0$ is the smallest integer for which
$\delta^{(\ell_0-1)}d^{(\ell_0-1)}\geq n$. Let $C'$ satisfy $1/C \ll 1/C' \ll 1$.

Let $K\sim G_d(n)$, form $E$ by including each element of $[n]^{(2)}\setminus E(K)$ uniformly at random with probability $p$, and let $F=K+E$.
 By Lemma~\ref{lem:lbpathcountF}, in order to prove Theorem~\ref{thm:randomproperties:lowerbounds:newnew}, it is sufficient to show that \ref{eq:degreesinFK:pseud}--\ref{prop:F-Kconnection} hold with probability $1-o(n^{-14})$.

Note that
\[
C'\sqrt{\frac{\log n}{\delta}}\overset{\eqref{eq:lambdadefn-1}}{\leq} C'\sqrt{\frac{\log n}{\delta}}\cdot \lambda \left(\frac{m}{n\log n}\right)^{1/4}
= C'\sqrt{\frac{\log n}{\delta}}\cdot \lambda \left(\frac{\delta}{2 \log n}\right)^{1/4}
=C'\lambda \left(\frac{\log n}{2 \delta}\right)^{1/4}\leq \lambda,
\]
where we have used that $\delta\geq \eps d\geq \eps C\log n$. Thus, by Lemma~\ref{lem:nbrhoodsum}, we have that \ref{eq:degreesinFK:pseud} holds with probability $1-o(n^{-14})$.

Furthermore, we have
\[
\frac{\lambda}{\log n}\overset{\eqref{eq:lambdadefn-1}}{\geq} \frac{1}{\log n}\left(\frac{n\log n}{m}\right)^{1/4}\geq \frac{1}{\log n}\left(\frac{\log n}{\delta}\right)^{1/4}=\frac{1}{\delta}\cdot \left(\frac{\delta}{\log n}\right)^{3/4}\geq \frac{C'}{\delta},
\]
so that, by Lemma~\ref{lem:nbrhoodsum}, we have that \ref{eq:nbrhoodsinFK:pseud} holds with probability $1-o(n^{-14})$. 

By Corollary~\ref{cor:fewbadgraphsexpander} with $\lambda'=\lambda/\log n$ and $C'$, as $\lambda'\delta d\geq \lambda\delta\geq C' \log n$, we have that, with probability $1-o(n^{-14})$,  \ref{eq:fewdoubled:K} holds whenever $|U|\leq \lambda n/10^6d\log n$.
Similarly, Corollary~\ref{cor:fewbadgraphsexpander} applied with $\lambda$ and $C'$ shows that, with probability $1-o(n^{-14})$, \ref{eq:fewdoubled:K} holds whenever $|U|\leq \lambda n/10^6d$. Thus, \ref{eq:fewdoubled:K} holds with probability $1-o(n^{-14})$.
Furthermore, by Lemmas~\ref{lem:gnpexpander} and~\ref{lem:gnpexpander2} applied with $C'$, we have that, with probability $1-o(n^{-14})$,  \ref{eq:fewdoubled:F-K} holds with probability $1-o(n^{-14})$, while, by Lemma~\ref{lem:newforpropnew}, \ref{prop:new} holds with probability $1-o(n^{-14})$.

We have, as $\delta\geq C\eps \log n$, that
\[
\lambda^3\delta\overset{\eqref{eq:lambdadefn-1}}{\geq} \left(\frac{n\log n}{m}\right)^{3/4}\delta\geq \left(\frac{\log n}{\delta}\right)^{3/4}\delta
=\log n\cdot \left(\frac{\delta}{\log n}\right)^{1/4}\geq 10^8C'\log n.
\]
Thus, by Lemma~\ref{lem:F-Kconnection} with $\mu=\lambda/10^8$ and $C'$, \ref{prop:F-Kconnection} holds with probability $1-o(n^{-14})$.
This completes the proof that \ref{eq:degreesinFK:pseud}--\ref{prop:F-Kconnection}  hold with probability $1-o(n^{-14})$, completing the proof of the theorem.
\end{proof}


\section{Upper bounds for switching paths}\label{sec:upper}

In this section, we are now able to show that our upper bound on the number of $(F\setminus K,K)$-alternating paths between two fixed vertices, that is, \ref{prop:upperboundFKKalternating}, is likely to hold in the regime $m\leq {1+\sigma}$, where $\sigma$ is a small fixed constant.

\begin{lemma}\label{lem:randomproperties:upperbounds:newnew}
Let $1/n\ll 1/C\ll \mu \ll 1/C_0\ll \eps\leq 1$ and $1/n\ll \sigma \ll 1$. Let $d\in [n]$ and $m\in \N$ satisfy $d\geq C\log n$ and  $\eps dn/2 \leq m\leq n^{1+\sigma}$. Let $\delta=2m/n$ and 
let
\begin{equation*}
\eta = C_0\max\left\{\frac{\mu}{\log n},\left(\frac{n\log n}{m}\right)^{1/8}\right\}.
\end{equation*}
Let $p=m(\binom{n}{2}-dn/2)^{-1}$. Let $K\sim G_d(n)$, form $E$ by including each element of $[n]^{(2)}\setminus E(K)$ uniformly at random with probability $p$, and let $F=K+E$. Then, with probability $1-o(n^{-10})$, we have the following.
\stepcounter{propcounter}
\begin{enumerate}[label = {\emph{\textbf{\Alph{propcounter}}}}]\addtocounter{enumi}{1}
\item  For each distinct $x,y\in [n]$ and each $\log n\leq \ell\leq 10\log n$, the number of $(F\setminus K,K)$-alternating $x,y$-paths of length $2\ell$ is at most $\left(1+\eta\right)d^{\ell}\delta^{\ell}/n$.
\labelinthm{prop:upperboundKFKalternating:newnewnew}
\end{enumerate}
\end{lemma}
\begin{proof}
We start by showing that, for each distinct $x,y\in [n]$ and each  $\log n \leq \ell\leq 10\log n$, with some very high probability, a bound on the number of $(F\setminus K,K)$-alternating $x,y$-paths holds if we restrict these paths to those which do not have their second vertex in some small set of vertices, as follows.
\begin{claim}\label{clm:exceptforZ} For each distinct $x,y\in [n]$ and each $\tfrac{1}{2}\log n \leq \ell\leq 10\log n$, with probability $1-o(n^{-14})$, the following holds.
\stepcounter{propcounter}
\begin{enumerate}[label = {{\textbf{\Alph{propcounter}\arabic{enumi}}}}]
\item There is some set $Z\subset [n]$ with $|Z|\leq \eta \delta/5$ for which the number of $(F\setminus K,K)$-alternating $x,y$-paths of length $2\ell$ in which the second vertex is not in $Z$ is at most $\left(1+\frac{\eta}{5}\right)d^{\ell}\delta^{\ell}/n$.\label{prop:exceptforZ}
\end{enumerate}
\end{claim}
\claimproofstart[Proof of Claim~\ref{clm:exceptforZ}]
Let $\lambda= (\eta/C_0)^2$. Fix $x,y\in [n]$. Reveal $K\sim G_d(n)$. Reveal the edges of $E$ which do not contain $x$.  By Theorem~\ref{thm:randomproperties:lowerbounds:newnew} (noting that $\tfrac{1}{2}\log{n} > 10\ell_0$ for $\ell_0$ as defined in Theorem~\ref{thm:randomproperties:lowerbounds:newnew}) and Lemma~\ref{lem:ubpathcountF:untethered}, we have that, with probability $1-o(n^{-14})$, the following hold.
\begin{enumerate}[label = {{\textbf{\Alph{propcounter}\arabic{enumi}}}}]\addtocounter{enumi}{1}
\item For each $z\in N_K(y)\setminus\{x\}$ and $w\in [n]\setminus \{x,y,z\}$, the number of $(F\setminus K,K)$-alternating $z,w$-paths of length $2\ell-2$ which do not contain $\{x,y\}$ is at least $\left(1-C_0\lambda - \frac{|N_K(w) \cap \{x,y,z\}|}{d}\right)d^{\ell-1}\delta^{\ell-1}/n$.\label{prop:fornewclaim1}
\item For each $z\in N_K(y)\setminus\{x\}$, the number of $(F\setminus K,K)$-alternating paths of length $2\ell-2$ starting at $z$ is at most $\left(1+C_0\lambda\right)d^{\ell-1}\delta^{\ell-1}$.\label{prop:fornewclaim2}
\end{enumerate}
Let $W=[n]\setminus \big(\{x,y\}\cup N_K(x,y) \cup N_K(N_K(y))\big)$, so that $N_K(y)\cap\big(\{w\}\cup N_K(w) \big)=\emptyset$ for each $w\in W$ and $|W|\geq n-2-2d-d^2\geq n-2d^2$.
By applying \ref{prop:fornewclaim1} for each $z\in N_K(y)\setminus \{x\}$, we have the following.

\begin{enumerate}[label = {{\textbf{\Alph{propcounter}\arabic{enumi}}}}]\addtocounter{enumi}{3}
\item For each $w\in W$, the number of $(K,F\setminus K)$-alternating $y,w$-paths of length $2\ell-1$ which do not contain $x$ is at least $\left(d-|N_K(y) \cap \{x\}|\right)\cdot \left(1-C_0\lambda\right)d^{\ell-1}\delta^{\ell-1}/n$.\label{prop:forwinW}
\end{enumerate}

Let $B$ be the set of vertices $w\in [n]\setminus \{x,y\}$ for which the number of $(K,F\setminus K)$-alternating $y,w$-paths of length $2\ell-1$ not containing $x$ is more than $(1+\eta/10)d^{\ell}\delta^{\ell-1}/n$.
Then, the number of $(K,F\setminus K)$-alternating paths of length $2\ell-1$ starting at $y$ which do not contain $x$ is at least
\begin{multline*}\label{eq:lb_B}
|B|\cdot (1+\eta/10)\frac{d^{\ell}\delta^{\ell-1}}{n} ~ + ~ |W\setminus B|\cdot \left(1-C_0\lambda\right)\left(d-|N_K(y) \cap \{x\}|\right)\frac{d^{\ell-1}\delta^{\ell-1}}{n} 
\\ \ge 
\Big(|B|\cdot \eta/10 +
(n - 2d^2)\left(1-C_0\lambda\right)\Big)\left(d-|N_K(y) \cap \{x\}|\right)\frac{d^{\ell-1}\delta^{\ell-1}}{n}. 
\end{multline*}
From \ref{prop:fornewclaim2}, this number is at most $\left(d-|N_K(y) \cap \{x\}|\right) \cdot \left(1+C_0\lambda\right)d^{\ell-1}\delta^{\ell-1}$, so we have
\[
|B|\leq \frac{\left(1+C_0\lambda\right)-(n-2d^2)\cdot \left(1-C_0\lambda\right)/n}{\eta /10n}\leq \frac{\left(2C_0\lambda + \frac{2d^2}{n}\right) 10n}{\eta}\leq               \frac{\eta n}{20},
\]
where we have used that $d \le n^{\sigma} \le \eta \sqrt{n} /100$ and $\lambda=(\eta/C_0)^2$.

Now, reveal the edges next to $x$ in $E$, and let $Z=B\cap N_{F\setminus K}(x)$. Note that $\E|Z|\leq p\eta n/20 \le \eta\delta/10$. As $\eta\delta\geq C_0\delta(\log n/\delta)^{1/8}\geq C_0\log n$, by a Chernoff bound (Lemma~\ref{chernoff}) we have that, with probability $1-o(n^{-14})$, $|Z|\leq \eta \delta/5$. Choose $C'$ satisfying $1/c \ll 1/C' \ll 1$ so that 
$\eta  \ge (\delta/\log{n})^{3/8} \sqrt{\frac{\log{n}}{\delta}} \ge (C\eps)^{3/8} \sqrt{\frac{\log{n}}{\delta}} \ge 20C'\sqrt{\frac{\log{n}}{\delta}}$.  Applying Lemma~\ref{lem:nbrhoodsum}, with probability $1-o(n^{-14})$
we have that $d_{F\setminus K}(v)\leq (1+\eta/20)\delta$. Assuming these properties, we have both that $|Z|\leq \eta \delta/5$ and that the number of 
$(F\setminus K,K)$-alternating $x,y$-paths of length $2\ell$ in which the second vertex is not in $Z$ is at most
\[
\left(1+\frac{\eta}{20}\right)\delta\cdot \left(1+\frac{\eta}{10}\right)\frac{d^{\ell}\delta^{\ell-1}}{n} \leq \left(1+\frac{\eta}{5}\right)\frac{d^{\ell}\delta^{\ell}}{n}.
\]
Thus, the claim holds.
\claimproofend

By Claim~\ref{clm:exceptforZ}, with probability $1-o(n^{-11})$, \ref{prop:exceptforZ} holds for each distinct $x,y\in [n]$ and each $\tfrac{1}{2}\log{n}\leq \ell\leq 10\log{n}$. By Lemma~\ref{lem:ubpathcountF:untethered}, with probability $1-o(n^{-14})$, the following holds for each $1 \leq \ell\leq 10\log{n}$.
\begin{enumerate}[label = {{\textbf{\Alph{propcounter}\arabic{enumi}}}}]\addtocounter{enumi}{4}
\item For each $z\in [n]$, the number of $(F\setminus K,K)$-alternating paths of length $2\ell$ starting at $z$ is at most $\left(1+C_0\lambda\right)d^{\ell}\delta^{\ell}$.\label{prop:generallowerbound}
\end{enumerate}
Assuming \ref{prop:generallowerbound}, and that \ref{prop:exceptforZ} holds for each distinct $x,y\in [n]$ and each $\tfrac{1}{2}\log{n}\leq \ell\leq 10\log{n}$, we will now show that \ref{prop:upperboundKFKalternating:newnewnew} holds.

Let $x,y\in [n]$ be distinct, and let $\mathcal{P}_{x,y}$ be the set of $(F\setminus K,K)$-alternating $x,y$-paths of length $2\ell$.
We will now show by induction that, for each $0\leq i\leq \ell/2$, there is a set $\mathcal{Q}_i$ of at most $(\eta \delta d/5)^i$ $(F\setminus K,K)$-alternating paths of length $2i$ which start at $x$ such that the number of paths in $\mathcal{P}_{x,y}$ not containing a subpath in $\mathcal{Q}_i$ is at most $\left(1+\frac{\eta}{5}\right)\cdot \frac{d^{\ell}\delta^{\ell}}{n}\cdot \sum_{j=0}^{i-1}\left(\frac{\eta}{5}\right)^j$.

This is trivially true for $i=0$ by letting $\mathcal{Q}_0$ contain only the degenerate path of length 0 with vertex set $\{x\}$. Let, then, $i>0$, and suppose there is such a set $\mathcal{Q}_{i-1}$. For each $v\in [n]$, using \ref{prop:exceptforZ}, let $Z_{i,v}$ be such that $|Z_{i,v}|\leq \eta \delta/5$ and the number of $(F\setminus K,K)$-alternating $v,y$-paths of length $2\ell-2(i-1)$ in which the second vertex is not in $Z_{i,v}$ is at most $\left(1+\frac{\eta}{5}\right)d^{\ell-(i-1)}\delta^{\ell-(i-1)}/n$.

Let $\mathcal{Q}_i$ be the set of $(F\setminus K,K)$-alternating paths $Q$ of length $2i$ starting at $x$ for which the subpath of $Q$ of length $(2i-2)$ containing $x$ is in $\mathcal{Q}_{i-1}$ and the penultimate vertex of $Q$ is in $Z_{i,v}$ where $v$ is the antepenultimate vertex of $Q$.
Then, the number of paths in $\mathcal{P}_{x,y}$ not containing a subpath in $\mathcal{Q}_i$ is, as $|\mathcal{Q}_{i-1}|\leq (\eta\delta d/5)^{i-1}$, at most
\[
\left(1+\frac{\eta}{5}\right)\cdot \frac{d^{\ell}\delta^{\ell}}{n}\cdot \sum_{j=0}^{i-2}\left(\frac{\eta}{5}\right)^j+|\mathcal{Q}_{i-1}|\cdot \left(1+\frac{\eta}{5}\right)\cdot \frac{d^{\ell-(i-1)}\delta^{\ell-(i-1)}}{n}
\leq 
\left(1+\frac{\eta}{5}\right)\cdot \frac{d^{\ell}\delta^{\ell}}{n}\cdot \sum_{j=0}^{i-1}\left(\frac{\eta}{5}\right)^j,
\]
while $|\mathcal{Q}_i|\leq |\mathcal{Q}_{i-1}|\cdot (\eta \delta/5)\cdot d\leq (\eta \delta d/5)^i$.

Therefore, by induction, there is a set $\mathcal{Q}_{\lfloor \ell/2 \rfloor}$ of
 at most $(\eta \delta d/5)^{\lfloor\ell/2\rfloor}$ $(F\setminus K,K)$-alternating paths of length $2\lfloor\ell/2 \rfloor$ which start at $x$ such that the number of paths in $\mathcal{P}_{x,y}$ not containing a subpath in $\mathcal{Q}_{\lfloor \ell/2 \rfloor}$ is at most
 \begin{equation}\label{eq:boundotherthanQ}
 \left(1+\frac{\eta}{5}\right)\cdot \frac{d^{\ell}\delta^{\ell}}{n}\cdot \sum_{j=0}^{\lfloor\ell/2 \rfloor-1}\left(\frac{\eta}{5}\right)^j\leq  \left(1+\frac{\eta}{5}\right)\cdot \frac{d^{\ell}\delta^{\ell}}{n}\cdot \frac{1}{1-\frac{\eta}{5}}\leq \left(1+\frac{\eta}{2}\right)\cdot \frac{d^{\ell}\delta^{\ell}}{n}.
 \end{equation}
Furthermore, using \ref{prop:generallowerbound},
the number of paths in $\mathcal{P}_{x,y}$ containing a subpath in $\mathcal{Q}_{\lfloor\ell/2 \rfloor}$ is at most
\[
|\mathcal{Q}_{\lfloor\ell/2 \rfloor}|\cdot (1+C_0\lambda)d^{\ell-\lfloor\ell/2 \rfloor}d^{\ell-\lfloor\ell/2 \rfloor}\leq 2\eta^{\lfloor\ell/2 \rfloor}d^{\ell}d^{\ell}\leq \frac{\eta}{2} \cdot \frac{d^{\ell}d^{\ell}}{n},
\]
where we have used that $\ell\geq \log n$ and $\eta\ll 1$. In combination with \eqref{eq:boundotherthanQ}, we have that \ref{prop:upperboundKFKalternating:newnewnew} holds for $x$ and $y$, as required.
\end{proof}

\section{Concentration for switching paths in the non-critical range}\label{sec:noncritical}

In this section, we prove corresponding bounds to those in Theorem~\ref{thm:randomproperties:lowerbounds:newnew} and Lemma~\ref{lem:randomproperties:upperbounds:newnew} when $\delta\geq n^{\sigma}$ for some fixed $\sigma>0$, using shorter alternating paths, as in the following lemma.

\begin{lemma}\label{lem:alternatingpathsfromKV} Let $1/n\ll 1/C\ll 1/k \ll \sigma \leq 1$ and $1/C\ll \eps,\mu,1/C_0\leq 1$. Let $d\geq C\log n$ and $\eps dn/2 \leq m\leq \binom{n}{2}-dn/2$. Let $\delta=2m/n$ and suppose that $\delta\geq n^{\sigma}$. 
Let $p=m(\binom{n}{2}-dn/2)^{-1}$ and $\eta = C_0 \mu /\log{n}$. Let $K\sim G_d(n)$, form $E$ by including each element of $[n]^{(2)}\setminus E(K)$ uniformly at random with probability $p$, and let $F=K+E$.

Then, with probability $1-o(n^{-10})$, for every distinct $x,y\in [n]$ and every $Z\subset [n]$ with $|Z|\leq 4k$, the number of $(F\setminus K,K)$-alternating $x,y$-paths with length $2k$ and no interior vertices in $Z$ is $(1\pm \eta)\delta^kd^k/n$.
\end{lemma}

As this includes the case where $d$ is large, we will first need to prove a result on the distribution of edges in $G_d(n)$ in this regime, which we do in Section~\ref{sec:Gndedgedistribution}. Then, in Section~\ref{sec:concpoly}, we quote a result of Kim and Vu~\cite{kim2000concentration} and prove Lemma~\ref{lem:alternatingpathsfromKV}.


\subsection{Edge distribution in $G_d(n)$ for large $d$.}\label{sec:Gndedgedistribution}
When $d$ is very large, switching methods in $G_d(n)$ are hard to implement and we need some additional preparation on the likely distribution of edges between large vertex subsets of $G_d(n)$. We will prove the result we need using that the random binomial graph $G(n,q)$ for some $q\approx d/n$ is $d$-regular with at least some very small probability. Thus, when properties hold with exceptionally high probability in $G(n,q)$, we can infer the property holds with high probability in $G_d(n)$. This was similarly done by Krivelevich, Sudakov, Vu and Wormald~\cite{krivelevich2001random} using a result of Shamir and Upfal~\cite{shamir1984large}. For completion, we will give a different simple proof of the specific result we need (Claim~\ref{clm:Gnqpossiblyregular}), while proving the following lemma.

\begin{lemma}\label{lem:probUVgoodinK} Let $1/n\ll 1$. Let $n^{0.99}\leq d < n$.
Then, for all but at most $e^{-n}|\cK_d(n)|$ graphs $K\in \cK_d(n)$, we have the following property.
\stepcounter{propcounter}
\begin{enumerate}[label = {\emph{\textbf{\Alph{propcounter}}}}]
\item For each $U,V\subset [n]$ with $|U|,|V|\geq n^{0.98}$, we have $e_K(U,V)=(1\pm n^{-0.01})\frac{d}{n}|U||V|$.\labelinthm{prop:UVgood}
\end{enumerate}
\end{lemma}
\begin{proof} That the lemma holds when $d\geq n-n^{0.96}$ follows immediately from the following claim.

\begin{claim}\label{clm:dveryhigh}
 If $d\geq n-n^{0.96}$, then \ref{prop:UVgood} holds for every $K\in \cK_d(n)$.
 \end{claim}
 \claimproofstart[Proof of Claim~\ref{clm:dveryhigh}] Let $K\in \cK_d(n)$ and $d\geq n-n^{0.96}$. Let $U,V\subset [n]$ with $|U|\geq |V|\geq n^{0.98}$. Then,
 \[
 e_K(U,V)\geq |U||V|- |U|(n-1-d)\geq |U||V|-|U|\cdot n^{0.96}\geq (1-n^{-0.02})|U||V|,
 \]
 so that $e_K(U,V)=(1\pm n^{-0.02})|U||V|=(1\pm n^{-0.01})\frac{d}{n}|U||V|$.
 \claimproofend

 Suppose, then, that $d< n-n^{0.96}$.
Let $m=\lfloor n^{0.9} \rfloor$ and
\begin{equation}\label{eq:q}
q=\frac{d(n-2m)}{(n-m)(n-m-1)}=(1\pm 10n^{-0.1})\frac{d}{n}.
\end{equation}
Let $G\sim G(n,q)$. We will show that, with at least some small probability, $G$ is $d$-regular, as follows.

\begin{claim}\label{clm:Gnqpossiblyregular}
With probability at least $\exp(-n^{1.91})$, $G$ is $d$-regular.
\end{claim}
\claimproofstart[Proof of Claim~\ref{clm:Gnqpossiblyregular}]
Let $A\subset [n]$ satisfy $|A|=m$ and let $B=[n]\setminus A$. Enumerate $A=\{a_1,\ldots,a_m\}$ and $B=\{b_1,\ldots,b_{n-m}\}$.  For each $i\in [n-m]$, let $d_i=d-d_{G[B]}(b_i)$. Note that $\mathbb{E}(d_{G[B]}(b_i))= q(n-m-1)\geq d/2\geq n^{0.99}/2$ for each $i\in [n-m]$. Thus, by a Chernoff bound (Lemma~\ref{chernoff}) and a union bound, with probability at least $1-1/2n^2$ we have that, for every $i\in [n-m]$,
\begin{align*}
d_{G[B]}(b_i)&=q(n-m-1)\pm \sqrt{qn}\cdot \log^2 n\overset{\eqref{eq:q}}{=}\frac{d(n-2m)}{(n-m)}\pm n^{0.6}=d-\frac{dm}{n-m}\pm n^{0.6}.
\end{align*}
Now,
\[
\frac{dm}{n-m}\geq \frac{n^{0.99}\cdot n^{0.9}/2}{n}\geq n^{0.6},
\]
and
\[
\frac{dm}{n-m}\leq m\cdot \frac{n-n^{0.96}}{n-n^{0.9}}=m-m\cdot \frac{n^{0.96}-n^{0.9}}{n-n^{0.9}}\leq  m-n^{0.9}\cdot n^{-0.04}/2\leq m-n^{0.6}.
\]
Therefore, with probability at least $1-1/2n^2$, we have that $0\leq d_i\leq m$ for every $i\in [n-m]$. Furthermore, by the choice of $q$, note that $\mathbb{P}(e(G-A) = i)$ is maximised  at $i = q\binom{n-m}{2}=d(n-2m)/2$ and thus we have $e(G-A)=q\binom{n-m}{2}=d(n-2m)/2$ with probability at least $1/n^2$. When this happens, $\sum_{i\in [n-m]}d_i=dm$.
Thus, with probability at least $1/2n^2$, $\sum_{i\in [n-m]}d_i=dm$ and $0\leq d_i\leq m$ for each $i\in [m]$.

Assuming these properties hold, we can observe there is a bipartite graph $H$ between vertex classes $A$ and $B$ such that $d_H(a_i)=d$ for each $i\in [m]$ and $d_H(b_i)=d_i$ for each $i\in [n-m]$. Indeed, this can be done by connecting $b_i$ to $a_k$ for each $k\in [m]$ for which there is some $k'=k\mod m$ with
${\sum_{j<i}d_j}<k'\leq {\sum_{j\leq i}d_j}$. Then, for each $i\in [n-m]$, as $d_i\leq m$, we have that $b_i$ is connected to $d_i$ different vertices in $A$, and, as $\sum_{j\in [n-m]}d_j=dm$, for each $i\in [m]$, we have that $a_i$ is connected to $d$ different vertices in $B$.

Futhermore, the probability that $G[A]=\emptyset$ and $G[A,B]=H$ is at least
\[
\min\{q,1-q\}^{m(n-m)+\binom{m}{2}}\overset{\eqref{eq:q}}{\geq} \left(\frac{\min\{d,n-d\}}{2n}\right)^{m(n-m) + \binom{m}{2}}\geq n^{-m(n-m)-\binom{m}{2}} \geq 2n^2\cdot \exp(-n^{1.91}),
\]
and when this happens $G$ is $d$-regular. Thus, we have that $G(n,d/n)$ is $d$-regular with probability at least $\exp(-n^{1.91})$.
\claimproofend
We now show that \ref{prop:UVgood}  is very likely to be true with $G$ in place of $K$, as follows.
\begin{claim}\label{clm:Gnqgood} The probability that \ref{prop:UVgood} with $G$ in place of $K$ is not satisfied is at most $\exp(-n^{1.92})$.
\end{claim}
\claimproofstart[Proof of Claim~\ref{clm:Gnqgood}]
Let $U,V\subset [n]$ satisfy $|U|,|V|\geq n^{0.98}$. Then, by a Chernoff bound (Lemma~\ref{chernoff}), we have that
\[
e_{G}(U,V)=e_G(U\setminus V,V)+e_G(U\cap V,V\setminus U)+2e_G(U\cap V)=(1\pm n^{-0.01}/2)q|U||V|
\]
with probability $1-\exp(-\Omega(n^{-0.02}\cdot n^{1.96}\cdot (d/n)))=1-\exp(-\Omega(n^{1.93}))$ using \eqref{eq:q} and that $d\geq n^{0.99}$. By a union bound, we therefore have that $G$ satisfies \ref{prop:UVgood} in place of $K$ with probability $1-2^{2n}\cdot \exp(-\Omega(n^{1.93}))\geq 1-\exp(-n^{1.92})$, as required.
\claimproofend

As $\P(G=K')$ is constant on $K'\in \cK_d(n)$, if $K\sim G_d(n)$ then the probability that $K$ does not satisfy \ref{prop:UVgood} is, by Claim~\ref{clm:Gnqpossiblyregular} and Claim~\ref{clm:Gnqgood}, at most
\[
\frac{\mathbb{P}(G \text{ does not satisfy \ref{prop:UVgood}})}{\mathbb{P}(G \text{ is $d$-regular})} \le \frac{\exp(-n^{1.92})}{\exp(-n^{1.91})}= \exp(n^{1.91}-n^{1.92})\leq \exp(-n),
\]
and thus the lemma also holds when $d < n - n^{0.96}$.
\end{proof}


\subsection{Concentration for switching paths for large $d$}\label{sec:concpoly}
We will use a result of Kim and Vu~\cite{kim2000concentration} on the concentration of positive polynomials with constant degree, whose variables are independent Bernoulli random variables. By the main theorem of \cite{kim2000concentration} (stated without varying probabilities for simplicity), we have the following.

\begin{theorem}\label{thm:concofpolys} Let $k$ be a fixed constant. Let $N\in\N$ and $p\in [0,1]$. Let $X_1,\ldots,X_N$ be independent random variables with $\P(X_i=1)=p$ and $\P(X_i=0)=1-p$ for each $i\in [N]$. Let $\mathbf{E}$ be a set of subsets of $[N]$ with size $k$ and let $w(e)\geq 0$ for each $e\in \mathbf{E}$. Let $Y=\sum_{e\in \mathbf{E}}w(e)\cdot \prod_{i\in e}X_i$.

For each $A\subset [N]$, let $\mathbf{E}_A=\{e\subset [N]\setminus A:e\cup A\in \mathbf{E}\}$ and let $Y_A={\sum_{e\in \mathbf{E}_A}w(e\cup A)\cdot \prod_{i\in e}X_i}$. For each $i\geq 0$, let $\E_i Y=\max_{A\subset [N]:|A|=i}\E Y_A$. Let $\E^{\max} Y=\max_{i\geq 0}\E_i Y$ and $\E' Y=\max_{i\geq 1}\E_i Y$.

Then, for each $\alpha>1$, we have
\[
\P\left(|Y-\E Y|>8^k(k!)^{1/2}(\E^{\max} Y\cdot \E' Y)^{1/2}\cdot \alpha^k\right)=O(\exp(-\alpha+(k-1)\log n)).
\]
\end{theorem}

We can now prove Lemma~\ref{lem:alternatingpathsfromKV} using a fairly direct application of Theorem~\ref{thm:concofpolys}.

\begin{proof}[Proof of Lemma~\ref{lem:alternatingpathsfromKV}] Note that as $\eps dn/2 \leq m\leq \binom{n}{2}-dn/2$, we have $(1+\eps)d\leq n-1$.
Thus, $d\leq n/(1+\eps)\leq (1-\eps/2)n$ and by Lemma~\ref{lem:probUVgoodinK}, with probability $1-o(n^{-10})$, we have the following property when $d\geq n^{0.99}$.

\stepcounter{propcounter}
\begin{enumerate}[label = {{\textbf{\Alph{propcounter}}}}]
\item For each $U,V\subset [n]$ with $|U|,|V|\geq n^{0.98}$, we have $e_K(U,V)=(1\pm n^{-0.01})\frac{d}{n}|U||V|$.\label{prop:UVgoodforuse}
\end{enumerate}

We now assume that $K$ has this property if $d\geq n^{0.99}$, and otherwise that $K\in \cK_d(n)$, and choose edges from $[n]^{(2)}\setminus E(K)$. For this, let $N=\binom{n}{2}-dn/2$, and enumerate $[n]^{(2)}\setminus E(K)=\{e_1,\ldots,e_N\}$. Let $p=m/N=\delta/(n-1-d)$, so that $p(n-1-d)=\delta$. Let $X_i$, $i\in [N]$, be independent Bernoulli random variables with probability $p$. Let $F'$ be the graph with vertex set $[n]$ and edge set $\{e_i:i\in [N],X_i=1\}$. 

Fix distinct $x,y\in [n]$ and a set $Z\subset [n]$ with $|Z|\leq 4k$. We will show the property we want holds with probability $1-o(n^{-12-4k})$ so that the result follows by a union bound.
Let $\mathbf{E}$ be the set of $I\subset [N]$ with $|I|=k$ such that if $H_I$ is the graph with vertex set $[n]$ and edge set $\{e_i:i\in I\}$, then there is at least one $(H_I,K)$-alternating $x,y$-path of length $2k$ with no interior vertices in $Z$. For each $I\in \mathbf{E}$, let $w(I)$ be the number of $(H_I,K)$-alternating $x,y$-paths of length $2k$ with no interior vertices in $Z$. 
 Let $Y=\sum_{I\in \mathbf{E}}w(I)\cdot \prod_{i\in I}\mathbf{1}_{\{e_i\in E(F')\}}$. Observe that $Y$ is the number of $(F',K)$-alternating $x,y$-paths of length $2k$ with no interior vertices in $Z$.

The result we want will follow from Theorem~\ref{thm:concofpolys} and the following two claims.

\begin{claim}\label{clm:EY} $\left|\E Y- \frac{\delta^kd^k}{n}\right|\leq \frac{\eta \delta^kd^k}{2n}$.
\end{claim}

\begin{claim}\label{clm:EiY} $\mathbb{E}'Y\leq \left(\frac{8k}{\eps}\right)^k\cdot \frac{d^{k}\delta^{k-1}}{n}$.
\end{claim}

From these claims it follows that $\mathbb{E}^{\max}Y = \max{\{\mathbb{E}Y, \mathbb{E}'Y\}} \leq 2\delta^kd^k/n$ and that $\mathbb{E}'Y\leq \left(\frac{8k}{\eps}\right)^k\cdot \delta^{k-1}d^{k}/n$. Let $\alpha=\log^2n$, and note, for an application of Theorem~\ref{thm:concofpolys}, that
\[
8^k(k!)^{1/2}(\E^{\max} Y\cdot \E' Y)^{1/2}\cdot \alpha^k\leq \frac{8^k(k!)^{1/2}(4\sqrt{k}/\sqrt{\eps})^k\log^{2k}n\cdot (\delta d)^{k}}{\delta^{1/2}n} \leq \frac{\eta}{2n}\delta^kd^k,
\]
where we have used that $\delta\geq n^{\sigma}$. Therefore, by Claim~\ref{clm:EY} and Theorem~\ref{thm:concofpolys}, we have that
\[
\P\left(Y\neq (1\pm \eta)\frac{\delta^kd^k}{n}\right)=O(\exp(-\log^2n+(k-1)\log n))=o(n^{-12-4k}),
\]
as required. It is left then only to prove the two claims.


\claimproofstart[Proof of Claim~\ref{clm:EY}] We prove this differently according to whether \textbf{a)} $d<n^{0.99}$ or \textbf{b)} $d\geq n^{0.99}$. Assume first that \textbf{a)} $d< n^{0.99}$ and note that in this case, $n/(n-d-1) \le (1+\eta/10)$. Note too that there are at most $d\cdot (dn)^{k-1}$ paths $v_0v_1\ldots v_{2k}$ such that $v_0=x$, $v_{2k}=y$, and $v_{2i-1}v_{2i}\in E(K)$ for each $0\leq i\leq k-1$. Thus we have,
\begin{align*}
\E Y\leq d\cdot (dn)^{k-1}\cdot p^{k}&= \frac{d^k}{n}\cdot\left(\frac{\delta n}{n-1-d}\right)^k\leq \left(1+\frac{\eta}{2}\right)\frac{\delta^kd^k}{n},
\end{align*}
where we have used that $\eta = C_0 \mu /\log{n}$ and $d<n^{0.99}$.

On the other hand, letting $v_0=x$ and $v_{2k}=y$, as $\eta d=C_0\mu d/\log n\geq C_0\mu C$, there are at least $(1-\eta/4)d$ choices of $v_{2k-1}\in N_K(v_{2k})\setminus (Z\cup \{v_{0}\})$. Then, for each $1\leq i\leq k-1$ in turn, choose $(v_{2i-1},v_{2i})$ so that $v_{2i-1}v_{2i}\in E(K)$ and $v_{2i-1}$ and $v_{2i}$ are distinct from $v_0,v_1,\ldots,v_{2i-2}, v_{2k-1}$ and $v_{2k}$, and not in $Z$, and $v_{2i-2}v_{2i-1},v_{2i}v_{2k-1}\notin E(K)$. For each such $i$, when we go to pick $(v_{2i-1},v_{2i})$, let
\[
W_i=Z\cup \{v_0,v_1,\ldots,v_{2i-2}\}\cup \{v_{2k-1},v_{2k}\}\cup N_K(v_{2i-2},v_{2k-1}),
\]
and note that $|W_i|\leq 6k+1+2d$. When we choose $(v_{2i-1},v_{2i})$, any $(v_{2i-1},v_{2i})$ with $v_{2i-1}v_{2i}\in E(K-W_i)$ has the properties we want, and 
\[
2e(K-W_i)\geq d(n-2|W_i|)\geq \left(1-\frac{\eta}{10k}\right)d(n-d-1), 
\]
where we have used that $d<n^{0.99}$ and $1/n\ll \mu,1/k,1/C_0$ and $\eta=C_0\mu /\log n$. Therefore, we have
\begin{align*}
\E Y\geq \left(1-\frac{\eta}{2k}\right)d\left(\left(1- \frac{\eta}{10k}\right)d(n-d-1)\right)^{k-1} p^k
\geq \left(1-\frac{\eta}{4}\right)\frac{d^k\delta^k}{(n-d-1)}
\geq \left(1-\frac{\eta}{2}\right)\frac{\delta^kd^k}{n}.
\end{align*}
Thus, the claim holds when $d< n^{0.99}$.


Therefore, we can assume that \textbf{b)} $d\geq n^{0.99}$, so that, in particular, \ref{prop:UVgoodforuse} holds. Note, moreover, that, as $d\leq (1-\eps/2)n$, we have $n-d-1\geq \eps n/3$.
Let $v_0=x$ and $v_{2k}=y$. For each $1\leq i\leq 2k-3$, pick $v_i\in [n]\setminus (\{v_{2k},v_0,v_1,\ldots,v_{i-1}\}\cup Z)$ such that $v_{i}\in N_{K_n\setminus K}(v_{i-1})$ if $i$ is odd and $v_{i}\in N_{K}(v_{i-1})$ is even. Note that the number of ways to do this is at least $\left(d - 6k\right)^{k-2}\left(n-d-1 - 6k\right)^{k-1} = (1 \pm \frac{\eta}{10k})^{2k-3}d^{k-2}(n-1-d)^{k-1} = ( 1 \pm \eta/4)d^{k-2}(n-1-d)^{k-1}$, where we have used that $d\geq n^{0.99}$ and $n-1-d\geq \delta\geq n^{\sigma}$. Then, pick distinct $v_{2k-2}$ and $v_{2k-1}$ in $[n]\setminus (\{v_{2k},v_0,v_1,\ldots,v_{2k-3}\}\cup Z)$ with $v_{2k-2}\in N_K(v_{2k-3})$, $v_{2k-1}\in N_K(v_{2k})$ and $v_{2k-2}v_{2k-1}\notin E(K)$, so that, by \ref{prop:UVgoodforuse}, the number of choices for $v_{2k-2}$ and $v_{2k-1}$ is, as $d\geq n^{0.99}$ and $n-1-d\geq \eps n/3$,
\[
(1\pm 2n^{-0.01})d^2-(1\pm n^{-0.01})d^2\cdot\frac{d}{n}=d^2\cdot \left(\frac{n-d-1}{n}\pm 4n^{-0.01}\right)= \left(1\pm \frac{\eta}{10}\right)\frac{d^2(n-d-1)}{n}.
\]
Therefore, the number of choices for $v_1,\ldots,v_{2k-1}$ is $(1\pm \eta/2)d^k(n-d-1)^k/n$. Thus, we have
\[
\E Y=\left(1\pm \frac{\eta}{2}\right)\frac{d^k(n-d-1)^kp^k}{n}=\left(1\pm \frac{\eta}{2}\right)\frac{d^k\delta^k}{n},
\]
as required, completing the proof of the lemma in the case $d\geq n^{0.99}$.
\claimproofend

\smallskip

\noindent\emph{Proof of Claim~\ref{clm:EiY}.}
Let $A\subset [N]$ with $1\leq |A|\leq k$. Let $v_0v_1\ldots v_{2k}$ be an arbitrary $x,y$-path (so that $v_0=x$ and $v_{2k}=y$) with $v_{2j-1}v_{2j}\in E(K)$ and $v_{2j-2}v_{2j-1}\notin E(K)$ for each $j\in [k]$ and $\{e_i:i\in A\}\subset \{v_{2j-2}v_{2j-1}:j\in [k]\}$. We will bound above the number of possibilities for such a path by counting the choices as we determine each of the vertices in this path. First, for each $i\in A$, choose which edge $v_{2j-2}v_{2j-1}$ is the edge $e_i$, and the order of its vertices, so that both $v_{2j-2}$ and $v_{2j-1}$ are determined. Note that there at most $(2k)^{|A|}\leq (2k)^k$ choices for this. 

Let $V_A=\{v_0,v_{2k}\}\cup\bigcup_{i \in A}V(e_i)$. Observe the following holds for each $j\in [k]$. If $\{v_{2j-1},v_{2j}\}\subset V_A$, then $(v_{2j-1},v_{2j})$ is already determined. If $|\{v_{2j-1},v_{2j}\}\cap V_A|=1$, then either $v_{2j-1}$ or $v_{2j}$ is determined, so there are at most $d$ possibilities for $(v_{2j-1},v_{2j})$ as $v_{2j-1}v_{2j}\in E(K)$. If $\{v_{2j-1},v_{2j}\}\cap V_A=\emptyset$, then there are at most $dn$ choices for $(v_{2j-1},v_{2j})$ as $v_{2j-1}v_{2j}\in E(K)$.  For each $s\in \{0,1,2\}$, let $r_s$ be the number of $j\in [k]$ with $|\{v_{2j-1},v_{2j}\}\cap V_A|=s$. Then, the number of choices for the vertices $v_0,v_1,\ldots,v_{2k}$ not in $V_A$ is at most $d^{r_1}(dn)^{r_0}$.
We have  $r_1 + 2r_2+1 = |V_A|$ and $r_0+r_1+r_2 = k$, so  $r_0=k-|V_A|/2 - r_1/2+1/2$. In particular, the number of choices  for the vertices is at most $d^{r_1}(dn)^{r_0} = d^{r_1}(nd)^{k-|V_A|/2 - r_1/2}$.

Suppose first that $|A|<k$.
If $v_{0}\in V(e_j)$ for some $j\in A$, then (as $v_{2k}\notin \cup_{i\in A}V(e_i)$) $|V_A|=2|A|+1$ and we can also observe that $r_1\geq 2$. Indeed, as $|A|<k$, not every edge $v_{2j}v_{2j+1}$, $0\leq j<k$, is in $\{e_i:i\in A\}$. Thus, we can choose $j_0$ to be the least $0\leq j<k$ with $v_{2j}v_{2j+1}\notin\{e_i:i\in A\}$, and $j_1$ to be the largest $0\leq j<k$ with $v_{2j}v_{2j+1}\notin\{e_i:i\in A\}$. Then, $v_{2j_1+1}\notin V_A$ and $v_{2j_1+2}\in V_A$, while, as $v_0v_1\in \{e_i:i\in A\}$, $j_0\geq 1$, so that $v_{2j_0-1}\in V_A$ but $v_{2j_0}\notin V_A$. Thus, as $j_0\leq j_1$, we have $r_1\geq 2$.
Therefore, the number of possibilities for $v_0,v_1\ldots ,v_{2k}$ not in $V_A$ is at most
\begin{equation}\label{eq:poss1}
d^{r_1}(nd)^{k-|V_A|/2 - r_1/2+1/2}
= \left(\frac{d}{n}\right)^{r_1/2}(nd)^{k-|A|} \le \frac{d}{n}\cdot (nd)^{k-|A|}.
\end{equation}

Suppose then that $v_{0}\notin V(e_j)$ for any $j\in A$, so that, therefore, $|V_A|=2|A|+2$. We can also observe that $r_1\geq 1$, as for the smallest $0\leq j<k$ such that $v_{2j}v_{2j+1}\in \{e_i:i\in A\}$ we have $|\{v_{2j-1},v_{2j}\}\cap V_A|=1$.
Then, the number of choices for $v_0,v_1,\ldots ,v_{2k}$ not in $V_A$ is at most
\begin{equation}\label{eq:poss2}
 d^{r_1}(nd)^{k-|V_A|/2 - r_1/2+1/2}
= d^{r_1}(nd)^{k-|A| - 1/2 - r_1/2}
=\left(\frac{d}{n}\right)^{r_1/2}\frac{1}{n^{1/2}}(nd)^{k-|A|}
\le \frac{\sqrt{d}}{n}\cdot (nd)^{k-|A|}.
\end{equation}

Recall that $n-1-d\geq \eps n/3$.
Therefore, including our initial choices, from \eqref{eq:poss1} and \eqref{eq:poss2}, we have that
\begin{align*}
\E Y_A&\leq (2k)^{k}\cdot \frac{d}{n}\cdot (nd)^{k-|A|}\cdot p^{k-|A|}=  (2k)^{k}\cdot \frac{d^{k-|A|+1}}{n}\cdot \left(\frac{\delta n}{n-1-d}\right)^{k-|A|}\\
&\leq \left(\frac{8k}{\eps}\right)^k\cdot \frac{d^{k-|A|+1}\delta^{k-|A|}}{n}\leq \left(\frac{8k}{\eps}\right)^k\cdot \frac{d^{k}\delta^{k-1}}{n},
\end{align*}
where we have used that $|
A|\geq 1$.

Furthermore, if $|A|=k$, then $V_A=\{v_i:0\leq i\leq 2k\}$. In this case the number of choices is $O(1)$, so that $\E Y_A=O(1)=o(\delta^{k-1}d^{k-1}/n)$, where we have used that $\delta\geq n^{\sigma}$ and $1/k\ll \sigma$.
Therefore, we have $\mathbb{E}'Y\leq \left(\frac{8k}{\eps}\right)^k\cdot \frac{d^{k}\delta^{k-1}}{n}$, as required.\hspace{9.70cm}$\boxdot$
\end{proof}

\section{Proof of Theorem~\ref{thm:maintechnicalthm-forprocess}}\label{sec:maintechnicalconclusion}

In Section~\ref{sec:translating}, we prove a simple result that allows us to convert likely properties of $(F,K)$, where $F\sim F(n,d,m)$ is generated by adding random non-edges to $K\sim G_d(n)$, to properties likely to hold for most of the $d$-regular subgraphs $K$ of $F\sim F(n,d,m)$. Then, in Section~\ref{subsec:proofofmaintechnical}, using the results in Sections~\ref{sec:switching}--\ref{sec:noncritical}, we prove Theorem~\ref{thm:maintechnicalthm-forprocess}.


\subsection{Translating likely properties of $G_d(n)$ into subgraphs of $F(n,d,m)$}\label{sec:translating}

Among all graphs $F'$ with $dn/2 + m$ edges, where $F\sim F(n,d,m)$, we have that $\P(F=F')$ is proportional to $|\mathcal{K}_d(F')|$. Thus, any property likely to hold for $G\sim G_d(n)$ is likely to hold for most of the $d$-regular subgraphs of $F\sim F(n,d,m)$. This is important for our argument, so we prove this precisely with the bounds we will use, as follows.

\begin{lemma}\label{lem:translate}
Let $d,n,m\in \N$ satisfy $0\leq  m\leq (n-1-d)n/2$. Let $K\sim G_d(n)$, let $E$ be a uniformly random set of $m$ edges from $[n]^{(2)}\setminus E(K)$, and let $F=K+E$.
Let $\mathcal{P}$ be such that $(F,K)\notin \cP$ with probability $o(n^{-8})$.
Then, with probability $1-o(n^{-2})$, $|\{K'\in \mathcal{K}_d(F):(F,K')\notin \cP\}|\leq  n^{-6}\cdot|\cK_d(F)|$.
\end{lemma}
\begin{proof} Let $\mathcal{F}$ be the set of graphs $F'$ with $m+dn/2$ edges, vertex set $[n]$, and $\cK_d(F')\neq\emptyset$, and note that $\P(F\in \mathcal{F})=1$. Let $\mathcal{F}'$ be the set of $F'\in \mathcal{F}$ for which
$|\{K'\in \mathcal{K}_d(F'):(F',K')\notin \cP\}|> n^{-6}|\cK_d(F')|$.
We wish to prove that $\P(F\in \mathcal{F}')=o(n^{-2})$. For this, note that
\begin{align*}
\P((F,K)\notin \cP)&=\sum_{F'\in \mathcal{F}}\sum_{K'\in \mathcal{K}_d(F')}\P((F,K)=(F',K'))\cdot \mathbf{1}_{(F',K')\notin\cP}\\
&=\frac{1}{|\cK_d(n)|\binom{\binom{n}{2}-dn/2}{m}}\sum_{F'\in \mathcal{F}}\sum_{K'\in \mathcal{K}_d(F')}\mathbf{1}_{(F',K')\notin\cP}\\
&>\frac{1}{|\cK_d(n)|\binom{\binom{n}{2}-dn/2}{m}}\sum_{F'\in \mathcal{F}'}\frac{1}{n^6}\cdot |\cK_d(F')|
=\frac{1}{n^6}\cdot \P(F\in \cF').
\end{align*}
Thus, as $\P((F,K)\notin \cP)=o(n^{-8})$, we have $\P(F\in \cF')=o(n^{-2})$, as required.
\end{proof}


\subsection{Proof of Theorem~\ref{thm:maintechnicalthm-forprocess}}\label{subsec:proofofmaintechnical}
Finally, we can combine our work so far to prove Theorem~\ref{thm:maintechnicalthm-forprocess}.
\begin{proof}[Proof of Theorem~\ref{thm:maintechnicalthm-forprocess}]
 Take the variables as set up in Theorem~\ref{thm:maintechnicalthm-forprocess}, so that $1/n\ll 1/C\ll\mu \ll 1/C_0 \ll \eps\ll 1$, $d\geq C\log n$, and $\eps dn/2\leq m\leq (n-1-d)n/2$. Let 
\begin{equation}
\eta = C_0\max\left\{\frac{\mu}{\log n},\left(\frac{n\log n}{m}\right)^{1/8}\right\}.
\end{equation}
Pick $k$ and $\sigma$ with $1/n\ll 1/k\ll \sigma\ll 1$. If $d\leq n^{\sigma}$, then let $\ell=\lceil 2\log n\rceil$, and otherwise let $\ell=k$.

Let $p=m(\binom{n}{2}-dn/2)^{-1}$. Let $K\sim G_d(n)$. Let $E$ be a set of $m$ edges chosen uniformly at random from $[n]^{(2)}\setminus E(K)$, and let $F=K+E$. Furthermore, form $E'$ by including each element of $[n]^{(2)}\setminus E(K)$ uniformly at random with probability $p$, and let $F'=K+E'$. Note that $p$ is chosen so that $\P(|E'|=i)$ is maximised at $i=m$, and, thus $\P(|E'|=m)\geq 1/n^2$.
We will show that, that with probability $1-o(n^{-10})$, \ref{prop:degreebound}, \ref{prop:lowerboundKFKalternating} and  \ref{prop:upperboundFKKalternating} hold for $F'$ in place of $F$. From this, it holds that, with probability $1-o(n^{-8})$, \ref{prop:degreebound}, \ref{prop:lowerboundKFKalternating} and \ref{prop:upperboundFKKalternating} (defined in Theorem~\ref{thm:switching:propstoapply}) hold for $F$. Then, by Lemma~\ref{lem:translate}, with probability $1-o(n^{-2})$ we have that \ref{prop:degreebound} holds and, for all but at most $n^{-6}|\mathcal{K}_d(F)|$ graphs $K\in \mathcal{K}_d(F)$ we have that \ref{prop:lowerboundKFKalternating} and  \ref{prop:upperboundFKKalternating} hold. Therefore, by Theorem~\ref{thm:switching:propstoapply}, with probability $1-o(n^{-2})$, \eqref{eq:dregcontainingeorfinthm} holds for each $e,f\notin E(F)$.

Thus, to complete the proof of  Theorem~\ref{thm:maintechnicalthm-forprocess} it suffices to proof that, with probability $1-o(n^{-10})$ \ref{prop:degreebound}, \ref{prop:lowerboundKFKalternating}, and  \ref{prop:upperboundFKKalternating} hold for $F'$ in place of $F$.
That \ref{prop:degreebound} holds with probability $1-o(n^{-10})$ follows from Lemma~\ref{lem:nbrhoodsum}. If $d\geq n^{\sigma}$, then \ref{prop:lowerboundKFKalternating} and  \ref{prop:upperboundFKKalternating} hold with probability $1-o(n^{-10})$ by Lemma~\ref{lem:alternatingpathsfromKV}. If $d< n^{\sigma}$, then \ref{prop:lowerboundKFKalternating} and  \ref{prop:upperboundFKKalternating} hold with probability $1-o(n^{-10})$ by Theorem~\ref{thm:randomproperties:lowerbounds:newnew} and Lemma~\ref{lem:randomproperties:upperbounds:newnew}, respectively.
\end{proof}



\section{The lower part of the sandwich}\label{sec:lowerpart}


In this section, we give the lower part $(G_*,G)$ of the sandwich and prove it satisfies Theorem~\ref{thm:sandwich}. As noted, this part of the theorem previously follows from the combination of work by Kim and Vu~\cite{kim2004sandwiching}, Dudek, Frieze, Ruci{\'n}ski and {\v{S}}ileikis~\cite{dudek2017embedding} and Gao, Isaev and McKay~\cite{gao2020sandwichingFIRSTBOUND,gao2022sandwichingFIRSTBOUND}. Moreover, in~\cite{gao2020kim}, Gao, Isaev and McKay use an edge addition process that is effectively the same as we use here, and analyse it as long as $d=\omega(\log n)$. We include the lower part of the sandwich for completion, and because, due to the use of our corresponding version of Lemma~\ref{lem:translate}, the analysis is different. For the coupling $(G_*,G)$, the process used to create $G$ is the complement of the process used for the upper part, while for ease of analysis we generate $G_*$ via a sequence of graphs $H_0,H_1,\ldots$ chosen to be more like the complement of the graphs $G_0^+,G_1^+,\ldots$ in Section~\ref{subsec:clm:main:c}. This coupling is given in Section~\ref{subsec:additionprocess}. We then given an overview of the analysis of the process required in Section~\ref{sec:lowerpartanalysis}, and outline its proof over the rest of this section.


\subsection{Edge addition process}\label{subsec:additionprocess}
Let $1/n\ll 1/C\ll \eps\ll 1$ and $d\geq C \log n$. Let $\eta$ satisfy $1/C\ll \eta \ll \eps$.
 Let $e_1,e_2,\ldots$ be a sequence of edges drawn independently and uniformly at random from $[n]^{(2)}$. Let $x_1,x_2,\ldots$ be independent random variables with $x_i\sim U([0,1])$ for each $i\geq 1$.
Let $F_0$ be the empty graph with vertex set $[n]$, and let $\ell(0)=0$. For each $1\leq i\leq dn/2$, do the following.
Let $\ell(i)$ be the least $\ell(i)>\ell(i-1)$ for which $e_{\ell(i)}\notin E(F_{i-1})$ and
\begin{equation*}
x_{\ell(i)}\leq \frac{|\{K\in \cK_d(n):F_{i-1}+e_{\ell(i)}\subset K\}|}{\max_{e\notin E(F_{i-1})}|\{K\in \cK_d(n):F_{i-1}+e\subset K\}|},
\end{equation*}
and let $F_i=F_{i-1}+e_{\ell(i)}$.
Then, let $G=F_{dn/2}$.

Let $H_0$ be the empty graph with vertex set $[n]$ and let $k(0)=0$. For each $1\leq i\leq \binom{n}{2}$, let $k(i)$ be the least $k>k(i)$ such that $e_{k}\notin \{e_1,\ldots,e_{k-1}\}$ and, if $x_i\leq 1-\eta$ then let $H_i=H_{i-1}+e_{k(i)}$, and, otherwise, let $H_i=H_{i-1}$. Let $N=\lfloor dn/2-\eta n\rfloor$. Let $M\sim \mathrm{Bin}\left((1-\eps)\frac{d}{n},\binom{n}{2}\right)$. If $e(H_N)\geq M$, then let $G_*$ have vertex set $[n]$ and edge set a uniformly random subset of $M$ edges of $E(H_N)$. If $e(H_N)<M$, then let $G_*$ have vertex set $[n]$ and edge set a uniformly random subset of $M$ edges of $[n]^{(2)}$.


\subsection{Analysis overview}\label{sec:lowerpartanalysis}

Similarly to $G^*$ in Section~\ref{sec:process}, it follows by symmetry that for any two graphs $G',G''$ with $V(G')=V(G'')=[n]$ and $e(G')=e(G'')$ we have $\P(G_*=G')=\P(G_*=G'')$, and thus, due to the choice of $M$, $G_*\sim G\left(n,(1-\eps)\frac{d}{n},\binom{n}{2}\right)$.
Furthermore, when $e(H_N)\geq M$ we have $G_*\subset H_N$.

Now, similarly to $H_i$ in Section~\ref{subsec:clm:main:c}, $e(H_N)\sim \mathrm{Bin}(N,1-\eta)$ so, by a Chernoff bound we have that $e(H_N)\geq (1-2\eta)N \ge (1-3\eta)dn/2$ with high probability. As, by another Chernoff bound, $M\leq (1-\eps+\eps/2)dn/2$ with high probability, we have $e(H_N)\geq M$ with high probability. Thus, to prove $G_*\subset G$ with high probability it is sufficient to prove that $H_N\subset G$ with high probability.

Similarly to Claim~\ref{clm:analysisoverview}~\ref{clm:main:b}, we have the following claim, which we prove in Section~\ref{subsec:lowerboundFdist}.
Recall from Definition~\ref{defn:Fminusndm} that we say $F\sim F^-(n,d,m)$ if $F$ is a random graph with vertex set $[n]$ which has the same distribution as $H-E$ where $H\sim G_d(n)$ and $E$ is a uniformly-chosen set of $m$ edges from $E(H)$.

\begin{claim}\label{clm:analysisoverview:lb} For each $0\leq i\leq dn/2$, $F_i\sim F^-(n,d,dn/2-i)$.
\end{claim}

We will also use Lemma~\ref{lem:maintechnicallem-forprocess}. After some preliminary work in Sections~\ref{subsec:concentrationforlowerpart} and~\ref{sec:translatingforlowerbound}, we prove this lemma in Section~\ref{subsec:LemmaTechlowerboundproof}.
We will now show that it follows from Claim~\ref{clm:analysisoverview:lb} and Lemma~\ref{lem:maintechnicallem-forprocess} that $(G_*,G)$ satisfies the desired properties in Theorem~\ref{thm:sandwich}. We have already noted that  $G_*\sim G\left(n,(1-\eps)\frac{d}{n},\binom{n}{2}\right)$, and it follows from Claim~\ref{clm:analysisoverview:lb} with $i=dn/2$, that $G=F_{dn/2}\sim G_d(n)$. Therefore, from the discussion above, we need only show that $\P(H_N\subset G)=1-o(1)$.

By Claim~\ref{clm:analysisoverview:lb} and Lemma~\ref{lem:maintechnicallem-forprocess} (applied with $\eps=\eta$) with a union bound, with high probability, for each $0\leq i\leq N$ and each $e,f\notin E(F_{i-1})$, we have
\begin{equation*}
|\{K\in \mathcal{K}_d(n):F_{i-1}+e\subset E(K)\}|\leq \left(1+\eta\right)|\{K\in \mathcal{K}_d(n):F_{i-1}+f\subset E(K)\}|
\end{equation*}
and in particular,
\begin{equation}\label{eq:dregcontainingeorfinlem:lowerbound:analysis}
    \frac{|\{K\in \cK_d(n):F_{i-1}+e_{\ell(i)}\subset K\}|}{\max_{e\notin E(F_{i-1})}|\{K\in \cK_d(n):F_{i-1}+e\subset K\}|} \ge \frac{1}{1+\eta} > 1 - \eta.
\end{equation}

Suppose, then, that $e\in E(H_N)$. Thus, there is some $1\leq i\leq N$ with $e_{k(i)}=e$, with $e$ not in $e_{1},\dots,e_{k(i)-1}$, and with $x_{k(i)}\leq 1-\eta$. Now, the edges $e_{\ell(1)},e_{\ell(2)},\ldots,e_{\ell(N)}$ are all distinct, so we must have $\ell(N)\geq k(N)$. Thus, there is some $j\in [N]$ with $\ell(j-1)<k(i)\leq \ell(j)$. As $e$ is not in $e_{1},\dots,e_{k(i)-1}$, $x_{k(i)} \le 1 - \eta$ and \eqref{eq:dregcontainingeorfinlem:lowerbound:analysis} holds, we therefore have $\ell(j)=k(i)$ and $e=e_{\ell(j)}\in E(F_{j})\subset E(F_{dn/2})=E(F)$. Thus, $H_N\subset F$.

It is left then only to prove Claim~\ref{clm:analysisoverview:lb} and Lemma~\ref{lem:maintechnicallem-forprocess}.
We prove Claim~\ref{clm:analysisoverview:lb} in Section~\ref{subsec:lowerboundFdist}. As discussed in Section~\ref{subsec:switching}, to prove Lemma~\ref{lem:maintechnicallem-forprocess} we will use switching techniques, and in Section~\ref{subsec:concentrationforlowerpart} we give the concentration result for switching paths that we will use for this, that is, Lemma~\ref{lem:concforlowerpart}. In Section~\ref{sec:translatingforlowerbound}, we note that a corresponding version of Lemma~\ref{lem:translate} holds here.


\subsection{Distribution of $F_i$: proof of Claim~\ref{clm:analysisoverview:lb}}
\label{subsec:lowerboundFdist}

Claim~\ref{clm:analysisoverview:lb} can be shown by induction similarly to Claim~\ref{clm:analysisoverview}\ref{clm:main:b}, or we can observe our edge addition process for $d$ is the complement of the edge deletion process for $n-1-d$.

Indeed, for each $0\leq i\leq dn/2$, let $F^c_i$ be the complement of $F_i$. Observe that $\ell(i)$ is the least $\ell(i)>\ell(i-1)$ for which $e_{\ell(i)}\in E(F^c_{i-1})$ and
\begin{align*}
x_{\ell(i)}&\leq \frac{|\{K\in \cK_d(n):F_{i-1}+e_{\ell(i)}\subset K\}|}{\max_{e\notin E(F_{i-1})}|\{K\in \cK_d(n):F_{i-1}+e\subset K\}|}
=\frac{|\{K\in \cK_d(n):K^c\subset F^c_{i-1}-e_{\ell(i)}\}|}{\max_{e\notin E(F_{i-1})}|\{K\in \cK_d(n):K^c\subset F^c_{i-1}-e\}|}\\
&=\frac{|\cK_{n-1-d}(F^c_{i-1}-e_{\ell(i)})|}{\max_{e\in E(F^c_{i-1})}|\cK_{n-1-d}(F^c_{i-1}-e)|}
\end{align*}
and $F_i^c=F_{i-1}^c-e_{\ell(i)}$.

Thus, we can form $F^c_i$ using an edge deletion process as in Section~\ref{sec:process}.  Thus (as there we did not use the condition $d\geq C\log n$), we have that $F^c_i\sim F(n,n-d-1,i)$ for each $i\in [dn/2]$, from which it can be easily seen that $F_i\sim F^-(n,d,i)$ for each $i\in [dn/2]$, as required.



\subsection{Concentration for switching paths}\label{subsec:concentrationforlowerpart}
As discussed in Section~\ref{subsec:switching}, in this case at least one of the graphs used for the switching paths in the proof of Lemma~\ref{lem:maintechnicallem-forprocess} is dense, so showing the concentration results that we need for the switching paths is much easier than for the upper part of the sandwich. We do this now, where the main subtlety we need to consider is the case where $d$ is very large (so that the graph $K_n\setminus F$ is very sparse).

\begin{lemma}\label{lem:concforlowerpart}
Let $1/n\ll 1/C\ll \eta,\eps$. Let $d\geq C\log n$ and $\eps dn\leq  m\leq d n/2$. Let $K\sim G_d(n)$, let $E$ be a uniformly random set of $m$ edges from $E(K)$, and let $F=K-E$. Then, with probability $1-o(n^{-8})$, the following holds.

Let $x,y\in [n]$ be distinct and let $Z\subset [n]\setminus \{x,y\}$ satisfy $|Z|\leq 10$. Suppose, moreover, that if $d\geq n-10^3/\eta$, then $N_{K_n\setminus K}(x,Z\cup \{y\})=\emptyset$. Then, the number of $(K_n\setminus K,K\setminus F)$-alternating $x,y$-paths of length $4$ is $(1\pm \eta)(2m)^2(n-1-d)^2/n^3$.
\end{lemma}
\begin{proof} Note that the result is trivial if $d=n-1$ as there are no such paths, so we can assume that $d<n-1$.
By Lemma~\ref{lem:probUVgoodinK}, with probability $1-o(n^{-8})$, the following property holds.

\stepcounter{propcounter}
\begin{enumerate}[label = {{\textbf{\Alph{propcounter}\arabic{enumi}}}}]
\item If $d\geq n^{0.99}$, then, for each $U,V\subset [n]$ with $|U|,|V|\geq n^{0.98}$, we have $e_{K}(U,V)=(1\pm n^{-0.01})\frac{d}{n}|U||V|$.\label{prop:UVedges}
\end{enumerate}

Assume, then that \ref{prop:UVedges} holds. We will proceed differently according to whether \textbf{a)} $d<n^{0.99}$, \textbf{b)} $n-1-d< n^{0.99}$, or \textbf{c)} $d\geq n^{0.99}$ and $n-1-d\geq n^{0.99}$.

Suppose first that \textbf{a)} $d<n^{0.99}$. Let $p=2m/dn$. Let $E'\subset K$ be chosen by including each edge in $K$ independently at random with probability $p$, so that $\P(|E'|=m)\geq 1/n^2$, and let $F'=K-E'$. As $m\geq \eps dn\geq \eps Cn\log n$, by a simple application of a Chernoff bound, and a union bound, we have that, with probability $1-o(n^{-8})$, $d_{K\setminus F'}(v)=\left(1\pm \frac{\eta}{20}\right)\frac{2m}{n}$ for each $v\in [n]$. Thus, the following property holds with probability $1-o(n^{-6})$.

\begin{enumerate}[label = {{\textbf{\Alph{propcounter}\arabic{enumi}}}}]\addtocounter{enumi}{1}
\item For each $v\in [n]$, $d_{K\setminus F}(v)=\left(1\pm \frac{\eta}{20}\right)\frac{2m}{n}$.\label{prop:K-Fdegrees}
\end{enumerate}

Assume, then, that \ref{prop:K-Fdegrees} holds.
Let $x,y\in [n]$ be distinct and suppose $Z\subset [n]\setminus \{x,y\}$ satisfies $|Z|\leq 10$.  Choose $v_3\in N_{K\setminus F}(y)\setminus (Z\cup \{x\})$, with at most $d_{K\setminus F}(y)=\left(1\pm \frac{\eta}{20}\right)\frac{2m}{n}$ and at least $d_{K\setminus F}(y) -|Z\cup \{x\}|=\left(1\pm \frac{\eta}{10}\right)\frac{2m}{n}$
choices by \ref{prop:K-Fdegrees} as $m \geq \eps Cn\log n$. Then, choose $v_1,v_2$ with $v_1v_2\in E(K\setminus F)$, $v_1\in N_{K_n\setminus K}(x)\setminus (Z\cup \{v_3,y\})$ and $v_2\in N_{K_n\setminus K}(v_3)\setminus (Z\cup \{x,y\})$. Note that the number of choices for $v_1$ and $v_2$ is at most $2e(K\setminus F)=2m$ and at least
\begin{equation*}
2e(K\setminus F)-2d|N_K(x)\cup N_K(v_3)\cup Z\cup \{x,v_3,y\}|\geq 2m-2d(2d+13)\geq \left(1-\frac{\eta}{10}\right)\cdot 2m,
\end{equation*}
where we used that $m\geq \eps d n$ and $d<n^{0.99}$. Thus, the number of choices for $v_1,v_2$ is $\left(1\pm \frac{\eta}{10}\right)\cdot 2m$.
Altogether, and using that $n=\left(1\pm \frac{\eta}{10}\right)(n-1-d)$ as $d<n^{0.99}$, the number of $(K_n\setminus K,K\setminus F)$-alternating $x,y$-paths of length 4 is $(1\pm \eta)(2m)^2(n-1-d)^2/n^3$.

Suppose then that \textbf{b)} $n-1-d<n^{0.99}$. Take distinct $x,y\in [n]$ and take $Z\subset [n]\setminus \{x,y\}$ with $|Z|\leq 10$ and, if $d\geq n-10^3/\eta$, then $N_{K_n\setminus K}(x,Z\cup \{y\})=\emptyset$. Choose $v_1\in N_{K_n\setminus K}(x)\setminus (Z\cup\{y\})$, with $\left(1\pm \frac{\eta}{10}\right)(n-1-d)$ choices (where we use that, if $d<n-10^3/\eta$, then $\eta(n-d-1)/10>11\geq |Z\cup \{y\}|$). We will show that, with probability $1-o(n^{-21})$, the number of $(K\setminus F,K_n\setminus K)$-alternating $v_1,y$-paths with length 3 avoiding $Z\cup \{x\}$ is $\left(1\pm \frac{\eta}{5}\right)(2m)^2(n-1-d)/n^3$. Thus, by a union bound, with probability at least $1-o(n^{-20})$, this holds for all choices of $v_1$ and hence the number of $(K_n\setminus K,K\setminus F)$-alternating $x,y$-paths with length 4 avoiding $Z$ is
\[
\left(1\pm \frac{\eta}{10}\right)(n-1-d)\cdot \left(1\pm \frac{\eta}{5}\right)\frac{(2m)^2(n-1-d)}{n^3}=(1\pm {\eta})\frac{(2m)^2(n-1-d)^2}{n^3}.
\]
Thus, by a union bound, with probability $1-o(n^{-8})$, this holds for all distinct $x,y\in [n]$ and $Z\subset [n]\setminus \{x,y\}$ with $|Z|\leq 10$ and, if $d\geq n-10^3/\eta$, then $N_{K_n\setminus K}(x,Z\cup \{y\})=\emptyset$.

Fix, then, such an $x,y,Z$ and $v_1$. Let $\hat{E}$ be the set of pairs $(u,v)$ with $u,v\in [n]\setminus (Z\cup \{x,v_1,y\})$, $uv\in E(K_n\setminus K)$ and $v_1u,vy\in E(K)$. Then, $|\hat{E}|\leq (n-1-d)n$ and, as $n-1-d<n^{0.99}$,
\begin{equation}\label{eq:size of hat E}
|\hat{E}|\geq (n-1-d)n-2(n-1-d)(13+2(n-1-d))\geq \left(1-\frac{\eta}{10}\right)(n-1-d)n,
\end{equation}
so that $|\hat{E}|=\left(1\pm \frac{\eta}{10}\right)(n-1-d)n$. 
Let $\hat{G}$ be an auxiliary multigraph with vertex class $[n]$ and an edge $uv$ with multiplicity $|\{(u,v),(v,u)\} \cap \hat{E}|$, so that $\hat{G}$ has maximum degree $n-d-1$. By Vizing's Theorem for multigraphs, the edge chromatic number of $\hat{G}$ is at most $(n-d -1) + 2$. Therefore, we can partition $\hat{E}$ into $n-d+1$ disjoint sets $\hat{E}_1,\hat{E}_2,\ldots, \hat{E}_{n-d+1}$, where each set  $\hat{E}_i$ is a matching -- that is, a vertex $v$ appears at most once in a pair in $\hat{E}_i$. 

Let $p=2m/dn$. Let $E'\subset E(K)$ be chosen by including each edge in $K$ independently at random with probability $p$, so that $\P(|E'|=m)\geq 1/n^2$, and let $F'=K-E'$. For each $i \in [n-d+1]$, let $X_i = |\hat{E}_i \cap E|$. 
Let $B = \{ i \in [n-1-d]: |\hat{E}_i| < C\log{n}\}|$.

Since each $\hat{E}_i$ is a matching, $X_i$ is a binomial random variable with mean $\mathbb{E}(X_i) = p |\hat{E}_i|$.  For $i \not\in B$ we have $|\hat{E}_i| > C\log{n}$. For such $i$, by a Chernoff bound (Lemma~\ref{chernoff}), with probability  $1-o(n^{-22})$ we have $X_i \ge \left(1 \pm \frac{\eta}{20}\right)p^2|\hat{E}_i|$, using $p > \eps \gg 1/C$ and $ \eta \gg 1/C$. 
Thus by a union bound,  with probability  $1-o(n^{-21})$ we have
\begin{equation}\label{eq:bound E' not in B}
|\hat{E} \cap E'| = \sum_{i = 1}^{n-d+1}X_i =  \left(1 \pm \frac{\eta}{20}\right)p^2|\hat{E}| +  \sum_{i \in B}\left(X_i - \left(1 \pm \frac{\eta}{20}\right)p^2|\hat{E}_i|\right).
\end{equation}

Let us bound the size of $B$. We have that $|\hat{E}| =\sum_{i=1}^{n-d+1}|\hat{E}_i|\le |B|C\log{n} + (n-d+1 - |B|)n $ and so rearranging and using~\eqref{eq:size of hat E} we obtain 
\[|B| \le \frac{2n + \eta/10 (n-1-d)n}{n-C\log{n}} \le 4 + \eta(n-1-d) < n^{0.99}.\]
Thus, using~$p > \eps$ and ~\eqref{eq:size of hat E} again, we have 
\begin{equation}\label{eq:bound bad hat E i}
\sum_{i \in B}|\hat{E}_i|< |B|C\log{n} < n^{0.99}\cdot C\log{n}  < \frac{\eta}{50} p^2|\hat{E}|.\end{equation}
Now, since $0 \le X_i \le |\hat{E}_i|$ we see that 
\[ - \left(1 + \frac{\eta}{20}\right)\sum_{i \in B}|\hat{E}_i|  \le   \sum_{i \in B}\left(X_i - \left(1 \pm \frac{\eta}{20}\right)p^2|\hat{E}_i|\right) \le  \sum_{i \in B}|\hat{E}_i|,\]
and in particular, by~\eqref{eq:bound bad hat E i} we have $\sum_{i \in B}\left(X_i - \left(1 \pm \frac{\eta}{20}\right)p^2|\hat{E}_i|\right) = \pm \frac{\eta}{20}p^2|\hat{E}|$.
Putting this together with \eqref{eq:bound E' not in B} we have with probability  $1-o(n^{-21})$ that $|\hat{E} \cap E'| = \left(1 \pm \frac{\eta}{10}\right)p^2|\hat{E}|$.
Therefore, with probability $1-o(n^{-21})$, the number of $(K\setminus F,K_n\setminus K)$-alternating $v_1,y$-paths with length 3 avoiding $Z\cup \{x\}$ is, as $d=\left(1\pm\frac{\eta}{10}\right)n$,
\[
\left(1 \pm \frac{\eta}{10}\right)p^2|\hat{E}|=\left(1 \pm \frac{\eta}{10}\right)\left(\frac{2m}{dn}\right)^2\cdot \left(1\pm \frac{\eta}{10}\right)(n-1-d)n=\left(1\pm \eta\right)\frac{(2m)^2(n-1-d)}{n^3},
\]
as required.

Suppose then that \textbf{c)} $d\geq n^{0.99}$ and $n-1-d\geq n^{0.99}$. Take distinct $x,y\in [n]$, $Z\subset [n]\setminus \{x,y\}$ with $|Z|\leq 10$, and $v_3\in [n]\setminus (Z\cup \{x,y\})$. We will show that, with probability $1-o(n^{-21})$, the number of $(K_n\setminus K,K\setminus F)$-alternating $x,v_3$-paths with length 3 avoiding $Z\cup \{y\}$ is $\left(1\pm \frac{4\eta}{5}\right)(2m)(n-1-d)^2/n^3$.

We have $|N_{K_n\setminus K}(x)\setminus (Z\cup \{y,v_3\})|=\left(1\pm \frac{\eta}{5}\right)(n-1-d)\geq n^{0.98}$ and $|N_{K_n\setminus K}(v_3)\setminus (Z\cup \{x,y\})|=\left(1\pm \frac{\eta}{5}\right)(n-1-d)\geq n^{0.98}$. Thus, by \ref{prop:UVedges}, the number of pairs $(v_1,v_2)$ with $v_1v_2\in E(K)$, $v_1\in N_{K_n\setminus K}(x)\setminus \{y,v_3\}$ and $v_2\in N_{K_n\setminus K}(v_3)\setminus \{x,y\}$ is
\[
(1\pm n^{-0.01})\left(1\pm \frac{\eta}{5}\right)^2\frac{d}{n}(n-1-d)^2=\left(1\pm \frac{3\eta}{5}\right)\frac{d}{n}(n-d-1)^2.
\]

Let $A$ be the set of such pairs $(v_1,v_2)$. For each $i\in [2]$, let $A_i=\{uv:|\{(u,v),(v,u)\}\cap A|=i\}$.
Let $p=2m/dn$. Let $E'\subset E(K)$ be chosen by including each edge in $K$ independently at random with probability $p$, so that $\P(|E'|=m)\geq 1/n^2$, and let $F'=K-E'$, and let $X=2|A_2\cap E'|+|A_1\cap E'|$. Note that $X$ is the number of $(K_n\setminus K,K\setminus F')$-alternating $x,v_3$-paths with length 3 avoiding $Z\cup \{y\}$, and
\[
\E X=\left(1\pm \frac{3\eta}{5}\right)\frac{pd}{n}(n-d-1)^2=\left(1\pm \frac{3\eta}{5}\right)\frac{(2m)}{n^2}(n-d-1)^2.
\]
Thus, by a Chernoff bound applied to $|A_2\cap E'|$ and to $|A_1\cap E'|$, we have, with probability $1-o(n^{-23})$, that $X=\left(1\pm \frac{4}{5}\eta\right)(2m)(n-1-d)^2/n^2$. Thus, with probability $1-o(n^{21})$, the number of $(K_n\setminus K,K\setminus F)$-alternating $x,v_3$-paths with length 3 avoiding $Z\cup \{y\}$ is $\left(1\pm \frac{4\eta}{5}\right)\frac{(2m)}{n^2}(n-d-1)^2$.

Therefore, using a union bound and that there are $\left(1\pm \frac{\eta}{10}\right)pd = \left(1\pm \frac{\eta}{10}\right)\frac{2m}{n}$ choices for $v_3\in N_{K\setminus F}(y)\setminus (Z\cup\{x\})$, with probability $1-o(n^{-20})$, there are $(1\pm {\eta}){(2m)^2(n-1-d)^2}/{n^3}$
$(K_n\setminus K,K\setminus F)$-alternating $x,y$-paths with length 4 avoiding $Z$.
By a union bound, then, with probability $1-o(n^{-8})$, this holds for all distinct $x,y\in [n]$, and $Z\subset [n]\setminus \{x,y\}$ with $|Z|\leq 10$.
\end{proof}



\subsection{Translating likely properties of $G_d(n)$ into subgraphs of $F^-(n,d,m)$}\label{sec:translatingforlowerbound}
We now note that a corresponding version of Lemma~\ref{lem:translate} for $F^-(n,d,m)$ follows from taking complements in Lemma~\ref{lem:translate}.

\begin{lemma}\label{lem:translateforFminus}
Let $d,n,m\in \N$ satisfy $0\leq  m \leq dn/2$. Let $K\sim G_d(n)$, let $E$ be a uniformly random set of $m$ edges from $E(K)$, and let $F=K- E$.
Let $\mathcal{P}$ be a set of pairs $(F,K)$ such that $(F,K)\notin \cP$ with probability $o(n^{-8})$.
Then, with probability $1-o(n^{-2})$,
\begin{equation}\label{eq:fortranslating}
|\{K'\in \mathcal{K}_d(n):(F,K')\notin \cP\land (F\subset K')\}|\leq  \frac{1}{n^6}\cdot|\{K'\in \cK_d(n):F\subset K'\}|.
\end{equation}
\end{lemma}
\begin{proof} Let $\bar{K}=K_n\setminus K$, so that $\bar{K}\sim G_{n-1-d}(n)$, and let $\bar{d}=n-1-d$. Note, moreover, that the distribution of $E$ is a uniformly random set of $m$ edges from $[n]^{(2)}\setminus E(\bar{K})$, and $0\leq m\leq (n-1-\bar{d})n/2$.
Let $\bar{F}=\bar{K}+E$, and note that, therefore, $\bar{F}\sim F(n,n-1-d,m)$.

Let $\mathcal{P}'$ be the set of pairs $(F',K')$ such that $F'$ has $n\bar{d}/2+m$ edges, $V(F')=[n]$, $K'$ is $\bar{d}$-regular, $K'\subset F'$ and $(K_n\setminus F',K_n\setminus K')\in \cP$. Then, we have that $(\bar{F},\bar{K})\in \mathcal{P}'$ with probability $1-o(n^{-6})$. Therefore, by Lemma~\ref{lem:translate}, we have that, with probability $1-o(n^{-2})$,
\[
|\{K'\in \mathcal{K}_{\bar{d}}(\bar{F}):(\bar{F},K')\notin \cP'\}|\leq  \frac{1}{n^6}\cdot|\cK_{\bar{d}}(\bar{F})|.
\]
Thus, as $|\{K'\in \mathcal{K}_d(n):(F,K')\notin \cP\land (F\subset K')\}|=|\{K'\in \mathcal{K}_{\bar{d}}(\bar{F}):(\bar{F},K')\notin \cP'\}|$ and $|\cK_{\bar{d}}(\bar{F})|= |\{K'\in \cK_d(n):F\subset K'\}|$, \eqref{eq:fortranslating} holds with probability $1-o(n^{-2})$.
\end{proof}


\subsection{Proof of Lemma~\ref{lem:maintechnicallem-forprocess}}\label{subsec:LemmaTechlowerboundproof}
Finally, then, we use our work so far in this section to prove Lemma~\ref{lem:maintechnicallem-forprocess} via a switching argument.

\begin{proof}[Proof of Lemma~\ref{lem:maintechnicallem-forprocess}] As set-up in Lemma~\ref{lem:maintechnicallem-forprocess}, let $1/n\ll 1/C\ll \eta,\eps\ll 1$, let $d\geq C\log n$ and $\eps dn\leq  m\leq d n/2$, and let $F\sim F^-(n,d,m)$.
Let $\cK'(F)$ be the set of $K\in \cK_d(n)$ for which $F\subset K$ and the following holds.
\stepcounter{propcounter}
\begin{enumerate}[label = {{\textbf{\Alph{propcounter}\arabic{enumi}}}}]
\item For each distinct $x,y\in [n]$ and each $Z\subset [n]\setminus \{x,y\}$ with $|Z|\leq 10$ such that if $d\geq n-10^3/\eta$,
then $N_{K_n\setminus K}(x,Z)=\emptyset$, the following holds. The number of $(K_n\setminus K,K\setminus F)$-alternating $x,y$-paths of length $4$ is $\left(1\pm \frac{\eta}{100}\right)(2m)^2(n-1-d)^2/n^3$.\label{prop:pathcountingforlowerbound}
\end{enumerate}
By Lemma~\ref{lem:concforlowerpart} and Lemma~\ref{lem:translateforFminus}, we have that, with probability $1-o(n^{-2})$,
\begin{equation}\label{eq:cKprimebound}
|\cK'(F)|\geq \left(1-\frac{1}{n^6}\right)|\{K\in \cK_d(n):F\subset K\}|.
\end{equation}

By a simple application of a Chernoff bound and a union bound (and working via a graph $F'$ as in the proof of Lemma~\ref{lem:concforlowerpart}), with probability $1-o(n^{-2})$, $F$ has the following property.
\begin{enumerate}[label = {{\textbf{\Alph{propcounter}\arabic{enumi}}}}]\addtocounter{enumi}{1}
\item For each $v\in [n]$, $d_{K\setminus F}(v)=\left(1\pm \frac{\eta}{100}\right){2m}/{n}$.\label{prop:degreesofKminusF}
\end{enumerate}

Assuming that \ref{prop:degreesofKminusF} and \eqref{eq:cKprimebound} hold, we will now prove that the property in the lemma holds. We start by proving the following claim.

\begin{claim}\label{clm:edgeswellspreadlowerbound}
For each $e\in E(K_n\setminus F)$, we have
\[
|\{K\in \cK_d(n):F+e\subset K\}|\geq \left(1-\frac{\eta}{2}\right)\frac{2m}{2m+n(n-1-d)} |\{K\in \cK_d(n):F\subset K\}|.
\]
\end{claim}
\claimproofstart[Proof of Claim~\ref{clm:edgeswellspreadlowerbound}] Let $e\in E(K_n \setminus F)$, and suppose, for contradiction, that
\begin{equation}\label{eq:forcontradiction}
|\{K\in \cK_d(n):F+e\subset K\}|< \left(1-\frac{\eta}{2}\right)\frac{2m}{2m+n(n-1-d)} |\{K\in \cK_d(n):F\subset K\}|.
\end{equation}
Let $\cK_e^+=\{K\in \cK_d(n):F+e\subset K\}$ and $\cK_e^-=\{K\in \cK_d(n):F\subset K,e\notin E(K)\}$. Let $L$ be the auxiliary bipartite graph with vertex classes $\cK_e^+$ and $\cK_e^-$ where there is an edge between $K\in \cK_e^+$ and $K'\in \cK_e^-$ if $E(K)\triangle E(K')$ is the edge set of a cycle of length $6$.

Let $K\in \cK_e^+$. Then, we have $e\in E(K\setminus F)$. Let $v_1,v_6$ be such that $e=v_1v_6$. As the number of sequences $v_2,v_3,v_4,v_5$ with $v_2v_3\in E(K)$, $v_4\in N_{K_n\setminus K}(v_3)$, $v_5\in N_{K_n\setminus K}(v_6)$ is at most $(n-1-d)^2dn$, we have
\[
d_L(K)\leq (n-1-d)^2dn\leq \frac{m(n-1-d)^2}{\eps}.
\]
Now, suppose $K\in K_e^+$ is such that \ref{prop:pathcountingforlowerbound} holds. Then, by picking a neighbour $v_2$ of $v_1$ in $K_n\setminus K$ with at most $(n-1-d)$ options, and then applying \ref{prop:pathcountingforlowerbound}, we have
\[
d_L(K)\leq (n-1-d)\cdot \left(1+ \frac{\eta}{10}\right)\frac{(2m)^2(n-1-d)^2}{n^3}.
\]
Note that by~\eqref{eq:forcontradiction} we have $|K_e^-|  = |\{K\in \cK_d:F\subset K\}| - |K_e^+|  > \frac{|\{K\in \cK_d:F\subset K\}|}{2m}$. Thus we have
\begin{align}
e(L)&\leq |\cK_e^+|\cdot \left(1+ \frac{\eta}{10}\right)\frac{(2m)^2(n-1-d)^3}{n^3}+\frac{|\{K\in \cK_d:F\subset K\}|}{n^6}\cdot \frac{m(n-1-d)^2}{\eps}\nonumber\\
&\leq |\cK_e^+|\cdot \left(1+ \frac{\eta}{10}\right)\frac{(2m)^2(n-1-d)^3}{n^3}+\frac{2|\cK_e^-|}{\eps (2m)\cdot 2n^2}\cdot \frac{(2m)^3(n-1-d)^2}{n^4}\nonumber\\
&\leq |\cK_e^+|\cdot \left(1+ \frac{\eta}{10}\right)\frac{(2m)^2(n-1-d)^3}{n^3}+|\cK_e^-|\cdot \frac{\eta}{10}\cdot \frac{(2m)^3(n-1-d)^2}{n^4}.\label{eq:eLlowerboundforlowerpart}
\end{align}

Furthermore, by \eqref{eq:cKprimebound} and \eqref{eq:forcontradiction}, for at least  $(1-\frac{\eta}{50})|K_e^-|$
graphs $K\in \cK_e^-$ we have that \ref{prop:pathcountingforlowerbound} holds. Let $K\in \cK_e^-$ be such a $K$ and let $v_1,v_6$ be such that $e=v_1v_6$.
Choose a neighbour $v_2$ of $v_1$ in $K\setminus F$ such that if $d\geq n-10^3/\eta$ then $v_2\notin N_{K_n\setminus K}(\{v_6\})$. If $d\geq n-10^3/\eta$, then the number of possibilities is, by \ref{prop:degreesofKminusF} and as $m/n\geq \eps d\geq C\eps \log n$
at least
\[
\left(1-\frac{\eta}{100}\right)\frac{2m}{n}-|N_{K_n\setminus K}(\{v_1,v_6\})|\geq \left(1-\frac{\eta}{100}\right)\frac{2m}{n}-\frac{2\cdot 10^3}{\eta}\geq \left(1-\frac{\eta}{50}\right)\frac{2m}{n},
\]
and we have that $N_{K_n\setminus K}(v_2,\{v_1,v_6\})=\emptyset$. When $d<n-10^3\eta$, the number of choices for $v_2$ is, by \ref{prop:degreesofKminusF} at least  $\left(1-\frac{\eta}{100}\right)2m/n$, so in either case there are at least $\left(1-\frac{\eta}{50}\right)2m/n$ choices.
Then, choose a $(K_n\setminus K,K\setminus F)$-alternating $v_2,v_6$-path with length $4$ avoiding $v_1$ so that, as \ref{prop:pathcountingforlowerbound} holds for $K$, there are at least $\left(1-\frac{\eta}{100}\right)(2m)^2(n-1-d)^2/n^3$ possibilities. Thus,
\[
d_L(K)\geq \left(1-\frac{\eta}{50}\right)\frac{2m}{n}\cdot \left(1-\frac{\eta}{100}\right)\frac{(2m)^2(n-1-d)^2}{n^3}\geq \left(1-\frac{\eta}{20}\right)\frac{(2m)^3(n-1-d)^2}{n^4}.
\]
As this holds for at least $(1-\frac{\eta}{50})|K_e^-|$ graphs $K\in \cK_e^-$, we have
\[
e(L)\geq
  \left(1-\frac{\eta}{10}\right)|\cK_e^-|\cdot \frac{(2m)^3(n-1-d)^2}{n^4}.
\]

In combination with \eqref{eq:eLlowerboundforlowerpart}, we have
\[
\left(1+ \frac{\eta}{5}\right)\cdot |\cK^+_e|\cdot \frac{(2m)^2(n-1-d)^3}{n^3}\geq \left(1-\frac{\eta}{5}\right)|\cK_e^-|\cdot \frac{(2m)^3(n-1-d)^2}{n^4}.
\]
Therefore,
\[
|\cK^+_e|\geq \left(1-\frac{\eta}{2}\right)\frac{2m}{n(n-1-d)}|\cK_e^-|=\left(1-\frac{\eta}{2}\right)\frac{2m}{n(n-1-d)}(|\{K\in \cK_d(n):F\subset K\}|-|\cK^+_e|),
\]
so that
\[
|\cK^+_e|\cdot\frac{2m+n(n-1-d)}{n(n-1-d)}\geq \left(1-\frac{\eta}{2}\right)\frac{2m}{n(n-1-d)}|\{K\in \cK_d(n):F\subset K\}|,
\]
and thus
\[
|\cK^+_e|\geq \left(1-\frac{\eta}{2}\right)\frac{2m}{2m+n(n-1-d)} |\{K\in \cK_d(n):F\subset K\}|,
\]
as required.
\claimproofend

Now, if $n-1-d\leq \eta m/10n$, then, from Claim~\ref{clm:edgeswellspreadlowerbound}, we have for each $f\in E(K_n\setminus F)$ that
\begin{align*}
(1+\eta)|\{K\in \cK_d(n):F+f\subset K\}|&\geq (1+\eta)\cdot \left(1-\frac{\eta}{2}\right)\frac{2m}{2m+\eta m/10}|\{K\in \cK_d(n):F\subset K\}|\\
&\geq |\{K\in \cK_d(n):F\subset K\}|,
\end{align*}
so that \eqref{eq:dregcontainingeorfinlem:lowerbound} holds for every $e\in E(K_n\setminus F)$. Assume, then, that $n-1-d>\eta m/10n$. As $m\geq \eps dn$, this implies that  $n-1-d>\eta\eps d/10\geq C\eta\eps\log n/10\geq \sqrt{C}\log n$. Note that this implies that $d < \frac{n}{1+\sqrt{C}\log n} < n - 10^3/\eta$.
Next we prove the following claim.

\begin{claim}\label{clm:mostedgepairs:lowerbound}
Let $e,f\in E(K_n\setminus F)$ share no vertices and let  $n-1-d>\eta m/10n$.
Then,
\begin{equation}\label{eq:dregcontainingeorfinthm:ifdisjoint:lb}
|\{K\in \mathcal{K}_d(n):F+e\subset E(K)\}|\leq \left(1+\frac{\eta}{3}\right)|\{K\in \mathcal{K}_d(n):F+f\subset E(K)\}|.
\end{equation}
\end{claim}
\claimproofstart[Proof of Claim~\ref{clm:mostedgepairs:lowerbound}] Fixing such $e,f\in E(K_n\setminus F)$, let $\mathcal{K}^e=\{K\in \cK_d(n):F+e\subset K,f\notin E(K)\}$ and $\mathcal{K}^f=\{K\in \cK_d(n):F+f\subset K, e\notin E(K)\}$.
Note that \eqref{eq:dregcontainingeorfinthm:ifdisjoint:lb} follows if we have that  $|\cK^e|\leq \frac{\eta}{3}|\{K\in \cK_d(n):F+f\subset K\}|$, so we can assume otherwise. Thus, by Claim~\ref{clm:edgeswellspreadlowerbound}, we can assume that
\begin{equation}\label{eq:cKebound}
|\cK^e|>\frac{\eta}{3}|\{K\in \cK_d(n):F+f\subset K\}|\geq \frac{\eta}{3n}|\{K\in \cK_d(n):F\subset K\}|.
\end{equation}

Form the auxiliary bipartite graph $L$ with vertex classes $\cK^e$ and $\cK^f$, where there is an edge between $K\in \cK^e$ and $K'\in \cK^f$ if $E(K)\triangle E(K')$ is the edge set of a 10-cycle in which $e$ and $f$ are opposing edges of the cycle (see Figure~\ref{fig:switch}).

Suppose that $K\in \cK^e$ satisfies \ref{prop:pathcountingforlowerbound}.
Let $x_1,x_2,y_1,y_2$ be such that $e=x_1x_2$ and $f=y_1y_2$. By \ref{prop:pathcountingforlowerbound} applied with $x=x_1$, $y=y_1$ and $Z=\{x_2,y_2\}$, we have that the number of $(K_n\setminus K, K\setminus F)$-alternating $x_1,y_1$-paths $P$ of length $4$ avoiding $\{x_2,y_2\}$ is at least $\left(1-\frac{\eta}{100}\right)\cdot {(2m)^2(n-1-d)^2}/{n^3}$. Pick such a path $P$. By \ref{prop:pathcountingforlowerbound}, the number of $(K_n\setminus K, K\setminus F)$-alternating $x_2,y_2$-paths of length $4$ avoiding $V(P)$ is at  at least $\left(1-\frac{\eta}{100}\right)\cdot {(2m)^2(n-1-d)^2}/{n^3}$
Therefore, we have that
\[
d_{L}(K)\geq \left(\left(1-\frac{\eta}{100}\right)\cdot \frac{(2m)^2(n-1-d)^2}{n^3}\right)^2
\geq \left(1-\frac{\eta}{6}\right)\cdot \frac{16m^4(n-1-d)^4}{n^6}.
\]
As, by \eqref{eq:cKebound}, $|\cK^e|>\frac{\eta}{3n}|\{K\in \cK_d(n):F\subset K\}|$, this holds for at least $\left(1-\frac{\eta}{50}\right)|\cK^e|$ different $K\in \cK^e$ by \eqref{eq:cKprimebound}. Therefore, we have
 \begin{align}
e(L)\geq \left(1-\frac{\eta}{50}\right)|\cK^e|\cdot \left(1-\frac{\eta}{6}\right)\frac{16m^4(n-1-d)^4}{n^6}\geq |\cK^e|\cdot \left(1-\frac{\eta}{5}\right)\frac{16m^4(n-1-d)^4}{n^6}.\label{eq:eLlowerbound}
 \end{align}

On the other hand, for each $K'\in \cK^f$, we have (working similarly to the way we did in Claim~\ref{clm:edgeswellspreadlowerbound}), $d_L(K')\leq (d^2(n-1-d)^2)^2$. Furthermore, for each $K'\in \cK^f$ for which \ref{prop:pathcountingforlowerbound} holds, we have $d_L(K')\leq \left(1+ \frac{3\eta}{100}\right)\cdot \frac{16m^4(n-1-d)^4}{n^6}$. As, by \eqref{eq:cKprimebound} and \eqref{eq:cKebound}, this latter bound holds for all but $|\{K\in \cK_d(n):F\subset K\}|/n^6\leq |\cK^e|/ n^3$ graphs $K'\in \cK^f$, we have
 \begin{align}
 e(L)&\leq |\cK^f|\cdot \left(1+ \frac{3\eta}{100}\right)\cdot \frac{16m^4(n-1-d)^4}{n^6}+\frac{|\cK^e|}{n^3}\cdot (d^2(n-1-d)^2)^2
\nonumber\\
 &=
 |\cK^f|\cdot \left(1+ \frac{3\eta}{100}\right)\cdot \frac{16m^4(n-1-d)^4}{n^6}+ |\cK^e|\cdot \frac{d^4n^3}{16 m^4}\cdot \frac{16m^4(n-1-d)^4}{n^6}\nonumber\\
 & \leq
 |\cK^f|\cdot \left(1+ \frac{3\eta}{100}\right)\cdot \frac{16m^4(n-1-d)^4}{n^6}+ |\cK^e|\cdot \frac{\eta}{100}\cdot \frac{16m^4(n-1-d)^4}{n^6},\nonumber
 \end{align}
 where we have used that $m\geq \eps dn$.

 In combination with \eqref{eq:eLlowerbound}, we have that
\[
|\cK^e|\cdot \left(1-\frac{\eta}{5}-\frac{\eta}{100}\right)\frac{16m^4(n-1-d)^4}{n^6}\leq |\cK^f|\cdot \left(1+ \frac{3\eta}{100}\right)\frac{16m^4(n-1-d)^4}{n^6},
\]
and, hence, $|\cK^e|\leq (1+\eta/3)|\cK^f|$, so that
\begin{align*}
|\{K\in \mathcal{K}_d(n):F+e\subset K\}|&\leq |\cK^e|+|\{K\in \cK_d(n):F+e+f\subset K\}|\\
&\leq \left(1+\frac{\eta}{3}\right)|\cK^f|+|\{K\in \cK_d(n):F+e+f\subset K\}|\\
&\leq \left(1+\frac{\eta}{3}\right)|\{K\in \mathcal{K}_d(n):F+f\subset K\}|,
\end{align*}
as required.
\claimproofend

Let, then $e,f\in [n]^{(2)}\setminus E(F)$. Pick some $e'\notin E(F)$ which is vertex-disjoint from $e$ and $f$. Then, by Claim~\ref{clm:mostedgepairs:lowerbound} applied twice, we have
\begin{align*}
|\{K\in \mathcal{K}_d(n):F+e\subset E(K)\}|&\leq \left(1+\frac{\eta}{3}\right)|\{K\in \mathcal{K}_d(n):F+e'\subset E(K)\}|\\
&\leq \left(1+\frac{\eta}{3}\right)^2|\{K\in \mathcal{K}_d(n):F+f\subset E(K)\}|,
\end{align*}
so that, as $\eta\ll 1$, \eqref{eq:dregcontainingeorfinlem:lowerbound} holds, as required.
\end{proof}

\end{document}